\documentclass[11pt, reqno]{amsart}
\usepackage{txfonts}
\usepackage{amsmath,amssymb,amsthm,mathrsfs,enumerate,bm,xcolor,multirow,pbox}
\usepackage{graphicx,color,framed,tikz,caption,subcaption}
\usepackage{enumitem}
\setlist{leftmargin=9mm}
\usepackage[colorlinks,linkcolor=black,citecolor=black,urlcolor=black]{hyperref}
\allowdisplaybreaks[4]
\numberwithin{equation}{section}
\newcommand{\N}{\mathbb{N}}
\newcommand{\R}{\mathbb{R}}

\newcommand{\pnorm}[2]{\lVert #1\rVert_{#2}}
\newcommand{\bigpnorm}[2]{\big\lVert#1\big\rVert_{#2}}
\newcommand{\biggpnorm}[2]{\bigg\lVert#1\bigg\rVert_{#2}}
\newcommand{\ang}[2]{\langle#1 \rangle_{#2}}
\newcommand{\bigang}[2]{\big\langle#1 \big\rangle_{#2}}
\newcommand{\biggang}[2]{\bigg\langle#1 \bigg\rangle_{#2}}
\newcommand{\abs}[1]{\lvert#1\rvert}
\newcommand{\bigabs}[1]{\big\lvert#1\big\rvert}
\newcommand{\biggabs}[1]{\bigg\lvert#1\bigg\rvert}
\newcommand{\iprod}[2]{\left\langle#1,#2\right\rangle}

\renewcommand{\epsilon}{\varepsilon}

\renewcommand{\d}[1]{\mathrm{d}#1}

\newcommand{\ceil}[1]{\left\lceil #1 \right\rceil}

\newcommand{\bigop}{\mathcal{O}_{\mathbf{P}}}

\newcommand{\smallo}{\mathfrak{o}}
\newcommand{\bigo}{\mathcal{O}}

\newcommand{\equald}{\stackrel{d}{=}}
\renewcommand{\hat}{\widehat}

\DeclareMathOperator{\E}{\mathbb{E}}
\DeclareMathOperator{\Prob}{\mathbb{P}}
\DeclareMathOperator{\Av}{\mathsf{Av}}

\DeclareMathOperator{\tr}{tr}
\DeclareMathOperator{\var}{Var}

\DeclareMathOperator{\op}{op}
\DeclareMathOperator{\lip}{Lip}
\DeclareMathOperator{\plim}{plim}

\DeclareMathOperator{\prox}{\mathsf{prox}}
\DeclareMathOperator{\env}{\mathsf{e}}

\DeclareMathOperator{\lasso}{\mathsf{L}}
\DeclareMathOperator{\delasso}{\mathsf{dL}}
\DeclareMathOperator{\ridge}{\mathsf{R}}
\DeclareMathOperator{\rob}{\mathsf{M}}
\DeclareMathOperator{\lse}{\mathsf{LSE}}
\DeclareMathOperator{\errg}{\textrm{g}}

\DeclareMathOperator*{\argmax}{arg\,max\,}
\DeclareMathOperator*{\argmin}{arg\,min\,}
\DeclareMathOperator*{\esssup}{ess\,sup}

\newcommand{\beq}{\begin{equation}}
\newcommand{\eeq}{\end{equation}}
\newcommand{\beqa}{\begin{equation} \begin{aligned}}
\newcommand{\eeqa}{\end{aligned} \end{equation}}
\newcommand{\beqas}{\begin{equation*} \begin{aligned}}
\newcommand{\eeqas}{\end{aligned} \end{equation*}}

\newcommand{\bit}{\begin{itemize}}
	\newcommand{\eit}{\end{itemize}}
\newcommand{\bmat}{\begin{bmatrix}}
	\newcommand{\emat}{\end{bmatrix}}

\theoremstyle{definition}\newtheorem{problem}{Problem}[section]
\theoremstyle{definition}
\theoremstyle{remark}\newtheorem{assumption}{Assumption}

\theoremstyle{remark}\newtheorem{remark}[problem]{Remark}
\theoremstyle{definition}\newtheorem{example}[problem]{Example}
\theoremstyle{plain}\newtheorem{theorem}[problem]{Theorem}
\theoremstyle{plain}\newtheorem{question}{Question}
\theoremstyle{plain}\newtheorem{lemma}[problem]{Lemma}
\theoremstyle{plain}\newtheorem{proposition}[problem]{Proposition}
\theoremstyle{plain}\newtheorem{corollary}[problem]{Corollary}
\theoremstyle{plain}

\AtBeginDocument{%
	\def\MR#1{}
}

\begin{document}

\title[Universality in high dimensions]{Universality of regularized regression estimators in high dimensions}
\thanks{The research of Q. Han is partially supported by NSF grants DMS-1916221 and DMS-2143468.}

\author[Q. Han]{Qiyang Han}

\address[Q. Han]{
	Department of Statistics, Rutgers University, Piscataway, NJ 08854, USA.
}
\email{qh85@stat.rutgers.edu}

\author[Y. Shen]{Yandi Shen}

\address[Y. Shen]{
	Department of Statistics, University of Chicago, Chicago, IL 60615, USA.
}
\email{ydshen@uchicago.edu}

\date{\today}


\begin{abstract}
The Convex Gaussian Min-Max Theorem (CGMT) has emerged as a prominent theoretical tool for analyzing the precise stochastic behavior of various statistical estimators in the so-called high dimensional proportional regime, where the sample size and the signal dimension are of the same order. However, a well recognized limitation of the existing CGMT machinery rests in its stringent requirement on the exact Gaussianity of the design matrix, therefore rendering the obtained precise high dimensional asymptotics largely a specific Gaussian theory in various important statistical models.

This paper provides a structural universality framework for a broad class of regularized regression estimators that is particularly compatible with the CGMT machinery. Here universality means that if a `structure' is satisfied by the regression estimator $\hat{\mu}_G$ for a standard Gaussian design $G$, then it will also be satisfied by $\hat{\mu}_A$ for a general non-Gaussian design $A$ with independent entries. In particular, we show that with a good enough $\ell_\infty$ bound for the regression estimator $\hat{\mu}_A$, any `structural property' that can be detected via the CGMT for $\hat{\mu}_G$ also holds for $\hat{\mu}_A$ under a general design $A$ with independent entries. 

As a proof of concept, we demonstrate our new universality framework in three key examples of regularized regression estimators: the Ridge, Lasso and regularized robust regression estimators, where new universality properties of risk asymptotics and/or distributions of regression estimators and other related quantities are proved. As a major statistical implication of the Lasso universality results, we validate inference procedures using the degrees-of-freedom adjusted debiased Lasso under general design and error distributions. We also provide a counterexample, showing that universality properties for regularized regression estimators do not extend to general isotropic designs. 

The proof of our universality results relies on new comparison inequalities for the optimum of a broad class of cost functions and Gordon's max-min (or min-max) costs, over arbitrary structure sets subject to $\ell_\infty$ constraints. These results may be of independent interest and broader applicability. 
\end{abstract}

\maketitle

\setcounter{tocdepth}{1}
\tableofcontents


\sloppy

\section{Introduction}
\subsection{Overview}
Consider the standard linear model
\begin{align}\label{eqn:reg_model}
Y = A \mu_0+\xi,
\end{align}
where $\mu_0 \in \R^n$ is the signal of interest, $A \in \R^{m\times n}$ is the design matrix, $\xi \in \R^m$ is the error vector, and $Y \in \R^m$ stands for the response vector. Here and below, we reserve the notation $n$ for signal dimension, and $m$ for sample size. We will be interested in understanding the precise stochastic behavior of a broad class of regularized estimators (of $\mu_0$) taking the following generic form
\begin{align}\label{def:estimator_generic}
\hat{\mu}_A\in \argmin_{\mu \in \R^n} \bigg\{\frac{1}{m}\sum_{i=1}^m \psi_0\big((A\mu)_i-Y_i\big)+ \mathsf{f}(\mu)\bigg\}.
\end{align}
Here $\psi_0: \R \to \R_{\geq 0}$ is a loss function, and $\mathsf{f}:\R^n \to \R_{\geq 0}$ is a structure-promoting regularizer for $\mu_0$.

As a canonical example of the regularized regression estimators in (\ref{def:estimator_generic}), the Lasso estimator $\hat{\mu}_A^{\lasso}$ (cf. \cite{tibshirani1996regression}) can be realized by taking $\psi_0(x)=x^2/2$ and $\mathsf{f}(\mu)=\lambda \pnorm{\mu}{1}/m$ with a tuning parameter $\lambda>0$. A notable recent line of Lasso theory attempts to characterize its exact behavior under certain specific settings. This line (i) postulates an exact distributional assumption on the design matrix, where
\begin{align}\label{intro:A}
\hbox{$A$ is a standard Gaussian design $G$ with i.i.d. $\mathcal{N}(0,1/m)$ entries},
\end{align}
and, (ii) works in the so-called `proportional regime', where
\begin{align}\label{intro:proportional_regime}
\hbox{the sample size $m$ is proportional to the signal dimension $n$}.
\end{align}
In particular, \cite{miolane2021distribution} showed that under (\ref{intro:A})-(\ref{intro:proportional_regime}), among with other conditions, with the (Gaussian) error $\xi$ possessing a noise level $\sigma>0$ and a tuning parameter $\lambda$, there exist some $\sigma_\ast,\lambda_\ast>0$ such that the distribution of the Lasso estimator $\hat{\mu}_G^{\lasso}$ can be identified as $\eta_1\big(\mu_0+\sigma_\ast Z_n;\lambda_\ast\big)$ in the following sense\footnote{Precisely, the formulation (\ref{intro:lasso_dist}) is taken from \cite{celentano2020lasso}; however the proofs in \cite{miolane2021distribution} also lead to (\ref{intro:lasso_dist}) with appropriate modifications.}: for any $1$-Lipschitz function $\mathsf{g}:\R^n \to \R$, it holds with high probability that
\begin{align}\label{intro:lasso_dist}
\mathsf{g}\big(\hat{\mu}_G^{\lasso}/\sqrt{n}\big)\approx \E \mathsf{g}\big( \eta_1\big(\mu_0+\sigma_\ast Z_n;\lambda_\ast\big)/\sqrt{n}\big).
\end{align}
Here $\eta_1$ is the soft-thresholding function (formally defined in (\ref{def:eta_1})), and $Z_n \stackrel{d}{=} \mathcal{N}(0,I_n)$ is a standard Gaussian vector in $\R^n$.

The method of proof for (\ref{intro:lasso_dist}) in \cite{miolane2021distribution} is based on a two-sided version of Gordon's Gaussian min-max theorem (cf. \cite{gordon1988milman}), now known as the Convex Gaussian Min-Max Theorem (CGMT) (cf. \cite{stojnic2013framework,thrampoulidis2018precise}); see Theorem \ref{thm:CGMT} for a formal statement. The CGMT approach is a flexible theoretical framework that reduces a given, complicated `primal min-max optimization problem' involving a standard Gaussian design matrix, to a much simpler `Gordon's min-max optimization problem' involving Gaussian vectors only. For the Lasso estimator, the CGMT machinery executed by \cite{miolane2021distribution,celentano2020lasso} substantially improves a weaker version of (\ref{intro:lasso_dist}) obtained in \cite{bayati2012lasso}\footnote{The proof of a weaker version of (\ref{intro:lasso_dist}) in \cite{bayati2012lasso} is based on the so-called state evolution analysis (cf. \cite{bayati2011dynamics}) of an approximate message passing (AMP) algorithm for Lasso that also relies crucially on the exact Gaussianity of the design matrix as in (\ref{intro:A}).}, by providing precise non-asymptotic descriptions of (\ref{intro:lasso_dist}) and the distributions of other quantities associated with the Lasso. These results are not only theoretically interesting in their own rights, they also provide important foundation for statistical inference using the Lasso estimator in the proportional regime (\ref{intro:proportional_regime}).

The flexible and principled nature of the CGMT method has led to systematic progress in understanding the precise risk/distributional behavior of canonical statistical estimators across a wide array of important statistical models, see e.g. \cite{thrampoulidis2015regularized,thrampoulidis2018precise,deng2019model,hu2019asymptotics,salehi2019impact,celentano2020lasso,montanari2019generalization,loureiro2021learning,han2022noisy,liang2020precise,wang2022does,zhang2022modern} for some samples. The power of the CGMT method is further demonstrated in some of the above cited works that deal with either general correlated Gaussian designs, cf. \cite{montanari2019generalization,celentano2020lasso,loureiro2021learning,liang2020precise}, or the `maximal' problem aspect ratio beyond the proportional regime (\ref{intro:proportional_regime}), cf. \cite{han2022noisy}.

Unfortunately, while being a powerful theoretical tool, the CGMT machinery relies on the Gaussianity of the design in an essential way via the use of Gaussian comparison inequalities, and therefore precise high-dimensional asymptotics results derived from the CGMT remain largely a specific Gaussian theory.

The main goal of this paper is to provide a general universality framework for `structural properties' of regularized regression estimators (\ref{def:estimator_generic}) that is compatible with the CGMT machinery. Here universality means that if a `structure' is satisfied by $\hat{\mu}_G$ for a standard Gaussian design $G$, then it will also be satisfied by $\hat{\mu}_A$ for a general non-Gaussian design $A$ with independent entries. A more concrete example for the prescribed structural universality, is to establish the validity of the distribution (\ref{intro:lasso_dist}) of the Lasso estimator for general non-Gaussian designs.

As already hinted above, a major theoretical advantage of our universality framework, lies in its compatibility with the CGMT method. Roughly speaking, we show that, with a good enough $\ell_\infty$ bound for $\hat{\mu}_A$ in (\ref{def:estimator_generic}), any structural property that can be detected via the CGMT for $\hat{\mu}_G$ also holds for $\hat{\mu}_A$ under a general design $A$ with independent entries. Due to the widespread use of the CGMT approach as mentioned above, we expect our universality framework to be of much broader applicability beyond the examples worked out in the current paper.

\subsection{Structural universality framework}\label{subsection:intro_universality_reg}

In the sequel, we will work with $\hat{w}_A\equiv \hat{\mu}_A-\mu_0$ instead of $\hat{\mu}_A$ for consistent presentation with the main results in Sections \ref{section:theory} and \ref{section:examples}. Clearly
\begin{align}\label{def:w_generic}
\hat{w}_A&\in \argmin_{w \in \R^n} H_{\psi_0,\mathsf{f}}(w,A,\xi) \equiv \argmin_{w \in \R^n} \bigg\{\frac{1}{m}\sum_{i=1}^m \psi_0\big((Aw)_i-\xi_i\big)+\mathsf{f}(\mu_0+w)\bigg\}.
\end{align}
Now we may formulate the universality problem precisely. 

\begin{question}\label{ques:universality}
Take any `structural property' $\mathcal{T}_n\subset \R^n$ such that $\Prob(\hat{w}_G \in \mathcal{T}_n)\approx 1$. Then is it true that $\Prob(\hat{w}_A \in \mathcal{T}_n)\approx 1$, when the design matrix $A$ has independent entries with matching first two moments as those of $G$?
\end{question}

Our main abstract universality framework for the regularized regression estimator $\hat{w}_A$, Theorem \ref{thm:universality_reg}, answers the above question in the affirmative in the proportional regime (\ref{intro:proportional_regime}), provided the following hold:
\begin{enumerate}
	\item[(U0)] The entries of $A$ and $\xi$ have `enough' moments, the loss function $\psi_0$ is `self-similar', and the regularizer $\mathsf{f}$ possesses `enough' continuity. 
	\item[(U1)] With high probability $\pnorm{\hat{w}_A}{\infty}\vee \pnorm{\hat{w}_G}{\infty}\leq L_n$ for some $L_n>0$ that grows mildly, say, $L_n=n^\epsilon$ for small enough $\epsilon>0$. 
	\item[(U2)] $\Prob(\hat{w}_G \in \mathcal{T}_n)\approx 1$ holds at the level of the cost function $H_{\psi_0,\mathsf{f}}$: for some non-random $z>0$ and small $\rho_0>0$, with high probability
	\begin{align*}
	\min_{w \in \mathcal{T}_n^c}H_{\psi_0,\mathsf{f}}(w,G,\xi)\geq z+2\rho_0>z+\rho_0\geq \min_{w \in \R^n}H_{\psi_0,\mathsf{f}}(w,G,\xi).
	\end{align*}
\end{enumerate}
Here (U0) should be viewed as regularity conditions. In particular, Theorem \ref{thm:universality_reg} is established for the square loss case and the  (possibly non-differentiable) robust loss case, but as will be clear below, other loss functions $\psi_0$ whose derivatives are `similar' to itself would also work.  Furthermore, the precise number of moments needed for $A$ and $\xi$ depends on the choice of the loss function $\psi_0$, and the moduli of continuity needed for $\mathsf{f}$ is almost minimal. Consequently, the essential conditions to apply the machinery of Theorem \ref{thm:universality_reg} are (U1) and (U2):
\begin{itemize}
	\item (U1) requires $\ell_\infty$ bounds for the regression estimator under both a standard Gaussian design $G$ and the targeted design $A$. While verification of (U1) can be performed in a case-by-case manner, a particular useful general method for obtaining $\ell_\infty$ bounds for $\hat{w}_A$ is to study perturbations of $\hat{w}_A$ by its column and row leave-one-out versions (cf. \cite{elkaroui2013asymptotic,elkaroui2018impact}). In essence, these leave-one-out perturbations are both close enough to $\hat{w}_A$ while creating sufficient independence to guarantee coordinate-wise controls for $\hat{w}_A$. 
	\item (U2) requires high probability detection of the structural property $\mathcal{T}_n$ for $\hat{w}_G$ via the cost function $H_{\psi_0,\mathsf{f}}$. A particularly appealing feature of (U2) lies in its compatibility with the CGMT approach, as one then only needs to verify for the simpler Gordon's problem a constant order gap ($=\rho_0$) between its cost optimum over $\mathcal{T}_n^c$ and the global cost optimum ($=z$). 
\end{itemize}
In summary, for a given structural universality problem of $\hat{w}_A$, once a good enough $\ell_\infty$ bound is verified, the problem is almost completely reduced to the standard Gaussian design in which the powerful CGMT can be directly applied. 

\subsection{Examples}
As a proof of concept, we apply the aforementioned universality framework to three canonical examples of regularized regression estimators (\ref{def:estimator_generic}) in the linear model (\ref{eqn:reg_model}), namely:
\begin{enumerate}
	\item[(E1)] the Ridge estimator,
	\item[(E2)] the Lasso estimator, and
	\item[(E3)] regularized robust regression estimators. 
\end{enumerate}
In particular, we prove that the $\ell_\infty$ bounds required in (U1) hold for all the above three examples (under appropriate moment conditions on $A$ and $\xi$), and therefore universality holds for any structural properties $\mathcal{T}_n$ of these estimators that can be verified under a standard Gaussian design in the sense of (U2).

For the Lasso estimator, our distributional universality of $\hat{\mu}_A^{\lasso}$ shows that (\ref{intro:lasso_dist}) is valid with $\mathsf{g}(\hat{\mu}_G^{\lasso}/\sqrt{n})$ replaced by $\mathsf{g}(\hat{\mu}_A^{\lasso}/\sqrt{n})$ for any $1$-Lipschitz function $\mathsf{g}:\R^n \to \R$. Using the same formulation as (\ref{intro:lasso_dist}), universality is also confirmed for the distributions of the Lasso residual $\hat{r}_A^{\lasso}\equiv Y-A\hat{\mu}_A^{\lasso}$ as a scaled convolution of $\xi$ and an extra Gaussian noise, of the subgradient $\hat{v}_A^{\lasso}\equiv \lambda^{-1}A^\top(Y-A\hat{\mu}_A^{\lasso})$ as a random variable taking value in the hypercube $[-1,1]^n$, and of the sparsity $\hat{s}_A^{\lasso}\equiv \pnorm{\hat{\mu}_A^{\lasso}}{0}/n$ as a discrete random variable; see Theorem \ref{thm:lasso_dist} for precise statements. Similar distributional universality properties are proved for the Ridge estimator and its residual; see Theorem \ref{thm:ridge_dist} for details. Using these Lasso universality results, we further verify asymptotic normality of the so-called degrees-of-freedom (dof) adjusted debiased Lasso (cf. \cite{javanmard2014hypothesis,miolane2021distribution,celentano2020lasso,bellec2021debias,bellec2022biasing}) under general design and error distributions; see Theorem \ref{thm:debiased_lasso} for details. This universality result validates statistical inference procedures based on dof adjusted debiased Lasso methodologies in the proportional regime (\ref{intro:proportional_regime}), beyond the exclusive focus on Gaussian designs in previous works (cited above).

It is worth mentioning that using the CGMT machinery (or AMP techniques), the emergence of the Gaussian component in (\ref{intro:lasso_dist}) for $\hat{\mu}_A^{\lasso}$ (or other quantities above) is crucially tied to the Gaussianity of the design matrix. As such, an interesting conceptual consequence of our universality results is to retrieve---in the challenging proportional regime (\ref{intro:proportional_regime})---the `traditional wisdom' that the Gaussianity in  $\hat{\mu}_A^{\lasso}$ origins from aggregation effects of the errors (or the design entries for some of the other quantities) rather than the specificity of design distributions.

For robust regression estimators, our universality results in Theorems \ref{thm:robust_universality_generic} and \ref{thm:robust_risk_asymp}, although proved using the general-purpose universality framework,  compare favorably to previous attempts by \cite{elkaroui2013asymptotic,elkaroui2018impact} using problem-specific techniques. In particular, \cite{elkaroui2013asymptotic,elkaroui2018impact} require strong regularity conditions on the loss function that exclude the canonical Huber/absolute losses, along with a strong exponential moment condition on the design. In contrast, our results hold under a wide range of non-smooth robust loss functions (including the canonical Huber/absolute losses), a much weaker $6+\epsilon$ moment assumption on the design matrix $A$, and no moment assumption on the error $\xi$.

\subsection{Universality of general cost optimum}

The proof of our universality framework relies on comparison inequalities for the optimum of the cost function $w \mapsto H_{\psi_0,\mathsf{f}}(w,A,\xi)$ over a generic structure set $\mathcal{S}_n \subset \R^n$. In particular, we show that for any structure set $\mathcal{S}_n \subset [-L_n,L_n]^n$ with $L_n\geq 1$ growing mildly,
\begin{align}\label{intro:comparison_ineq}
\E \mathsf{g}\Big(\min_{w \in \mathcal{S}_n } H_{\psi_0,\mathsf{f}}(w,A,\xi)\Big) \approx \E  \mathsf{g}\Big(\min_{w \in \mathcal{S}_n } H_{\psi_0,\mathsf{f}}(w,B,\xi)\Big) \,\,\, \hbox{for all } \mathsf{g} \in C^3(\R), 
\end{align}
whenever (i) the design matrices $A,B$ possess independent entries with matching first two moments to the standard Gaussian design $G$ in (\ref{intro:A}), (ii) the loss function $\psi_0$ grows mildly at $\infty$ and its derivatives satisfy certain `self-similarity' properties, and (iii) $\mathsf{f}$ enjoys certain degree of continuity. See Theorem \ref{thm:universality_smooth} for a formal statement that holds for a more general class of cost functions. 

The proof of the comparison inequality (\ref{intro:comparison_ineq}) is based on the quantitative Lindeberg's method (cf. \cite{chatterjee2006generalization}), coupled with an almost dimension-free third derivative bound for every $\mathcal{S}_n$ with the prescribed $\ell_\infty$ constraint. The $\ell_\infty$ constraint plays a crucial role in circumventing the undesirable yet unavoidable logarithmic dependence on the `effective size' in the minimum that scales exponentially in $n$, previously obtained in the high dimensional central limit theorem literature (see e.g., \cite{chernozhukov2022high} for a recent review). These techniques are further generalized to a class of Gordon's max-min (or min-max) cost optimum. Let $
X_n(u,w;A)\equiv m^{-1} u^\top A w+ Q_n(u,w)$. We show that for any pair of structure sets $\mathcal{S}_u \subset [-L_u,L_u]^m$ and $\mathcal{S}_w \subset [-L_w,L_w]^n$  with $L_u,L_w\geq 1$ growing mildly, 
\begin{align}\label{intro:comparison_gordon}
\E \mathsf{g}\bigg(\max_{u \in \mathcal{S}_u} \min_{w \in \mathcal{S}_w} X_n(u,w;A)\bigg)\approx\E \mathsf{g}\bigg(\max_{u \in \mathcal{S}_u} \min_{w \in \mathcal{S}_w} X_n(u,w;B)\bigg) \,\,\,\hbox{for all } \mathsf{g}\in C^3(\R),
\end{align}
again whenever (i) the design matrices $A,B$ possess independent entries with matching first two moments to the standard Gaussian design $G$ in (\ref{intro:A}), and (ii) $Q_n$ enjoys certain degree of continuity. See Theorem \ref{thm:min_max_universality} and Corollary \ref{cor:min_max_universality} for formal statements. In the regression examples we study here, we use the comparison inequality (\ref{intro:comparison_gordon}) to derive universality properties beyond the regression estimator itself, but we also expect it to be of broader applicability in view of its intimate resemblance to the `primal optimization problem' in the CGMT machinery (cf. Theorem \ref{thm:CGMT}).

\subsection{Related literature and non-universality for general isotropic designs}
A number of universality results are obtained for design matrices consisting of independent entries in the proportional regime (\ref{intro:proportional_regime}). \cite{korada2011applications} obtained, among other results, asymptotic universality of box-constrained Lasso cost optimum. \cite{montanari2017universality} obtained asymptotic universality for the elastic net. \cite{panahi2017universal} obtained asymptotic universality for certain special test functions applied to the least squares regression coefficients with strongly convex penalties, along with some results on Lasso; see Section \ref{section:ex_lasso} for a more detailed comparison. Universality results for various quantities of interest in noiseless random linear inverse problems are obtained in \cite{bayati2015universality,oymak2018universality,abbasi2019universality}. As mentioned above, \cite{elkaroui2013asymptotic,elkaroui2018impact} obtained universality of precise risk asymptotics and residual distributions in the context of robust regression. To the best of our knowledge, none of these methods are generally compatible with the CGMT, and nor are applicable for studying universality properties of the broad class of regularized regression estimators (\ref{def:estimator_generic}).

Going beyond independent components, universality results are also obtained in several interesting models, under designs (features) whose rows have  matching first two moments. For instance, \cite{hu2020universality} obtained asymptotic universality results concerning training/generalization errors in the random feature model, between non-Gaussian features (non-linear transforms of underlying Gaussian feature matrix and input vector) and `linearized' Gaussian features, thus verifying a so-called Gaussian equivalence conjecture (cf. \cite{gerace2020generalisation,loureiro2021learning,goldt2022gaussian}).  \cite{montanari2022universality} obtained further asymptotic universality results for these errors, under an `asymptotic Gaussian' assumption on the feature vectors (see Assumption 6 therein), that apply to other significant models including the two-layer neural tangent model. \cite{gerace2022gaussian} obtained universality for the training loss of ridge regularized generalized linear classification with random labels, under a similar asymptotic Gaussian assumption (see Assumption 4 therein). These results motivate the natural question:
\begin{question}
Do the forgoing structural universality properties for regularized regression estimators (\ref{def:estimator_generic}) proved for entrywise independent designs, extend to isotropic designs or more general row independent designs with matching first two moments?
\end{question}
In Section \ref{section:non_universality}, we answer the above question in the negative, by showing that risk universality for the simple ordinary least squares estimator in the basic linear model (\ref{eqn:reg_model}) already fails to hold under an explicitly constructed row independent isotropic design. As will be clear therein, the failure of risk universality is intrinsically due to the non-universality of the spectrum of the sample covariance for general i.i.d. samples of isotropic random vectors. Simulation results in Section \ref{section:simulation} further confirm this risk non-universality phenomenon for the Ridge and Lasso estimators under the same isotropic design used in the construction of the counterexample.

Universality results of a different nature, for instance under rotational invariance assumptions on the design matrix, are obtained in \cite{gerbelot2020asymptotic} for a class of regularized least squares problems with convex penalties (depending on the universality target, strong convexity may be required), and in \cite{gerbelot2020asymptotic2} for a broader class of generalized linear estimation problems.

\subsection{Organization}

The rest of the paper is organized as follows. Section \ref{section:theory} presents comparison inequalities for general cost optimum in (\ref{intro:comparison_ineq}) and Gordon's min-max cost optimum in (\ref{intro:comparison_gordon}). As an application of these comparison inequalities, we establish the structural universality framework in Section \ref{subsection:universality_framework}. Examples on the Ridge, Lasso and regularized robust regression estimators are detailed in Sections \ref{section:ex_ridge}-\ref{section:ex_robust}. The non-universality counterexample is given in Section \ref{section:non_universality}. Simulation results that confirm both universality and non-universality results are provided in Section \ref{section:simulation}. Most proofs are collected in Sections \ref{section:proof_theory}-\ref{section:proof_robust} and the appendices. 

\subsection{Notation}\label{section:notation}

For any positive integer $n$, let $[n]=[1:n]$ denote the set $\{1,\ldots,n\}$. For $a,b \in \R$, $a\vee b\equiv \max\{a,b\}$ and $a\wedge b\equiv\min\{a,b\}$. For $a \in \R$, let $a_\pm \equiv (\pm a)\vee 0$. For $a>0$, let $\log_+(a)\equiv 1\vee \log(a)$. For $x \in \R^n$, let $\pnorm{x}{p}$ denote its $p$-norm $(0\leq p\leq \infty)$, and $B_{n;p}(R)\equiv \{x \in \R^n: \pnorm{x}{p}\leq R\}$. We simply write $\pnorm{x}{}\equiv\pnorm{x}{2}$ and $B_n(R)\equiv B_{n;2}(R)$. For a matrix $M \in \R^{m\times n}$, let $\pnorm{M}{\op}$ denote the spectral norm of $M$. For a measurable map $f:\R^n \to \R$, let $\pnorm{f}{\lip}\equiv \sup_{x\neq y} \abs{f(x)-f(y)}/\pnorm{x-y}{}$. $f$ is called \emph{$L$-Lipschitz} iff $\pnorm{f}{\lip}\leq L$.

We use $C_{x}$ to denote a generic constant that depends only on $x$, whose numeric value may change from line to line unless otherwise specified. $a\lesssim_{x} b$ and $a\gtrsim_x b$ mean $a\leq C_x b$ and $a\geq C_x b$ respectively, and $a\asymp_x b$ means $a\lesssim_{x} b$ and $a\gtrsim_x b$ ($a\lesssim b$ means $a\leq Cb$ for some absolute constant $C$). For two nonnegative sequences $\{a_n\}$ and $\{b_n\}$, we write $a_n\ll b_n$ (respectively~$a_n\gg b_n$) if $\lim_{n\rightarrow\infty} (a_n/b_n) = 0$ (respectively~$\lim_{n\rightarrow\infty} (a_n/b_n) = \infty$). We follow the convention that $0/0 = 0$. $\bigo$ and $\smallo$ (resp. $\mathcal{O}_{\mathbf{P}}$ and $\mathfrak{o}_{\mathbf{P}}$) denote the usual big and small O notation (resp. in probability). 

For a proper, closed convex function $f$ defined on $\R$, its \emph{Moreau envelope} $\mathsf{e}_f(\cdot;\tau)$ and \emph{proximal operator} $\prox_f(\cdot;\tau)$ for any $\tau>0$ are defined by 
\begin{align*}
\mathsf{e}_f(x;\tau)\equiv \min_{z\in \R}\bigg\{\frac{1}{2\tau}(x-z)^2+f(z)\bigg\},\, \prox_f(x;\tau)\equiv \argmin_{z \in \R} \bigg\{\frac{1}{2\tau}(x-z)^2+f(z)\bigg\}.
\end{align*}
Finally, let for $p>0,z\in \R, \lambda\geq 0$
\begin{align}\label{def:eta_1}
\eta_p(z;\lambda)\equiv \argmin_{x\in \R^n}\bigg\{\frac{1}{2}\pnorm{z-x}{}^2+\lambda\cdot \frac{\pnorm{x}{p}^p}{p}\bigg\}=\prox_{\pnorm{\cdot}{p}^p/p}(z;\lambda).
\end{align}
We will only use $p=1,2$ in this paper.

\section{Universality of general cost optimum}\label{section:theory}

\subsection{Basic setup and assumptions}
Let $A \in \R^{m\times n}$ be a $m\times n$ matrix, $\psi_i: \R \to \R_{\geq 0}, \mathsf{f}:\R^n\to \R$ be measurable functions, and
\begin{align}\label{def:H_psi}
H_\psi(w,A)\equiv \frac{1}{m}\sum_{i=1}^m \psi_i\big((Aw)_i\big) + \mathsf{f}(w).
\end{align}
Let $\bar{H}_{\psi}\equiv m\cdot H_\psi$ be the un-normalized version of $H_\psi$. We will be interested in the universality properties related to the optimum and optimizers of $H_\psi$ with respect to the law of the random matrix $A$ (with independent entries).

First we formalize the precise meaning of the `proportional regime' in (\ref{intro:proportional_regime}).

\begin{assumption}[Proportional regime]\label{assump:setting}
$\tau\leq m/n\leq 1/\tau$ holds for some $\tau \in (0,1)$.
\end{assumption}

Next we state the assumptions on the loss functions $\{\psi_i\}$.

\begin{assumption}[Loss function]\label{assump:loss}
There exist reals $q_\ell\geq 0$ ($\ell=0,1,2,3$), constants $\bar{\rho} \in (0,1], \{L_{\psi_i}\geq 1:i\in[m]\}$ and two measurable functions $\mathscr{D}_{\psi}, \mathscr{M}_\psi: (0,\bar{\rho})\to \R_{\geq 0}$, $\mathscr{M}_\psi(\rho)\leq 1\leq \mathscr{D}_\psi(\rho)$, with the following properties:
\begin{enumerate}
	\item $\psi_i$'s grow at mostly polynomially in the sense that for $i \in [m]$,
	\begin{align*}
	\sup_{x \in \R}\frac{\abs{\psi_i(x)}}{1+\abs{x}^{q_0}}\leq L_{\psi_i}.
	\end{align*}
	\item Smooth approximations $\{\psi_{i;\rho}: \R \to \R_{\geq 0}\}_{\rho \in (0,\bar{\rho})}$ of $\psi_i$ exist so that (i) $\psi_{i;\rho} \in C^3(\R)$, (ii) $\max_{i \in [m]}\pnorm{\psi_{i;\rho}-\psi_i}{\infty}\leq \mathscr{M}_\psi(\rho)$, and (iii)  derivatives of $\{\psi_{i;\rho}\}$ satisfy the following self-bounding property:
	\begin{align*}
	\max_{\ell=1,2,3} \max_{i \in [m]} \sup_{x \in \R} \frac{\abs{\partial^\ell \psi_{i;\rho}(x)}}{1+\abs{\psi_{i;\rho}(x)}^{q_\ell}}&\leq \mathscr{D}_{\psi}(\rho).
	\end{align*}
\end{enumerate}
\end{assumption}

The first requirement (1) says that $\psi_i$ cannot grow too fast at $\infty$. The constants $L_{\psi_i}$ will be important as well; in applications to regression problems in Section \ref{section:examples}, these constants are typically related to the moment of the `errors'. The thrust of the second requirement (2) is that the form of the derivatives of (smoothed versions of) $\psi_i$ should be `similar' to itself. Its statement appears however slightly involved; the purpose of this is to include non-smooth loss functions that occur frequently in robust regression problems. 

For later purposes, we define
\begin{align}\label{def:q_bar}
\mathsf{q}\equiv \max\{q_3,q_1+q_2, 3q_1\},\quad \bar{\mathsf{q}}\equiv q_0 \mathsf{q}+3.
\end{align}
We now give two examples of loss functions $\{\psi_i\}$ that satisfy Assumption \ref{assump:loss} above. 

\begin{example}[Square-type loss]\label{ex:square_loss}
Let $\psi_i(x)\equiv (x-\xi_i)^{2s}$ for some real $\xi_i \in \R$ and $s \in \N$. It is easy to verify Assumption \ref{assump:loss} with $\psi_{i;\rho}\equiv \psi_i$, $q_0=2s$, $L_{\psi_i}=C(1+\xi_i^{2s})$, $q_\ell=(2s-\ell)_+/(2s)$ for $\ell=1,2,3$, and $\mathscr{D}_{\psi}(\cdot) \equiv C, \mathscr{M}_\psi(\cdot)\equiv 0$ for some constant $C>1$ depending on $s$ only. As no smoothing is required, the choice of $\bar{\rho}$ is arbitrary. In the most common square loss case ($s=1$),  $\mathsf{q}=3/2$ and $\bar{\mathsf{q}}=6$. 
\end{example}

\begin{example}[Robust loss]\label{ex:robust_loss}
Let  $\psi_i(x)\equiv \psi_0(x-\xi_i)$ for some absolute continuous function $\psi_0: \R \to \R_{\geq 0}$ with $\abs{\psi_0(0)}\vee \esssup\,\abs{\psi_0'}\leq L_0$, and real $\xi_i \in \R$. Under this condition, Assumption \ref{assump:loss}-(1) is satisfied with $q_0=1$ and $L_{\psi_i}= C L_0(1+\abs{\xi_i})$ for some absolute constant $C>1$. Two concrete examples:
\begin{itemize}
	\item (Least absolute loss) $\psi_0(x)=\abs{x}$, so $q_0=1, L_0=1, L_{\psi_i}=C (1+\abs{\xi_i})$. 
	\item (Huber loss) For any $\eta>0$, let $
	\psi_0(x;\eta)\equiv (x^2/2)\bm{1}_{\abs{x}\leq \eta} + \big(\eta\abs{x}-(\eta^2/2)\big)\bm{1}_{\abs{x}>\eta}$, 
	so $q_0=1, L_0=\eta, L_{\psi_i}=C\eta(1+\abs{\xi_i})$. 
\end{itemize}
Consider the smooth approximation $\psi_{i;\rho}(x)\equiv \E \psi_0(x-\xi_i+\rho Z)$, where $Z\sim \mathcal{N}(0,1)$ and $\rho \in (0,1)$. Lemma \ref{lem:smooth_approx_l1} entails that Assumption \ref{assump:loss} is satisfied with $q_1=q_2=q_3=0$, $\mathscr{D}_{\psi}(\rho)=C\cdot L_0/\rho^2$, $\mathscr{M}_\psi(\rho)=C\cdot L_0\rho$, and $\bar{\rho} = 1/(CL_0)$ for some absolute constant $C>1$. Consequently, $\mathsf{q}=0$ and $\bar{\mathsf{q}}=3$. 
\end{example}

Finally we state the assumption on the random design matrix $A$. 

\begin{assumption}[Design matrix]\label{assump:random}
	Let $A\equiv A_0/\sqrt{m}$, where $A_0\in \R^{m\times n}$ is a random matrix with independent entries $\{A_{0;ij}\}$ such that $\E A_{0;ij}=0$, $\E A_{0;ij}^2=1$ for all $i\in[m], j\in[n]$, and $
	M\equiv  \max_{i \in [m],j\in[n]}\E \abs{A_{0;ij}}^{\bar{\mathsf{q}}}<\infty$ ($\bar{\mathsf{q}}$ defined in (\ref{def:q_bar})). 
\end{assumption}

Here $A_0$ is the standardized version of $A$ with entry-wise variance $1$. The variance scaling $1/m$ (or equivalently $1/n$ under Assumption \ref{assump:setting}) is quite common in the high dimensional asymptotics literature, see e.g. \cite{bayati2012lasso,thrampoulidis2015regularized,donoho2016high,elkaroui2018impact,montanari2018mean,thrampoulidis2018precise,barbier2019optimal,deng2019model,lelarge2019fundamental,salehi2019impact,sur2019modern,montanari2019generalization,miolane2021distribution,bellec2021debias,loureiro2021learning,celentano2022fundamental,han2022noisy,liang2020precise,mei2022generalization,wang2022does,zhang2022modern} for an incomplete list of recent statistical papers on this topic. 

\subsection{Universality of the optimum}

First we establish universality of the cost optimum $\min_{w \in \mathcal{S}_n} H_\psi(w,A)$ with respect to $A$, over an arbitrary structure set $\mathcal{S}_n$ that is `compact' in an $\ell_\infty$ sense. Its proof can be found in Section \ref{section:proof_universality_smooth}.

\begin{theorem}\label{thm:universality_smooth}
	Suppose Assumptions \ref{assump:setting} and \ref{assump:loss} hold. Let $A_0,B_0 \in \R^{m\times n}$ be two random matrices with independent components, such that $\E A_{0;ij}=\E B_{0;ij}=0$ and $\E A_{0;ij}^2 = \E B_{0;ij}^2$ for all $i \in [m], j \in [n]$. Further assume that 
	\begin{align*}
	M\equiv  \max_{i\in[m],j\in[n]}\big(\E \abs{A_{0;ij}}^{\bar{\mathsf{q}}}+\E \abs{B_{0;ij}}^{\bar{\mathsf{q}}}\big)<\infty.
	\end{align*}
	Let $A\equiv A_0/\sqrt{m}$ and $B\equiv B_0/\sqrt{m}$. Then there exists some $C_0=C_0(\tau,q,M)>0$ such that	the following hold\footnote{Here we abbreviate `$C_0$ depends on $\{q_\ell: \ell \in [0:3]\}$ in Assumption \ref{assump:loss}' as `$C_0$ depends on $q$', and simply write `$C_0=C_0(\{q_\ell: \ell \in [0:3]\})$' as `$C_0=C_0(q)$'. The same convention will be adopted in the statements of other results. }: For any $\mathcal{S}_n \subset [-L_n,L_n]^n$ with $L_n\geq 1$, and any $\mathsf{g} \in C^3(\R)$, we have
	\begin{align*}
	&\biggabs{\E \mathsf{g}\Big(\min_{w \in \mathcal{S}_n } H_\psi(w,A)\Big) - \E  \mathsf{g}\Big(\min_{w \in \mathcal{S}_n } H_\psi(w,B)\Big)  }\leq C_0\cdot K_{\mathsf{g}}\cdot \mathsf{r}_{\mathsf{f}}(L_n).
	\end{align*}
	Here $K_{\mathsf{g}}\equiv 1+\max_{\ell \in [0:3]} \pnorm{\mathsf{g}^{(\ell)}}{\infty}$, and $\mathsf{r}_{\mathsf{f}}(L_n)$ is defined by
	\begin{align*}
	\mathsf{r}_{\mathsf{f}}(L_n)&\equiv \inf_{\rho \in (0,\bar{\rho})} \bigg\{\mathscr{M}_\psi(\rho) + \mathscr{D}_{\psi}^3(\rho)\inf_{\delta \in (0,\omega_n)}\bigg[ \mathscr{M}_{\mathsf{f}}(L_n,\delta) +\Av^{1/3}\big(\{L_{\psi_i}^{\mathsf{q}}\}\big)\cdot \frac{L_n^{\bar{\mathsf{q}}/3}  \log_+^{2/3}(L_n/\delta)}{n^{1/6} }\bigg]\bigg\},
	\end{align*}
	where $\omega_n\equiv n^{-(q_0q_1+4)/2}$, $\Av\big(\{L_{\psi_i}^{\mathsf{q}}\}\big)\equiv m^{-1}\sum_{i=1}^m L_{\psi_i}^{\mathsf{q}}$, and
	\begin{align*}
	\mathscr{M}_{\mathsf{f}}(L_n,\delta)\equiv \sup\, \abs{\mathsf{f}(w)-\mathsf{f}(w') }
	\end{align*}
	with the supremum taken over all $w,w' \in [-L_n,L_n]^n$ such that $\pnorm{w-w'}{\infty}\leq \delta$. Consequently, for any $z\in \R,\epsilon>0$,
	\begin{align*}
	&\Prob\Big(\min_{w \in \mathcal{S}_n } H_\psi(w,A)>z+3\epsilon \Big)\leq \Prob\Big(\min_{w \in \mathcal{S}_n } H_\psi(w,B)>z+\epsilon \Big)+ C_1(1\vee \epsilon^{-3}) \mathsf{r}_{\mathsf{f}}(L_n).
	\end{align*}
	Here $C_1>0$ is an absolute multiple of $C_0$. 
\end{theorem}

The strength of Theorem \ref{thm:universality_smooth} rests in allowing arbitrary structure sets $\mathcal{S}_n \subset [-L_n,L_n]^n$, where $L_n$ grows slowly enough in the high dimensional limit $n\to \infty$. Of course, there is no apriori reason to believe that the exponent $1/6$ in the rate $\mathsf{r}_{\mathsf{f}}(L_n)$ is optimal. In fact, we have made no efforts to optimize this exponent, as it appears of little importance in the applications in Section \ref{section:examples}. 

To put this result in a broader context, Theorem \ref{thm:universality_smooth} is closely related to recent developments on the high-dimensional central limit theorems (see e.g., the review article \cite{chernozhukov2022high} for many references). This line of works considers universality of the maxima $S_m(X)\equiv \max_{1\leq j\leq p}m^{-1/2}\sum_{i=1}^m X_{ij}$ (or its variants), where $X\equiv \{X_i \in \R^p\}$ contains $m$ independent $p$-dimensional random vectors $X_i$'s. A crucial step therein is to exploit the $\ell_\infty$-like structure, i.e., the maximum over $j \in [p]$, so that the resulting Berry-Esseen bounds scale as a multiple of $(\log^a p/m)^b$, with the optimal choice $a= 3, b= 1/2$ (e.g. \cite{fang2021high,chernozhukov2020nearly}). The situation in Theorem \ref{thm:universality_smooth} is more subtle: in the proportional regime $m\asymp n$, the typical `effective dimension $p$' of $\mathcal{S}_n$ in the minimum scales as $\log p \asymp n\asymp m$, so existing techniques in the above references do not (cannot) yield meaningful bounds. Inspired by the derivative calculations in \cite{korada2011applications,oymak2018universality} embedded in the quantitative Lindeberg's method (cf. \cite{lindeberg22eine,chatterjee2006generalization}), here we get around this issue by providing an almost dimension-free third derivative bound along the entire Lindeberg path, using essentially the $\ell_\infty$ constraint on $\mathcal{S}_n$; see Section \ref{section:proof_universality_smooth} for details.

\subsection{Universality of the optimizer}

Next we will establish universality properties for the optimizer of the cost function $w\mapsto H_\psi(w,A)$, defined as any minimizer
\begin{align}
\hat{w}_A \in \argmin_{w \in \R^n} H_\psi(w,A).
\end{align}
Recall the standard Gaussian design $G$ in (\ref{intro:A}). The following result is proved in Section \ref{section:proof_universality_optimizer}.

\begin{theorem}\label{thm:universality_optimizer}
Suppose Assumptions \ref{assump:setting}, \ref{assump:loss} and \ref{assump:random} hold. Fix a measurable subset $\mathcal{S}_n \subset \R^n$. Suppose there exist $z \in \R$, $\rho_0>0$, $L_n\geq 1$ and $\epsilon_n \in [0,1/4)$ such that the following hold:
\begin{enumerate}
	\item[(O1)] Both $\pnorm{\hat{w}_G}{\infty}$ and $\pnorm{\hat{w}_A}{\infty}$ grow mildly in the sense that
	\begin{align*}
	\Prob\big(\pnorm{\hat{w}_G}{\infty}>L_n\big)\vee \Prob\big(\pnorm{\hat{w}_A}{\infty}>L_n\big)\leq \epsilon_n.
	\end{align*}
	\item[(O2)] $\hat{w}_G$ violates the structural property $\mathcal{S}_n$ in the sense that
	\begin{align*}
	\Prob\bigg(\min_{w \in \R^n} H_\psi(w,G)\geq z+\rho_0 \bigg)\vee 	\Prob\bigg(\min_{w \in \mathcal{S}_n} H_\psi(w,G)\leq z+2\rho_0 \bigg) \leq \epsilon_n. 
	\end{align*}
	
\end{enumerate}
Then $\hat{w}_A$ also violates  $\mathcal{S}_n$ with high probability: 
\begin{align*}
\Prob\big(\hat{w}_A \in \mathcal{S}_n\big)\leq 4\epsilon_n + C_0 (1\vee \rho_0^{-3}) \mathsf{r}_{\mathsf{f}}(L_n).
\end{align*}
Here $\mathsf{r}_{\mathsf{f}}(L_n)$ is defined in Theorem \ref{thm:universality_smooth}, and $C_0>0$ depends on $\tau,q,M$ only.
\end{theorem}
Here $\mathcal{S}_n$ is regarded as the `exceptional set' into which the optimizer $\hat{w}_A$ is unlikely to fall. As mentioned in the Introduction, while condition (O1) typically requires application-specific techniques, this will be verified in a unified manner via `leave-one-out' techniques in the regression examples in Section \ref{section:examples}. To verify condition (O2), $z$ is usually regarded as the `high-dimensional limit' of the global minimum $\min_{w \in \R^n} H_\psi(w,G)$, while $\rho_0>0$ is a small enough constant (typically of constant order) that guarantees the cost function admits a strict gap between the global optimum and the optimum over the exceptional set $\mathcal{S}_n$. See the discussion after Theorem \ref{thm:universality_reg} for more details in applications to regression problems.

\subsection{Universality of the Gordon's max-min (min-max) cost optimum}

The techniques in proving Theorem \ref{thm:universality_smooth} can be generalized to establish universality for Gordon's max-min (min-max) cost optimum, defined as below. Let for $u \in \R^m, w \in \R^n, A \in \R^{m\times n}$ and a measurable function $Q: \R^m\times \R^n \to \R$
\begin{align}
X(u,w;A)\equiv  u^\top A w + Q(u,w).
\end{align}
\begin{theorem}\label{thm:min_max_universality}
	Let $A,B \in \R^{m\times n}$ be two random matrices with independent entries and matching first two moments, i.e., $\E A_{ij}^\ell = \E B_{ij}^\ell$ for all $i \in [m], j\in [n],\ell=1,2$. There exists a universal constant $C_0>0$ such that the following hold. For any measurable subsets $\mathcal{S}_u \subset [-L_u,L_u]^m$, $\mathcal{S}_w \subset [-L_w,L_w]^n$ with $L_u,L_w\geq 1$, and any $\mathsf{g} \in C^3(\R)$, we have
	\begin{align*}
	&\biggabs{\E \mathsf{g}\bigg(\max_{u \in \mathcal{S}_u} \min_{w \in \mathcal{S}_w} X(u,w;A)\bigg)-\E \mathsf{g}\bigg(\max_{u \in \mathcal{S}_u} \min_{w \in \mathcal{S}_w} X(u,w;B)\bigg)}\\
	&\leq C_0\cdot K_{\mathsf{g}}\cdot \inf_{\delta \in (0,1)} \bigg\{ M_1 L\delta+ \mathscr{M}_Q(L,\delta)+ \log_+^{2/3}(L/\delta)\cdot (m+n)^{2/3} M_3^{1/3}L^2\bigg\}.
	\end{align*}
	Here $K_{\mathsf{g}}\equiv 1+\max_{\ell \in [0:3]} \pnorm{\mathsf{g}^{(\ell)}}{\infty}$, $L\equiv L_u+L_w$, $M_\ell\equiv \sum_{i \in [m], j \in [n]} \big(\E \abs{A_{ij}}^\ell+\E\abs{B_{ij}}^\ell\big)$ and
	\begin{align*}
	\mathscr{M}_Q(L,\delta)&\equiv \sup\,\abs{Q(u,w)-Q(u',w')}
	\end{align*}
	with the supremum taken over all $u,u' \in [-L,L]^m, w,w' \in [-L,L]^n$ such that $\pnorm{u-u'}{\infty}\vee \pnorm{w-w'}{\infty}\leq \delta$.
	The conclusion continues to hold when max-min is flipped to min-max. 
\end{theorem}

The proof of the above theorem can be found in Section \ref{section:proof_universality_min_max}. Due to widespread applications of the CGMT in the theoretical analysis of high-dimensional/over-parametrized statistical models (as mentioned in the Introduction), we expect Theorem \ref{thm:min_max_universality} to be useful in establishing universality properties for statistical estimators in other high dimensional problems beyond the ones considered in this paper. Pertinent to this paper, this result will also be useful in some of the applications in Section \ref{section:examples} below.

To facilitate easy applications of Theorem \ref{thm:min_max_universality}, below we work out a particularly useful version where the design matrices have centered entries with variance $1/m$. Its proof is contained in Section \ref{section:proof_universality_min_max_cor}.

\begin{corollary}\label{cor:min_max_universality}
Suppose Assumption \ref{assump:setting} holds. Let $A_0,B_0 \in \R^{m\times n}$ be two random matrices with independent entries, $\E A_{0,ij}^\ell = \E B_{0,ij}^\ell = \bm{1}_{\ell=2}$ for all $i \in [m], j\in [n],\ell=1,2$, and $M_0\equiv  \max_{i \in [m], j \in [n]} (\E \abs{A_{0,ij}}^3+ \E \abs{B_{0,ij}}^3)<\infty$. Let $A\equiv A_0/\sqrt{m}$, $B\equiv B_0/\sqrt{m}$, and recall
\begin{align*}
X_n(u,w;A)= \frac{1}{m} u^\top A w+ Q_n(u,w).
\end{align*}
Then there exists $C_0=C_0(\tau,M_0)>0$ such that for any measurable subsets $\mathcal{S}_u \subset [-L_u,L_u]^m$, $\mathcal{S}_w \subset [-L_w,L_w]^n$ with $L_u,L_w\geq 1$, and any $\mathsf{g} \in C^3(\R)$, we have
\begin{align*}
&\biggabs{\E \mathsf{g}\bigg(\max_{u \in \mathcal{S}_u} \min_{w \in \mathcal{S}_w} X_n(u,w;A)\bigg)-\E \mathsf{g}\bigg(\max_{u \in \mathcal{S}_u} \min_{w \in \mathcal{S}_w} X_n(u,w;B)\bigg)}\\
&\leq C_0\cdot K_{\mathsf{g}}\cdot \mathsf{r}_n,  \text{ with } \mathsf{r}_n \equiv C_0\cdot K_{\mathsf{g}}\cdot \inf_{0<\delta\leq n^{-1}} \bigg\{\mathscr{M}_{Q_n}(L,\delta)+ \frac{L^2 \log^{2/3}(L/\delta)}{n^{1/6}}\bigg\}.
\end{align*}
Here $K_{\mathsf{g}}\equiv 1+\max_{\ell \in [0:3]} \pnorm{\mathsf{g}^{(\ell)}}{\infty}$ and $L\equiv L_u+L_w$. Consequently,
\begin{align*}
&\Prob\Big(\max_{u \in \mathcal{S}_u} \min_{w \in \mathcal{S}_w} X_n(u,w;A)>z+3\epsilon\Big)\\
&\leq \Prob\Big(\max_{u \in \mathcal{S}_u} \min_{w \in \mathcal{S}_w} X_n(u,w;B)>z+\epsilon\Big)+ C_1 (1\vee\epsilon^{-3}) \mathsf{r}_n
\end{align*}
holds for any $z\in \R$ and $\epsilon>0$. Here $C_1>0$ is an absolute multiple of $C_0$. The conclusion continues to hold both when (i) max-min is flipped to min-max, and (ii) there exists some set $S\subset[m]\times [n]$ such that $A_{ij} = B_{ij} = 0$ for $(i,j)\in S$.
\end{corollary}

The extension to scenario (ii) will be useful in situations where the drift function $Q_n$ contains certain extra variable, say, $v$ over which the maximum is also taken. This will be used in the proof of some applications (in particular, the distribution of Lasso subgradient) in Section \ref{section:examples}.

\section{Applications to high dimensional regression}\label{section:examples}

\subsection{General regression setting}\label{subsection:universality_framework}
In the linear regression model (\ref{eqn:reg_model}), recall $A=A_0/\sqrt{m}$ and we also write $\xi=\sigma \xi_0$, where the variance of the entries of $A_0$ and $\xi_0$ are standardized to be $1$. Further recall that $\hat{\mu}_A\equiv \mu_0+\hat{w}_A$, where the estimator of interest $\hat{\mu}_A$ is defined in (\ref{def:estimator_generic}) and $\hat{w}_A$ is defined in (\ref{def:w_generic}).

Below we will work out Theorem \ref{thm:universality_optimizer} in the above regression setting for the square loss $\psi_0(x)=x^2/2$ (Example \ref{ex:square_loss}) and the robust loss $\psi_0$ (Example \ref{ex:robust_loss}). While we do not pursue the most general possible form here, adaptation to other loss functions is straightforward. 

\begin{theorem}\label{thm:universality_reg}
Consider the above regression setting with either (i) square loss $\psi_0(x)=x^2/2$ or (ii) robust loss $\psi_0$ satisfying $\abs{\psi_0(0)}\vee \esssup\, \abs{\psi_0'}\leq L_0$ for some $L_0>0$. 
Suppose Assumption \ref{assump:setting} holds, Assumption \ref{assump:random} holds with
\begin{align*}
M\equiv 
\begin{cases}
 \max_{i,j}\E \abs{A_{0;ij}}^6<\infty,& \hbox{square loss case};\\
  \max_{i,j}\E \abs{A_{0;ij}}^3<\infty,& \hbox{robust loss case},
\end{cases}
\end{align*}
and $\xi_0$ has independent components that are also independent of $A_0$. Further assume that there exists some $K_{\mathsf{f}}>1$ such that
\begin{align}\label{cond:f_moduli}
\log \mathscr{M}_{\mathsf{f}}(L,\delta)\leq K_{\mathsf{f}} \big(\log L+\log n\big) - \log(1/\delta)/K_{\mathsf{f}},\quad \forall L\geq 1,\delta \in (0,1).
\end{align}
Fix a measurable subset $\mathcal{S}_n \subset \R^n$. Suppose there exist $z \in \R$, $\rho_0>0$, $1\leq L_n\leq n$ and $\epsilon_n \in [0,1/4)$ such that (O1) in Theorem \ref{thm:universality_optimizer} is fulfilled under the joint probability of $(A,\xi)$ and $(G,\xi)$, and (O2) fulfilled for $H_{\psi_0,\mathsf{f}}(\cdot,G,\xi)$ under the joint probability of $(G,\xi)$. Then under the joint probability of $(A,\xi)$,
	\begin{align*}
	\Prob\big(\hat{w}_A \in \mathcal{S}_n\big)
	&\leq 4\epsilon_n + C_0 (1\vee \rho_0^{-3})
	\begin{cases}
	 M_\xi^{1/3} \frac{ L_n^2 \log^{2/3} n}{n^{1/6} },& \hbox{square loss case};\\
     \big(\frac{ L_n \log^{2/3} n}{n^{1/6}}\big)^{1/7},& \hbox{robust loss case}.
	\end{cases}
	\end{align*}
	Here $M_\xi \equiv 1+m^{-1}\sum_{i=1}^m \E \abs{\sigma\xi_{0,i}}^{3}$. The constant $C_0>0$ depends on $\tau,M,K_{\mathsf{f}}$ only in the square loss case and depends further on $L_0$ in the robust loss case. 
\end{theorem}

The condition (\ref{cond:f_moduli}) on the penalty function $\mathsf{f}$ is imposed here to simplify the final bound, and is easily verified as long as $\mathsf{f}$ has some degree of global moduli of continuity. The major non-trivial condition is the $\ell_\infty$ bounds for $\hat{w}_A$ and $\hat{w}_G$ required in (O1). In the examples to be detailed below, we will use the so-called `leave-one-out' method to establish the desired $\ell_\infty$ bounds. Formally, we study perturbation of $\hat{w}_A$ by (i) its column leave-one-out version
\begin{align*}
\hat{w}^{(s)}_A\equiv \argmin_{w\in\R^n: w_s = 0} \bigg\{\frac{1}{m}\sum_{i=1}^m \psi_0\big((Aw)_i-\xi_i\big)+\mathsf{f}(\mu_0+w)\bigg\},\quad s \in [n], 
\end{align*}
and (ii) its row leave-one-out version
\begin{align*}
\hat{w}^{[t]}_A\equiv  \argmin_{w\in\R^n} \bigg\{\frac{1}{m}\sum_{i \in [m], i\neq t} \psi_0\big((Aw)_i-\xi_i\big)+\mathsf{f}(\mu_0+w)\bigg\},\quad t \in [m].
\end{align*}
Intuitively, both $\hat{w}^{(s)}_A$ and $\hat{w}^{[t]}_A$ should be very close to $\hat{w}_A$ for designs $A$ with independent entries. We will show that indeed in many examples the orders of $\pnorm{\hat{w}^{(s)}_A-\hat{w}_A}{}, \pnorm{\hat{w}^{[t]}_A-\hat{w}_A}{}$ are almost $\bigop(1)$, while the typical order of $\pnorm{\hat{w}_A}{}$ is $\bigop(n^{1/2})$. The independence of $\hat{w}^{(s)}_A$ (resp. $\hat{w}^{[t]}_A$) with respect to the $s$-th column (resp. $t$-th row) of $A$, will then play a crucial role in establishing element-wise bounds for $\hat{w}_A$. 

The method described above is closely related to the one used in  \cite{elkaroui2018impact} under the name `leave-one-observation/predictor out approximations'; see also \cite{montanari2017universality,javanmard2018debiasing} for related techniques. 

Once the $\ell_\infty$ bound condition (O1) is verified, we then only need to study the behavior of $\hat{w}_G$ for the standard Gaussian design $G$ (\ref{intro:A}) with i.i.d. $\mathcal{N}(0,1/m)$ entries, by creating an $\bigo(1)$ gap between $\min_{w \in \mathcal{S}_n} H_{\psi_0,\mathsf{f}}(w,G,\xi)$ over the `exceptional set' $\mathcal{S}_n$ and the global optimum $\min_{w \in \R^n} H_{\psi_0,\mathsf{f}}(w,G,\xi)$. Such a goal is particularly amenable to analysis via the CGMT, as it reduces the analysis of $H_{\psi_0,\mathsf{f}}(w,G,\xi)$ that involves a standard Gaussian design matrix $G$ to a Gordon problem that involves two Gaussian vectors only. 

\begin{proof}[Proof of Theorem \ref{thm:universality_reg}]
	By the assumed moduli of continuity in $\mathsf{f}$, we may take $\delta=(L_n n)^{-K_{\mathsf{f}}-1}$ in the definition of $\mathsf{r}_{\mathsf{f}}(L_n)$. Now we apply Theorem \ref{thm:universality_optimizer} first conditionally on $\xi$ and then take expectation. There in the square loss case, $L_{\psi_i}=C(1+\xi_i^2)$, $\mathsf{q}=3/2$, $\bar{\mathsf{q}}=6$, $\mathscr{D}_\psi(\rho)\equiv C$ and $\mathscr{M}_\psi(\rho)=0$, so
	\begin{align*}
	\mathsf{r}_{\mathsf{f}}(L_n)\lesssim M_\xi^{1/3} \cdot \frac{ L_n^2 \log^{2/3} (nL_n)}{n^{1/6}}.
	\end{align*}
	In the robust loss case, $L_{\psi_i}=C L_0 (1+\abs{\xi_i})$, $\mathsf{q}=0$, $\bar{\mathsf{q}}=3$, $\mathscr{D}_\psi(\rho)\equiv C L_0/\rho^2$, $\mathscr{M}_\psi(\rho)=C L_0\rho$ and $\bar{\rho}=1/(C L_0)$, so
	\begin{align*}
	\mathsf{r}_{\mathsf{f}}(L_n)\lesssim_{L_0} \inf_{\rho \in (0,c')}\bigg\{\rho+\rho^{-6}\frac{ L_n \log^{2/3} (nL_n)}{n^{1/6}}\bigg\}\asymp_{L_0} \bigg(\frac{ L_n \log^{2/3} (nL_n)}{n^{1/6}}\bigg)^{1/7}.
	\end{align*}	
	In both displays, the term $\log^{2/3}(nL_n)\asymp \log^{2/3} n$  thanks to $1\leq L_n\leq n$.
\end{proof}

\subsection{Example I: Ridge regression}\label{section:ex_ridge}

In this section, we consider universality properties for the Ridge estimator \cite{hoerl1970ridge}. Formally, let the Ridge cost function be
\begin{align}\label{def:ridge_cost}
\bar{H}^{\ridge}(w,A,\xi)\equiv \frac{1}{2}\pnorm{Aw - \xi}{}^2 + \frac{\lambda}{2}\pnorm{w + \mu_0}{}^2,
\end{align}
and its normalized version $H^{\ridge}\equiv \bar{H}^{\ridge}/m$. The Ridge solution is given by $\hat{\mu}^{\ridge}_A=\hat{w}^{\ridge}_A+\mu_0$ with
\begin{align}\label{eqn:ridge}
\hat{w}_A^{\ridge}\equiv \argmin_{w\in\R^n} H^{\ridge}(w,A,\xi)=(A^\top A+\lambda I)^{-1}(A^\top \xi-\lambda \mu_0).
\end{align}
We will work with the following conditions instead of referring back to the assumptions listed in Section \ref{section:theory}:
\begin{enumerate}
	\item[(R1)] $\tau\leq m/n\leq 1/\tau$ holds for some $\tau \in (0,1)$, and $\lambda>0$.
	\item[(R2)] $\pnorm{\mu_0}{}^2/n\leq M_2$ for some $M_2 > 0$.
	\item[(R3)] $A_0=\sqrt{m}A$ and $\xi_0 =\xi/\sigma$ are independent, and their entries are all independent, mean $0$, variance $1$ and uniformly sub-Gaussian variables.
\end{enumerate}
The precise mathematical meaning of uniform sub-Gaussianity in (R3) is that $\sup_n \max_{i,j} (\pnorm{A_{0,ij}}{\psi_2}+\pnorm{\xi_{0,i}}{\psi_2})<\infty$, where $\pnorm{\cdot}{\psi_2}$ is the Orcliz-2 norm or the sub-Gaussian norm (definition see \cite[Section 2.1]{van1996weak}). 

The following theorem establishes the generic universality of $\hat{w}_A^{\ridge}$ with respect to the design matrix $A$. All proofs in this section can be found in Section \ref{section:proof_ridge}.

\begin{theorem}\label{thm:ridge_universality_generic}
Suppose (R1)-(R3) hold. Fix $\mathcal{S}_n \subset \R^n$. Suppose there exist $z \in \R, \rho_0>0$ and $\epsilon_n \in [0,1/4)$ such that 
\begin{align*}
\Prob\bigg(\min_{w \in \R^n} H^{\ridge}(w,G,\xi)\geq z+\rho_0 \bigg)\vee 	\Prob\bigg(\min_{w \in \mathcal{S}_n} H^{\ridge}(w,G,\xi)\leq z+2\rho_0 \bigg) \leq \epsilon_n. 
\end{align*}
Then there exists some $K=K(\sigma,\lambda,\tau,M_2)>0$ such that
\begin{align*}
\Prob\big(\hat{w}_A^{\ridge} \in \mathcal{S}_n\big)\leq 4\epsilon_n+K  \cdot (1\vee \rho_0^{-3}) (1\vee\pnorm{\mu_0}{\infty}^2)\cdot  n^{-1/6}\log^2 n .
\end{align*}
\end{theorem}

The sub-Gaussian moments in (R3) are assumed for simplicity; easy modifications of the proofs allow for weaker conditions, for example the existence of high enough moments, at the cost of a possible worsened probability bound. We do not pursue these non-essential refinements here for clarity of presentation and proofs.

The key to the proof of Theorem \ref{thm:ridge_universality_generic} is the following $\ell_\infty$ bound (and other risk bounds) for $\hat{w}_A^{\ridge}$, which may be of independent interest.

\begin{proposition}\label{prop:ridge_risk}
	Assume the same conditions as in Theorem \ref{thm:ridge_universality_generic}. Then the following hold with probability at least $1 -Cn^{-100}$ with respect to the randomness of $(A,\xi)$:
	\begin{enumerate}
		\item (Prediction risk) $\pnorm{A\hat{w}^{\ridge}_A}{}^2 \leq K\cdot n$.
		\item ($\ell_\infty$ risk)  $\pnorm{\hat{w}^{\ridge}_A}{\infty} \leq K\sqrt{\log n}+2\pnorm{\mu_0}{\infty}$.
		\item (Prediction $\ell_\infty$ risk) $\pnorm{A\hat{w}^{\ridge}_A}{\infty} \leq K\sqrt{\log n}$.
	\end{enumerate}
	Here $C,K>0$ depend on $\sigma,\lambda,\tau,M_2$ only. 
\end{proposition}

Thanks to the closed form formula for the Ridge estimator in (\ref{eqn:ridge}), some of its properties that can be related to the spectrum of $A^\top A$ including estimation error/prediction error (in the Euclidean norm), can be established directly via existing random matrix theory (RMT) (cf. \cite{dicker2016ridge,dobriban2018high,hastie2022surprises,mei2022generalization}). Here we will illustrate the power of Theorem \ref{thm:ridge_universality_generic} and its compatibility with the CGMT, by establishing non-asymptotic distributional approximations of $\hat{w}_A^{\ridge}$ and its residual given by $\hat{r}^{\ridge}_A \equiv Y-A\hat{\mu}_A^{\ridge}$. These results appear to be less amenable to direct applications of existing RMT techniques. More important, the formulation of these results suggests natural generalizations to the Lasso case in which no closed form formulas are available (cf. Section \ref{section:ex_lasso}). 

Recall $\eta_2(\cdot;\cdot)$ defined in (\ref{def:eta_1}). By Proposition \ref{prop:ridge_empirical_char}-(2), the system of equations
\begin{align}\label{eqn:ridge_fpe}
(\gamma_\ast^{\ridge})^2& = \sigma^2 + \frac{1}{m/n}\cdot \E\bigg[\eta_2\bigg(\Pi_{\mu_0}+\gamma_\ast^{\ridge} Z; \frac{\gamma_\ast^{\ridge} \lambda }{\beta_\ast^{\ridge}  }\bigg)-\Pi_{\mu_0}\bigg]^2,\nonumber\\
\beta_\ast^{\ridge}& = \gamma_\ast^{\ridge}\bigg[1-\frac{1}{m/n} \cdot \E \eta_2'\bigg(\Pi_{\mu_0}+\gamma_\ast^{\ridge} Z; \frac{\gamma_\ast^{\ridge} \lambda }{\beta_\ast^{\ridge} }\bigg)\bigg],
\end{align}
where $\Pi_{\mu_0}\otimes Z \equiv \big(n^{-1}\sum_{j=1}^n \delta_{\mu_{0,j}}\big)\otimes \mathcal{N}(0,1)$, 
admits a unique solution $(\beta_\ast^{\ridge},\gamma_\ast^{\ridge})$ within compacta of $[0,\infty)^2$ provided (R1)-(R2) are satisfied. Again, one may give explicit formulae for $(\beta_\ast^{\ridge},\gamma_\ast^{\ridge})$ defined using (\ref{eqn:ridge_fpe}), but we stick to the above fixed point equation formulation for transparent comparison to the Lasso case in (\ref{eqn:lasso_fpe}) below. 

Now we define the `population version' of $\hat{w}_A^{\ridge}$ and $\hat{r}_A^{\ridge}$ via $(\beta_\ast^{\ridge},\gamma_\ast^{\ridge})$ by
\begin{align}\label{def:w_ridge}
w_\ast^{\ridge}\equiv  \eta_2\bigg(\mu_0+\gamma_\ast^{\ridge} g; \frac{\gamma_\ast^{\ridge} \lambda}{\beta_\ast^{\ridge}}\bigg)-\mu_0,\quad r_\ast^{\ridge} \equiv \frac{\beta_\ast^{\ridge}}{\gamma_\ast^{\ridge}} \bigg(\sigma\cdot \xi_0+\sqrt{ (\gamma_\ast^{\ridge})^2-\sigma^2}\cdot h\bigg),
\end{align}
where $g \sim \mathcal{N}(0,I_n)$ and $h\sim \mathcal{N}(0,I_m)$ are independent standard Gaussian vectors that are also independent of the noise vector $\xi=\sigma \xi_0$. 

\begin{theorem}\label{thm:ridge_dist}
	Assume the same conditions as in Theorem \ref{thm:ridge_universality_generic}. Then there exists some $K=K(\sigma,\lambda,\tau,M_2)>0$ such that for all $1$-Lipschitz functions $\mathsf{g}:\R^n \to \R$, $\mathsf{h}: \R^m \to \R$ and $\epsilon>0$,
	\begin{align*}
	& \Prob\bigg( \bigabs{ \mathsf{g}\big(\hat{w}_A^{\ridge}/\sqrt{n}\big)- \E \mathsf{g}\big(w_\ast^{\ridge}/\sqrt{n}\big) }\geq \epsilon \bigg)\vee \Prob\bigg( \bigabs{ \mathsf{h}\big(\hat{r}_A^{\ridge}/\sqrt{m}\big)- \E_h \mathsf{h}\big(r_\ast^{\ridge}/\sqrt{m}\big) }\geq \epsilon \bigg)\\
	&\leq K\cdot  (1\vee \epsilon^{-6})(1\vee\pnorm{\mu_0}{\infty}^2)\cdot   n^{-1/6}\log^2 n.
	\end{align*}
	Here $\E_h$ indicates that the expectation is taken with respect to $h$ only. 
\end{theorem}

The first claim on $\hat{w}_A^{\ridge}$ in the above theorem is proved via an application of Theorem \ref{thm:ridge_universality_generic}, by analyzing the corresponding Gaussian design problem via the CGMT. The proof for the second claim on $\hat{r}_A^{\ridge}$ in the above theorem is more involved, and requires an application of the universality result for Gordon's max-min (min-max) cost in Theorem \ref{thm:min_max_universality}.

As a quick application of Theorem \ref{thm:ridge_dist} above, we may obtain universality of the distribution of $\hat{w}_A^{\ridge}$ and $\hat{r}_A^{\ridge}$ in an average sense as follows:
\begin{itemize}
	\item Let $\mathsf{g}(v)\equiv n^{-1}\sum_{j=1}^n \phi(\sqrt{n} v_j+\mu_{0,j},\mu_{0,j})$ for some $1$-Lipschitz function $\phi: \R^2\to \R$ in the first probability, we have with high probability
	\begin{align*}
	\frac{1}{n}\sum_{j=1}^n \phi\big(\hat{\mu}^{\ridge}_{A,j}, \mu_{0,j}\big)\approx  \E \phi\Big( \eta_2\big(\Pi_{\mu_0}+\gamma_\ast^{\ridge} Z; {\lambda\gamma_\ast^{\ridge} }\big/{\beta_\ast^{\ridge}}\big), \Pi_{\mu_0} \Big).
	\end{align*}
	Here the expectation is taken over $\Pi_{\mu_0}\otimes Z= (n^{-1}\sum_{j=1}^n \delta_{\mu_{0,j}})\otimes\mathcal{N}(0,1)$. 
	\item Let $\mathsf{h}(\nu)= m^{-1}\sum_{i=1}^m \phi(\sqrt{m}\nu_i)$ for some $1$-Lipschitz function $\phi: \R\to \R$ in the second probability, we have with high (unconditional) probability
	\begin{align*}
	\frac{1}{m}\sum_{i=1}^m \phi\big(\hat{r}_{A,i}^{\ridge}\big)\approx  \E \phi\bigg[\frac{\beta_\ast^{\ridge}}{\gamma_\ast^{\ridge}} \bigg(\sigma\cdot \Pi_{\xi_0}+\sqrt{ (\gamma_\ast^{\ridge})^2-\sigma^2}\cdot Z\bigg)\bigg].
	\end{align*}
	Here the expectation is taken over $\Pi_{\xi_0}\otimes Z=(m^{-1}\sum_{i=1}^m \delta_{\xi_{0,i}})\otimes \mathcal{N}(0,1)$. 
\end{itemize}

\begin{remark}\label{rmk:ridge_comparison}
We compare Theorem \ref{thm:ridge_dist} to several results in the literature:
\begin{itemize}
	\item For the distribution of $\hat{w}_A^{\ridge}$, \cite{panahi2017universal} obtained the following special version of universality for design matrices consisting of i.i.d. entries with vanishing third/fifth moments (almost symmetry): for convex $\mathsf{g}_0:\R \to \R$ with bounded second and third derivatives, or $\mathsf{g}_0=\bm{1}_{\cdot \geq x}$ for any $x \in \R$, $n^{-1}\sum_{j=1}^n \mathsf{g}_0(\hat{w}_{A,j}^{\ridge})$ and $n^{-1}\sum_{j=1}^n \mathsf{g}_0(\hat{w}_{G,j}^{\ridge})$ converge to the same limit. Our results are non-asymptotic allowing for arbitrary non-separable Lipschitz test functions, and do not require prior distributions on $\mu_0$ and vanishing third/fifth moments (almost symmetry) of the design entries.
	\item For the distribution of $\hat{r}_A^{\ridge}$, \cite[Theorem 3.1]{bellec2021derivatives} obtained stochastic representation of $\hat{r}_G^{\ridge}$ under (correlated) Gaussian designs in a broader class of problems. The results in \cite{bellec2021derivatives} depend on the Gaussian design assumption crucially via repeated applications of Gaussian integration by parts (e.g., the Stein's identity and the second-order Stein formula in \cite{stein1981estimation,bellec2021second}).
\end{itemize}

\end{remark}

\subsection{Example II: Lasso}\label{section:ex_lasso}

In this section, we consider universality properties for the Lasso estimator \cite{tibshirani1996regression}. Formally, let the Lasso cost function be
\begin{align}\label{def:lasso_cost}
\bar{H}^{\lasso}(w,A,\xi)\equiv \frac{1}{2}\pnorm{Aw - \xi}{}^2 + \lambda\pnorm{w + \mu_0}{1},
\end{align}
and its normalized version $H^{\lasso}\equiv \bar{H}^{\lasso}/m$. The Lasso solution is $\hat{\mu}^{\lasso}_A\equiv\hat{w}^{\lasso}_A+\mu_0$ with 
\begin{align*}
\hat{w}^{\lasso}_A\equiv  \argmin_{w\in\R^n} H^{\lasso}(w,A,\xi).
\end{align*}
We continue working with the conditions (R1)-(R3) in Section \ref{section:ex_ridge}. The following theorem establishes the generic universality of $\hat{w}_A^{\lasso}$ with respect to the design matrix $A$. All proofs in this section can be found in Section \ref{section:proof_lasso}.

\begin{theorem}\label{thm:lasso_universality_generic}
	Suppose (R1)-(R3) hold.	Suppose further that $\lambda \geq K_0 (1\vee \sigma)$ for some $K_0=K_0(M_2,\tau)>0$. Fix $\mathcal{S}_n \subset \R^n$. Suppose there exist $z \in \R, \rho_0>0$ and $\epsilon_n \in [0,1/4)$ such that 
	\begin{align*}
	\Prob\bigg(\min_{w \in \R^n} H^{\lasso}(w,G,\xi)\geq z+\rho_0 \bigg)\vee 	\Prob\bigg(\min_{w \in \mathcal{S}_n} H^{\lasso}(w,G,\xi)\leq z+2\rho_0 \bigg) \leq \epsilon_n. 
	\end{align*}
	Then there exists some $K=K(\sigma,\lambda,\tau,M_2)>0$ such that
	\begin{align*}
	\Prob\big(\hat{w}_A^{\ridge} \in \mathcal{S}_n\big)\leq 4\epsilon_n+K \cdot (1\vee \rho_0^{-3}) \cdot n^{-1/6}\log^2 n .
	\end{align*}
	The lower bound on $\lambda$ can be eliminated when $m/n\geq  1+\epsilon$ for some $\epsilon>0$ at the cost of possibly enlarged constants $K$ depending further on $\epsilon$.
\end{theorem}

Note that a lower bound on the tuning parameter $\lambda$ is imposed only in the regime $m/n<1$; a precise value for this lower bound can be found in the statement of Lemma \ref{lem:lasso_sparsity}. Such a condition renders sufficient linear-order sparsity of the regression estimator in the proportional regime $m\asymp n$, and is quite common in the literature; see e.g. \cite{bellec2021debias,bellec2022biasing} for related results in the Gaussian design case.

The key to the proof of Theorem \ref{thm:lasso_universality_generic} is the following $\ell_\infty$ bound (and other risk bounds) for $\hat{w}_A^{\lasso}$, which may be of independent interest.

\begin{proposition}[Lasso risk bounds]\label{prop:lasso_risk}
	Assume the same conditions as in Theorem \ref{thm:lasso_universality_generic}. Suppose $\lambda \geq K_0 (1\vee \sigma)$  for some $K_0=K_0(M_2,\tau)>0$. Then the following holds with probability at least $1 -Cn^{-100}$ with respect to the randomness of $(A,\xi)$:
	\begin{enumerate}
		\item (Prediction risk) $\pnorm{A\hat{w}^{\lasso}_A}{}^2 \leq K\cdot  n$.
		\item ($\ell_\infty$ risk)  $\pnorm{\hat{w}^{\lasso}_A}{\infty} \leq K\sqrt{\log n}$.
		\item (Prediction $\ell_\infty$ risk) $\pnorm{A\hat{w}^{\lasso}_A}{\infty} \leq K \sqrt{\log n}$.
	\end{enumerate}
	Here $C,K>0$ depend on $\sigma,\lambda,\tau,M_2$ only. The lower bound on $\lambda$ can be eliminated when $m/n\geq  1+\epsilon$ for some $\epsilon>0$ at the cost of possibly enlarged constants $C,K$ depending further on $\epsilon$.
\end{proposition}

To the best of our knowledge, $\ell_\infty$ bounds for Lasso in the proportional regime $m\asymp n$ without exact sparsity conditions on $\mu_0$ (or in the linear order sparsity regime) are available only in the Gaussian design case, under a similar lower bound requirement on $\lambda$ when $m/n<1$; see \cite[Theorem 5.1]{bellec2022biasing} for a precise statement. The `interpolation' proof techniques used therein are specific to the Gaussianity of the design matrix, so cannot be extended easily to non-Gaussian design matrices. Here we use leave-one-out methods (as mentioned after Theorem \ref{thm:universality_reg}) to establish $\ell_\infty$ bounds for Lasso for general design matrices in the proportional regime.

Now we give an application of Theorem \ref{thm:lasso_universality_generic}, coupled with the CGMT method and the comparison inequalities in Theorem \ref{thm:min_max_universality} or Corollary \ref{cor:min_max_universality}, that establishes the universality of the distributions of
\begin{itemize}
	\item the error $\hat{w}_A^{\lasso}=\hat{\mu}_A^{\lasso}-\mu_0$, 
	\item the residual $\hat{r}_A^{\lasso}\equiv Y-A\hat{\mu}^{\lasso}_A$, 
	\item the subgradient $\hat{v}^{\lasso}_A\equiv\lambda^{-1}A^\top(Y-A \hat{\mu}_A^{\lasso})$, and
	\item the sparsity $\hat{s}_A^{\lasso} \equiv \pnorm{\hat{\mu}_A^{\lasso}}{0}/n$.
\end{itemize}
Recall $\eta_1(\cdot;\cdot)$ defined in (\ref{def:eta_1}) and $\Pi_{\mu_0}\otimes Z \equiv \big(n^{-1}\sum_{j=1}^n \delta_{\mu_{0,j}}\big)\otimes \mathcal{N}(0,1)$. By \cite{miolane2021distribution}, the system of equations
\begin{align}\label{eqn:lasso_fpe}
(\gamma_\ast^{\lasso})^2& = \sigma^2 + \frac{1}{m/n}\cdot \E\bigg[\eta_1\bigg(\Pi_{\mu_0}+\gamma_\ast^{\lasso} Z; \frac{\gamma_\ast^{\lasso} \lambda }{\beta_\ast^{\lasso}  }\bigg)-\Pi_{\mu_0}\bigg]^2,\nonumber\\
\beta_\ast^{\lasso}& = \gamma_\ast^{\lasso}\bigg[1-\frac{1}{m/n} \cdot \E \eta_1'\bigg(\Pi_{\mu_0}+\gamma_\ast^{\lasso} Z; \frac{\gamma_\ast^{\lasso} \lambda }{\beta_\ast^{\lasso} }\bigg)\bigg]
\end{align} 
admits a unique solution $(\beta_\ast^{\lasso},\gamma_\ast^{\lasso})$ within compacta of $[0,\infty)^2$ provided (R1)-(R2) are satisfied. Let the `population version' of $\hat{w}_A^{\lasso},\hat{r}_A^{\lasso},\hat{v}_A^{\lasso},\hat{s}_A^{\lasso}$ be
\begin{align}
w_\ast^{\lasso}&\equiv  \eta_1\bigg(\mu_0+\gamma_\ast^{\lasso} g; \frac{\gamma_\ast^{\lasso} \lambda}{\beta_\ast^{\lasso}}\bigg)-\mu_0,\quad r_\ast^{\lasso} \equiv \frac{\beta_\ast^{\lasso}}{\gamma_\ast^{\lasso}} \bigg(\sigma\cdot \xi_0+\sqrt{ (\gamma_\ast^{\lasso})^2-\sigma^2}\cdot h\bigg),\label{def:w_lasso}\\
v_\ast^{\lasso} &\equiv  -\frac{\beta_\ast^{\lasso}}{\gamma_\ast^{\lasso}\lambda}\bigg[\eta_1\bigg(\mu_0 + \gamma_\ast^{\lasso} g;\frac{\gamma_\ast^{\lasso}\lambda}{\beta_\ast^{\lasso}}\bigg) - \big(\mu_0 + \gamma_\ast^{\lasso} g\big)\bigg] = -\frac{\beta_\ast^{\lasso}}{\gamma_\ast^{\lasso}\lambda}\big(w_\ast^{\lasso}-\gamma_\ast^{\lasso} g\big),\label{def:v_lasso}\\
s_\ast^{\lasso}& = \E \eta_1'\bigg(\Pi_{\mu_0} + \gamma_\ast^{\lasso} Z;\frac{\gamma_\ast^{\lasso}\lambda}{\beta_\ast^{\lasso}}\bigg) = \Prob\bigg(\abs{ \Pi_{\mu_0} + \gamma_\ast^{\lasso} Z }\geq \frac{\gamma_\ast^{\lasso}\lambda}{\beta_\ast^{\lasso}}\bigg),\label{def:s_lasso}
\end{align}
where $g \sim \mathcal{N}(0,I_n)$ and $h \sim \mathcal{N}(0,I_m)$ are independent standard Gaussian vectors that are also independent of the noise vector $\xi$. Clearly the fixed point equation (\ref{eqn:lasso_fpe}) and the `population' quantities $w_\ast^{\lasso}, r_\ast^{\lasso}$ in (\ref{def:w_lasso}) for the Lasso estimator are in complete analogue to those for the Ridge estimator defined in (\ref{eqn:ridge_fpe}) and (\ref{def:w_ridge}). The `population' quantities $v_\ast^{\lasso},s_\ast^{\lasso}$ are of special interest to the Lasso.

\begin{theorem}\label{thm:lasso_dist}
	Assume the same conditions as in Theorem \ref{thm:lasso_universality_generic}. Then there exists some $K=K(\sigma,\lambda,\tau,M_2)>0$ such that for all $1$-Lipschitz functions $\mathsf{g}:\R^n \to \R, \mathsf{h}: \R^m \to \R$ and $\epsilon>0$, all the following probabilities
	\begin{itemize}
		\item $\Prob\big( \abs{ \mathsf{g}\big(\hat{w}_A^{\lasso}/\sqrt{n}\big)- \E \mathsf{g}\big(w_\ast^{\lasso}/\sqrt{n}\big) }\geq \epsilon \big)$,
		\item $\Prob\big( \abs{ \mathsf{h}\big(\hat{r}_A^{\lasso}/\sqrt{m}\big)- \E_h \mathsf{h}\big(r_\ast^{\lasso}/\sqrt{m}\big) }\geq \epsilon \big)$,
		\item $\Prob\big( \abs{ \mathsf{g}\big(\hat{v}_A^{\lasso}/\sqrt{n}\big)- \E \mathsf{g}\big(v_\ast^{\lasso}/\sqrt{n}\big) }\geq \epsilon \big)$,
		\item $\Prob\big(\abs{\hat{s}_A^{\lasso}-s_\ast^{\lasso}}\geq \epsilon^{1/2}\big)$
	\end{itemize}
	are bounded by $K\cdot(1\vee \epsilon^{-6}) \cdot n^{-1/6}\log^3 n$.
\end{theorem}
The proofs of the results in Theorem \ref{thm:lasso_dist} are  fairly involved, even given Theorem \ref{thm:lasso_universality_generic}, Proposition \ref{prop:lasso_risk} and the results in \cite{miolane2021distribution}---one needs to pay special attention to (suitable versions of) $\ell_\infty$ constrained `Gordon problems' over exception sets. We also note that the result for $\hat{s}_A^{\lasso}$ does not follow directly from $\hat{w}_A^{\lasso}$---in fact, similar to \cite{miolane2021distribution}, the distributional characterization for $\hat{w}_A^{\lasso}$ only provides a lower bound for $\hat{s}_A^{\lasso}$, while a matching upper bound is provided by the control of the subgradient $\hat{v}_A^{\lasso}$.

To put Theorem \ref{thm:lasso_dist} in the literature, \cite{panahi2017universal} obtained universality for $\hat{w}_A^{\lasso}$ in a quite restrictive sense under several strong conditions on the design distributions (details see Remark \ref{rmk:ridge_comparison}). \cite{miolane2021distribution} obtained distributional characterizations in the isotropic Gaussian design and Gaussian error case; our results here extend those of \cite{miolane2021distribution} to general designs and errors. 

As an immediate application of Theorem \ref{thm:lasso_dist}, we may use the observable quantities $\pnorm{\hat{r}_A^{\lasso}}{}, \pnorm{\hat{v}_A^{\lasso}}{},\hat{s}_A^{\lasso}$ to form consistent estimators for the estimation error $\pnorm{\hat{w}_A^{\lasso}}{}^2/n$, the prediction error $\pnorm{A\hat{w}_A^{\lasso}}{}^2/m$, the original noise level $\sigma$ and the effective noise level $\gamma_\ast^{\lasso}$ under general designs $A$ and errors $\xi$. For instance, we may use
\begin{align}\label{def:gamma_est_lasso}
\hat{\gamma}_A^{\lasso}\equiv \frac{\pnorm{\hat{r}_A^{\lasso}}{}/\sqrt{m}}{1-\frac{1}{m/n}\hat{s}_A^{\lasso}} = \frac{  \sqrt{m}\pnorm{Y-A\hat{\mu}_A^{\lasso}}{} }{m-\pnorm{\hat{\mu}_A^{\lasso}}{0}}
\end{align}
as a consistent estimator for $\gamma_\ast^{\lasso}$. See \cite[Section 4.1]{miolane2021distribution} for precise formulae of estimators for other quantities mentioned above.

As another important outlet of the proofs of Theorem \ref{thm:lasso_dist}, we consider the distribution of the degrees-of-freedom (dof) adjusted debiased Lasso $\hat{\mu}_A^{\delasso}$ (cf. \cite{zhang2014confidence,van2014asymptotically,javanmard2014confidence,javanmard2014hypothesis,javanmard2018debiasing,bellec2021debias,bellec2022biasing}), defined by
\begin{align}\label{def:debias_lasso}
\hat{\mu}_A^{\delasso}\equiv \hat{\mu}_A^{\lasso}+ \frac{A^\top (Y-A\hat{\mu}_A^{\lasso})}{1-\pnorm{\hat{\mu}_A^{\lasso}}{0}/m},
\end{align}
and the validity of the following $(1-\alpha)$ confidence intervals for $\{\mu_{0,j}\}$:
\begin{align}\label{def:debias_lasso_CI}
\mathsf{CI}^{\delasso}_j\equiv \big[\hat{\mu}_{A,j}^{\delasso}-z_{\alpha/2}\cdot \hat{\gamma}_A^{\lasso}, \hat{\mu}_{A,j}^{\delasso}+z_{\alpha/2}\cdot \hat{\gamma}_A^{\lasso}\big],\quad j \in [n].
\end{align}
Here  $z_\alpha$ is the normal upper $\alpha$-quantile defined via $\Prob(\mathcal{N}(0,1)>z_\alpha)=\alpha$.
\begin{theorem}\label{thm:debiased_lasso}
	Assume the same conditions as in Theorem \ref{thm:lasso_universality_generic}. Then there exists some $K=K(\sigma,\lambda,\tau,M_2)>0$ such that for any  $\mathsf{g}:\R^2\to \R$ and $\epsilon \in (0,1)$,
	\begin{align*}
	&\Prob\Big(\bigabs{ \E^\circ\mathsf{g}\big(\Pi_{\hat{\mu}_A^{\delasso}},\Pi_{\mu_0}\big)-\E \mathsf{g}\big(\Pi_{\mu_0}+\gamma_\ast^{\lasso} Z,\Pi_{\mu_0}\big) } \geq ( \pnorm{\mathsf{g}}{\lip}\vee\pnorm{\mathsf{g}}{\infty})\cdot\epsilon \Big)\\
	&\leq K \cdot \epsilon^{-12} \cdot n^{-1/6}\log^3 n.
	\end{align*}
	Here we write $\E^\circ[\cdot]=\E[\cdot|A,\xi]$, and $\big(\Pi_{\hat{\mu}_A^{\delasso}},\Pi_{\mu_0}\big)= n^{-1}\sum_{j=1}^n \delta_{(\hat{\mu}_{A,j}^{\delasso},\mu_{0,j})}$, $\Pi_{\mu_0}\otimes Z=(n^{-1}\sum_{j=1}^n \delta_{\mu_{0,j}})\otimes\mathcal{N}(0,1)$. Consequently, with the averaged empirical coverage for $\{\mathsf{CI}^{\delasso}_j\}$  defined as $\hat{\mathscr{C}}_A^{\delasso}\equiv n^{-1}\sum_{j=1}^n \bm{1}\big(\mu_{0,j} \in \mathsf{CI}^{\delasso}_j\big)$, for any $\epsilon \in (0,1)$,
	\begin{align*}
	\Prob\big(\abs{\hat{\mathscr{C}}_A^{\delasso}-(1-\alpha)}>\epsilon\big)\leq K\cdot \epsilon^{-24}\cdot n^{-1/6}\log^3 n.
	\end{align*}
\end{theorem}

Note that the above theorem does not directly follow from Theorem \ref{thm:lasso_dist} due to the lack of the joint distributional characterizations for $(\hat{w}_A^{\lasso},\hat{r}_A^{\lasso})$. Inspired by \cite{miolane2021distribution}, this technical issue is overcome by establishing distributional characterizations of $(\hat{w}_A^{\lasso},\hat{r}_A^{\lasso})$ in Wasserstein-2 distance that provide couplings to relate the joint distribution of $(\hat{w}_A^{\lasso},\hat{r}_A^{\lasso})$; see Proposition \ref{prop:lasso_w2} for details.

To put Theorem \ref{thm:debiased_lasso} in the literature, for the dof adjusted debiased Lasso (\ref{def:debias_lasso}), \cite{javanmard2014hypothesis} obtained 
an asymptotic version and \cite{miolane2021distribution} obtained an improved non-asymptotic version, of the above theorem in the isotropic Gaussian design and Gaussian error case. Distributional characterizations for dof adjusted debiased Lasso under general correlated Gaussian designs and Gaussian errors are obtained in \cite{celentano2020lasso,bellec2021debias,bellec2022biasing}. These works rely crucially on the Gaussianity of the design via either the CGMT (cf. \cite{miolane2021distribution,celentano2020lasso}) or Gaussian integration by parts techniques (cf. \cite{bellec2021debias,bellec2022biasing}). To the best of our knowledge, Theorem \ref{thm:debiased_lasso} provides the first theoretical justification for the dof adjusted debiased Lasso beyond Gaussian designs.

A limitation of the coverage guarantee for $\{\mathsf{CI}^{\delasso}_j\}$ in Theorem \ref{thm:debiased_lasso} above is its average nature. In the (general correlated) Gaussian design and Gaussian error case, \cite[Theorem 3.10]{bellec2021debias} obtained stronger coverage guarantees for $\{\mathsf{CI}^{\delasso}_j\}$  that hold for individual coordinates; see also the discussion after \cite[Theorem 4.1]{bellec2022observable}. Whether such stronger guarantees also hold for general designs and errors remains an interesting open question.

\subsection{Example III: Regularized robust regression}\label{section:ex_robust}
In this section, we consider universality properties for robust regression estimators \cite{huber1964robust,huber1973robust}. Let the robust cost function be
\begin{align*}
\bar{H}^{\rob}(w,A,\xi) \equiv \sum_{i=1}^m \psi_0\big((Aw)_i-\xi_i\big) + \frac{\lambda}{2}\pnorm{w + \mu_0}{}^2,
\end{align*}
and its normalized version $H^{\rob} \equiv \bar{H}^{\rob}/m$. The robust regression solution is given by $\hat{\mu}^{\rob}_A = \hat{w}^{\rob}_A + \mu_0$ with
\begin{align*}
\hat{w}^{\rob}_A \equiv \argmin_{w\in\R^n} \bar{H}^{\rob}(w,A,\xi).
\end{align*}
Instead of the conditions (R1)-(R3), we work with the following alternative set of conditions:
\begin{enumerate}
	\item[(M1)] $\tau\leq m/n\leq 1/\tau$ holds for some $\tau \in (0,1)$, and $\lambda>0$. 
	\item[(M2)] $\psi_0: \R \to \R$ is convex with weak derivative $\psi_0'$ satisfying $\abs{\psi_0(0)}\vee \esssup\, \abs{\psi_0'} \leq L_0$ for some $L_0> 0$.
	\item[(M3)] $A_0 = \sqrt{m}A$ and $\xi_0=\xi$ are independent. The entries of $A_0$ are independent, mean $0$, variance $1$ with $M_{6+\delta;A}\equiv \max_{i \in [m], j\in [n]} \E \abs{A_{0;ij}}^{6+\delta}<\infty$ for some $\delta \in (0,1)$. The entries of $\xi_0$ are independent.
\end{enumerate}
Note that under the above assumption on $\psi_0(\cdot)$, the Ridge penalty guarantees the existence and uniqueness of $\hat{w}^{\rob}_A$. 

The following theorem establishes the generic universality of $\hat{w}_A^{\rob}$ with respect to the design matrix $A$. All proofs in this section can be found in Section \ref{section:proof_robust}. 

\begin{theorem}\label{thm:robust_universality_generic}
	Suppose (M1)-(M3) hold. Fix $\mathcal{S}_n \subset \R^n$. Suppose there exist $z \in \R, \rho_0>0$ and $\epsilon_n \in [0,1/4)$ such that 
	\begin{align*}
	\Prob\bigg(\min_{w \in \R^n} H^{\rob}(w,G,\xi)\geq z+\rho_0 \bigg)\vee 	\Prob\bigg(\min_{w \in \mathcal{S}_n} H^{\rob}(w,G,\xi)\leq z+2\rho_0 \bigg) \leq \epsilon_n. 
	\end{align*}
	Then there exists some $K=K(\lambda,\tau,M_{6+\delta;A},\delta,L_0)>0$ such that
	\begin{align*}
	&\Prob\big(\hat{w}_A^{\rob} \in \mathcal{S}_n\big)\leq 4\epsilon_n+K \big(1+\pnorm{\mu_0}{\infty}^{6+\delta}+ \rho_0^{-3})\cdot  n^{-(1\wedge \delta)/500}.
	\end{align*}
\end{theorem}

A significant feature of Theorem \ref{thm:robust_universality_generic} above is that no apriori moment conditions on the error vector $\xi$ are required. This is particularly appealing from the perspective of robust regression \cite{huber1964robust,huber1973robust}. 

The next proposition establishes an element-wise bound for $\hat{w}_A^{\rob}$ that serves as the key to the proof of Theorem \ref{thm:robust_universality_generic}.

\begin{proposition}\label{prop:robust_risk}
	Suppose (M1)-(M3) hold. Then for any $p\geq 2$, there exists some $K=K(p)>0$ such that
	\begin{align*}
	\max_{j\in[n]} \E \abs{\hat{w}_{A,j}^{\rob}}^p\leq K\cdot \big\{(L_0/\lambda)^p M_{p;A}+ \pnorm{\mu_0}{\infty}^p\big\}. 
	\end{align*}
	Here $M_{p;A}\equiv \max_{i \in [m], j\in[n]} \E \abs{A_{0;ij}}^p$.	
\end{proposition}

Below we give a quick demonstration of the power of the above results, by establishing risk universality for $\hat{\mu}^{\rob}_A$ with the help of essentially existing Gaussian design results in \cite{thrampoulidis2018precise} proved via the CGMT method.

\begin{theorem}\label{thm:robust_risk_asymp}
Suppose the following hold.
\begin{enumerate}
	\item $m/n\to\tau_0 \in (0,\infty)$ and $\lambda>0$ is fixed.
	\item $\psi_0$ satisfies (M2) and either (i) $\psi_0$ is not differentiable at certain point, or (ii) $\psi_0$ contains an interval on which $\psi_0$ is differentiable with strictly increasing derivative.
	\item The entries of $A_0$ are independent, mean $0$, variance $1$ with $\sup_n\max_{i\in [m], j\in[n]}\E \abs{A_{0;ij}}^{6+\delta}<\infty$ for some $\delta \in (0,1)$. 
	\item $\xi=(\xi_i)$ contains i.i.d. components with a continuous Lebesgue density. 
	\item $\mu_0$ contains i.i.d. components whose law $\Pi_0$ possesses moment of any order. 
\end{enumerate}
Then with $Z\sim \mathcal{N}(0,1)$, the system of equations
\begin{align}\label{eqn:robust_fpn}
(\gamma_\ast^{\rob})^2 /\tau_0 &=  \E \Big[\gamma_\ast^{\rob} Z+\xi_1- \prox_{\psi_0}\Big(\gamma_\ast^{\rob} Z+\xi_1;\beta_\ast^{\rob}\Big)\Big]^2 +\lambda^2 (\beta_\ast^{\rob})^2 \cdot \E \Pi_0^2,\nonumber\\
1-\tau_0^{-1}+ \lambda \beta_\ast^{\rob}&=\E \prox_{\psi_0}'\Big({\gamma}_\ast^{\rob} Z+\xi_1;{\beta}_\ast^{\rob}\Big)
\end{align}
admits a unique non-trivial solution $(\beta_\ast^{\rob},\gamma_\ast^{\rob}) \in (0,\infty)^2$ such that
\begin{align*}
 \frac{\pnorm{\hat{\mu}_A^{\rob}-\mu_0}{}^2}{n} \to \tau_0(\gamma_\ast^{\rob})^2\quad \hbox{in probability}.
\end{align*}
\end{theorem}

In the second equation of (\ref{eqn:robust_fpn}), $\prox_{\psi_0}'(x;\tau)=(\d/\d x) \prox_{\psi_0}(x;\tau)$ is interpreted as the weak derivative thanks to the $1$-Lipschitz property of the proximal map $x\mapsto \prox_{\psi_0}(x;\tau)$ for any $\tau>0$ (cf. Lemma \ref{lem:prox_lipschitz}). 

We now compare Theorem \ref{thm:robust_risk_asymp} to the risk results in \cite{elkaroui2018impact}. The most significant advantage of Theorem \ref{thm:robust_risk_asymp} rests in its much weaker condition on the loss function $\psi_0$. In particular, \cite{elkaroui2018impact} requires strong regularity assumptions on $\psi_0$ (e.g., $\psi_0''$ is required to be Lipschitz), which exclude the two canonical examples in Example \ref{ex:robust_loss} in robust regression that are covered by our theory. In addition, the $6+\delta$ moment assumption on the design matrix is also much weaker in Theorem \ref{thm:robust_risk_asymp} compared to the exponential moments required in \cite{elkaroui2018impact}. 

An interesting question is whether the universality results in Theorem \ref{thm:robust_risk_asymp} hold for the unregularized case $\lambda=0$ when $\tau_0>1$. The only result in this direction appears to be \cite[Section 6]{elkaroui2013asymptotic} (some heuristics are presented in \cite{elkaroui2013robust}), where strong convexity of $\psi_0$ and local Lipschitzness of $\psi_0''$ are required; see also related results in \cite{donoho2016high} under Gaussian designs. Under these strong assumptions on $\psi_0$, it is possible to establish $\ell_\infty$ bounds for $\hat{w}_A^{\rob}$ using similar techniques as in Proposition \ref{prop:robust_risk}, and therefore the risk universality in Theorem \ref{thm:robust_risk_asymp}. It however remains an open question to establish (risk) universality under the weak conditions on $\psi_0$ as in Theorem \ref{thm:robust_risk_asymp}.

\subsection{Non-universality for general isotropic designs}\label{section:non_universality}
Consider the regression model (\ref{eqn:reg_model}) with $m/n>1$ and the ordinary least squares estimator (LSE):
\begin{align*}
\hat{\mu}^{\lse}_A&\equiv \argmin_{\mu \in \R^n} \pnorm{Y-A\mu}{}^2 = (A^\top A)^{-1} A^\top Y.
\end{align*}
Suppose that the error vector satisfies $\xi \sim \mathcal{N}(0,I_m)$ for simplicity. It is easy to prove risk universality of $\hat{\mu}^{\lse}_A$ for design matrices $A$ with independent entries satisfying Assumption \ref{assump:random}, by either using results in this paper or directly resorting to random matrix theory. However, independence across the entries of $A$ cannot be relaxed to the weaker row independent isotropic setting, as the following proposition shows.

\begin{proposition}\label{prop:non_universality_linear_reg}
Fix $m, n \in \N$ with $m\geq 2$, $m>n$, and $L_n> 1$. There exists some centered random vector $b_0 \in \R^n$ with $\E b_0^{\otimes 2} = \E (\mathcal{N}(0,I_n))^{\otimes 2}$ such that the following hold: With $B_0 \in \R^{m\times n}$ denoting a random matrix whose rows are i.i.d. as $b_0$, and $B\equiv B_0/\sqrt{m}$, we have $
n^{-1}\E \pnorm{\hat{\mu}^{\lse}_B-\mu_0}{}^2 \geq L_n\cdot n^{-1}\E \pnorm{\hat{\mu}^{\lse}_G-\mu_0}{}^2$. 
\end{proposition}

\begin{proof}
	Take $L_n>1$. Let $U$ be a discrete distribution supported on $\{\pm L_n^{-1}, \pm S_n \}$ with $2\Prob(U=\pm L_n^{-1})=1-1/m$ and $2\Prob(U=\pm S_n)=1/m$. Here $S_n>0$ is determined by the condition $\E U^2 = 1$; in fact some simple calculation shows that $S_n = \big\{m\big(1-L_n^{-2}(1-1/m)\big)\big\}^{1/2}$. Now let $b_0\equiv U\cdot Z$, where $Z\sim \mathcal{N}(0,I_n)$ is independent of $U$. Then by construction $\E b_0^{\otimes 2} = \E (\mathcal{N}(0,I_n))^{\otimes 2}$. With $U_i$'s and $Z_i$'s being independent copies of $U$ and $Z$, we have
	\begin{align*}
	&n^{-1}\E \pnorm{\hat{\mu}^{\lse}_B-\mu_0}{}^2 = n^{-1} \E \tr\big((B^\top B)^{-1}\big) \geq n^{-1} \E \tr\bigg[\bigg(\frac{1}{m}\sum_{i=1}^m U_i^2 Z_iZ_i^\top \bigg)^{-1}\bigg] \bm{1}_{U_i^2 = L_n^{-2},\forall i \in [m]}\\
	&\geq \bigg(1-\frac{1}{m}\bigg)^m\cdot  L_n^2\cdot \E \tr\bigg[\bigg(\frac{1}{m}\sum_{i=1}^m Z_iZ_i^\top \bigg)^{-1}\bigg] \geq  e^{-2} L_n^2\cdot n^{-1}\E \pnorm{\hat{\mu}^{\lse}_G-\mu_0}{}^2,
	\end{align*}
	where the last inequality used the simple fact that $(1-1/m)^m\geq e^{-2}$ for $m\geq 2$. Now the claim follows by adjusting constants.
\end{proof}

From the proof above, it is clear that the failure of risk universality is due to the non-universality of the spectrum of the sample covariance $\hat{\Sigma}_a\equiv m^{-1}\sum_{i=1}^m a_i a_i^\top \in \R^{n\times n}$, where $\{a_i\}_{i=1}^m $ are the rows of the design matrix $A$ being i.i.d. isotropic $\R^n$-valued random vectors. In a recent work \cite{yaskov2016necessary}, it is shown that the empirical spectral distribution (ESD) of $\hat{\Sigma}_a$ converges to the Marchenko-Pastur law if and only if 
\begin{align}\label{eqn:MP_law_1}
\frac{1}{n}\Big(a_1^\top(\hat{\Sigma}_a+\epsilon I)^{-1}a_1-\tr\{(\hat{\Sigma}_a+\epsilon I)^{-1}\}\Big)\stackrel{\mathrm{p}}{\to} 0,\quad \forall \epsilon>0.
\end{align}
This characterizing condition (\ref{eqn:MP_law_1}) is satisfied if a weak law of large numbers holds for quadratic forms in the following form:
\begin{align}\label{eqn:MP_law_2}
\frac{1}{n}\Big(a_1^\top D_n a_1-\tr D_n\Big)\stackrel{\mathrm{p}}{\to} 0,\quad \forall\, \Big\{\hbox{real symmetric p.s.d. $D_n$ with }\pnorm{D_n}{\op}\leq 1\Big\}.
\end{align}
The condition (\ref{eqn:MP_law_2}) is strictly stronger than (\ref{eqn:MP_law_1}), cf. \cite{adamczak2011marchenko}. (\ref{eqn:MP_law_2}) or its variants are known in the literature as a sufficient condition for the ESD of the sample covariance to converge to the Marchenko-Pastur law; see e.g. \cite[Theorem 1.1]{bai2008large} for an $L_2$ version of (\ref{eqn:MP_law_2}) and also \cite{elkaroui2009concentration,adamczak2013remarks} for related results under variants of (\ref{eqn:MP_law_2}). An interesting question arising from the above discussion is whether (risk) universality properties of general regularized regression estimators hold under design matrices $A$ with i.i.d. rows satisfying the condition (\ref{eqn:MP_law_1}) or the slightly stronger condition (\ref{eqn:MP_law_2}). We leave this to a future work.

\subsection{Some illustrative simulation}\label{section:simulation}

We perform a small scale simulation here to illustrate the (non-)universality results proved in the previous subsections. 

First we examine (non-)universality of risk asymptotics under different distributions of $(A,\xi)$. As can be seen from Figure \ref{fig:risk}, for both Ridge and Lasso estimators, universality of risk asymptotics holds for $(\sqrt{m}A,\xi)$ with i.i.d. entries from a $t$ distribution with only $3.5$ degrees of freedom (dof), and then gradually breaks down when the dof approaches $2.5$. It also seems reasonable to conjecture that a phase transition near $t(3)$ occurs for the risk universality for both Ridge and Lasso estimators. On the other hand, under the setup of (non-universality Section \ref{section:non_universality}) with a simple three-point delta prior, we see a matching second moment of the design does not guarantee universality.

\begin{figure}
	\centering
	\begin{subfigure}[b]{0.48\textwidth}
		\centering
		\includegraphics[width=\textwidth]{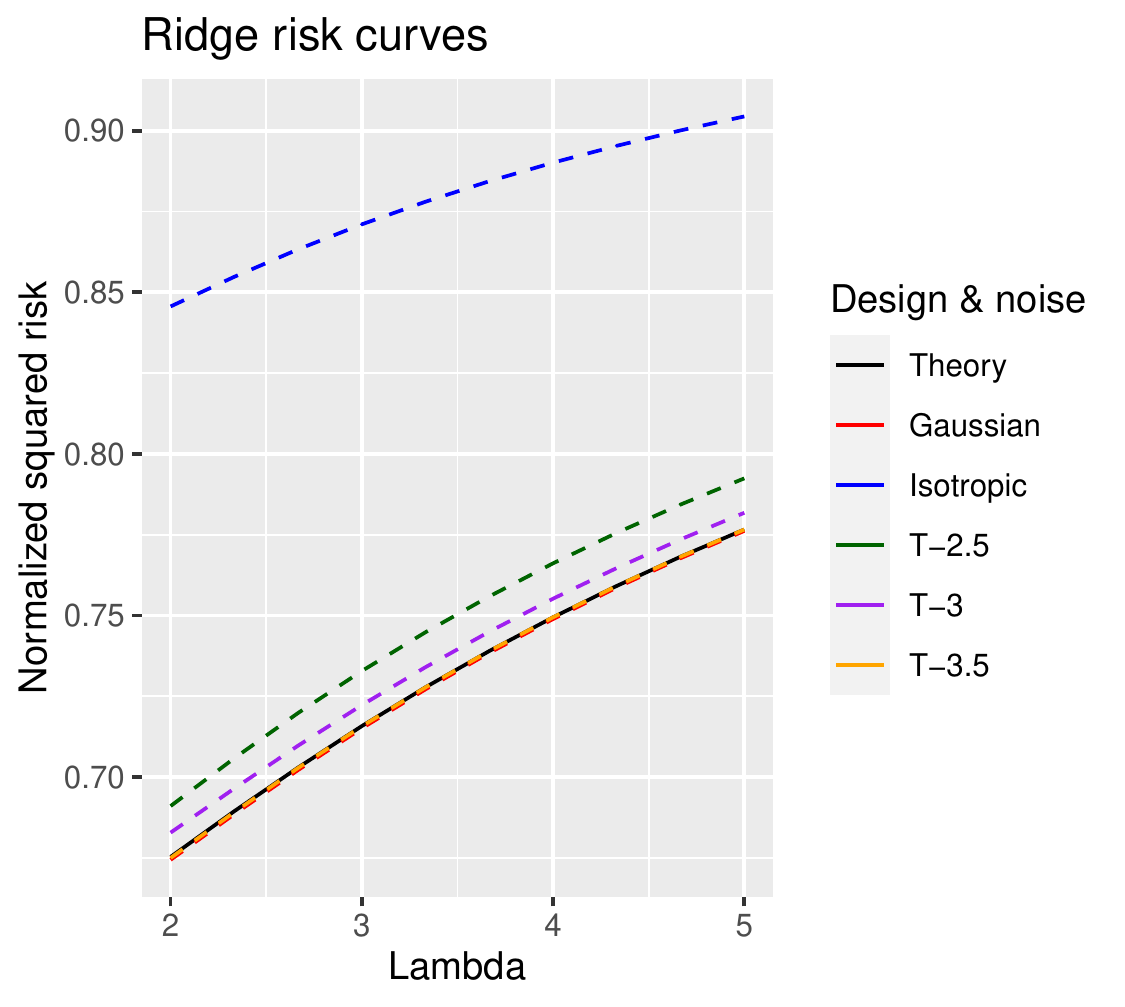}
		\caption{Ridge risk curves}
		\label{fig:y equals x}
	\end{subfigure}
	\hfill
	\begin{subfigure}[b]{0.48\textwidth}
		\centering
		\includegraphics[width=\textwidth]{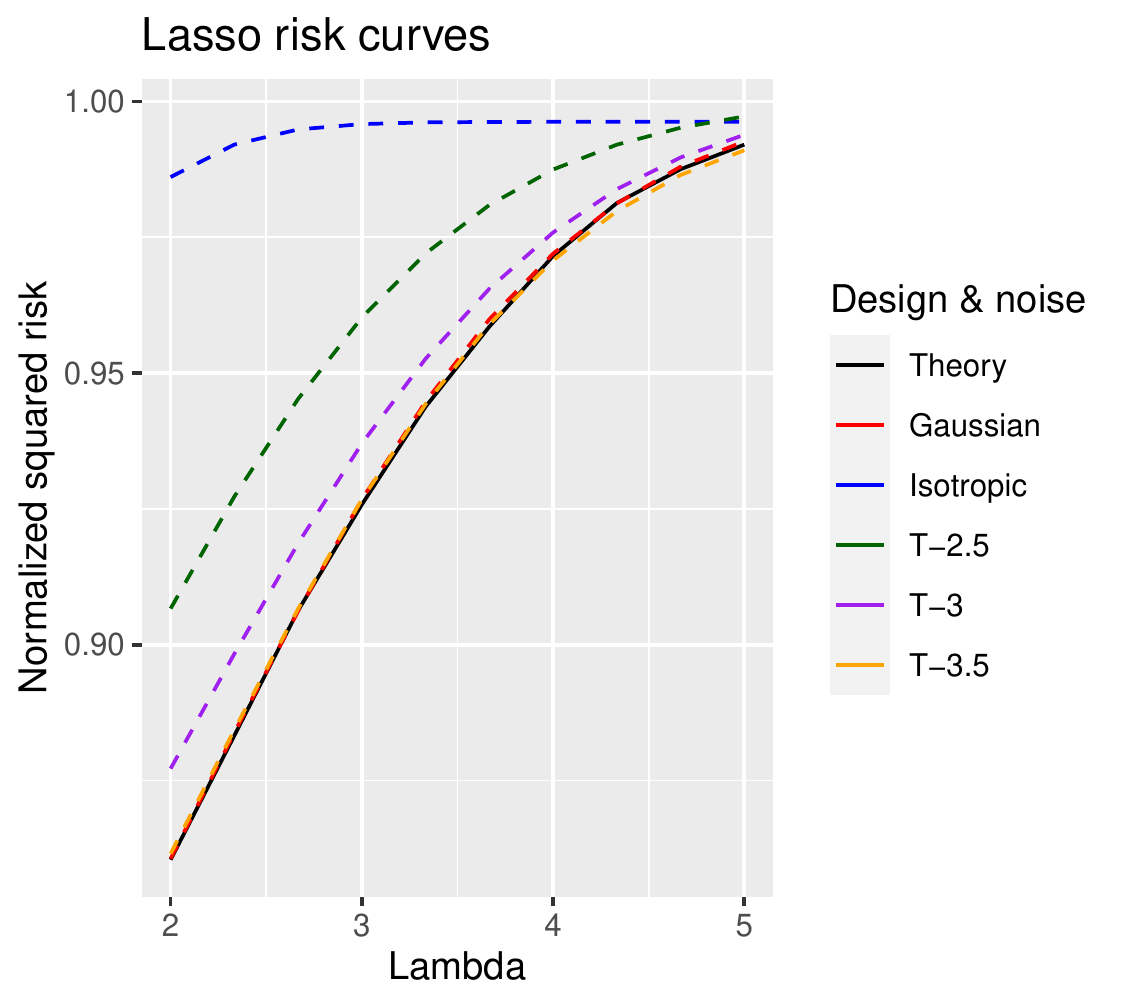}
		\caption{Lasso risk curves}
		\label{fig:three sin x}
	\end{subfigure}
	\caption{The black solid line marks the theoretical risk given by $w^{\mathsf{L}}_*$ and $w^{\mathsf{R}}_*$. Simulation parameters: $m = 1200$, $n = 1500$, $\mu_0\in\R^n$ are i.i.d. $\mathcal{N}(0,1)$, $(\sqrt{m}A, \xi)$ have i.i.d. entries following $\mathcal{N}(0,1)$ (red), $t$ distribution with df 3.5 (yellow), df 3 (purple), df $2.5$ (green), with proper normalization in the latter two cases so that the variance is $1$. In the isotropic case (blue), $\xi$ has i.i.d. $\mathcal{N}(0,1)$ entries and $(\sqrt{m}A)_{i,\cdot} = Z_iU_i$ with $Z_i \stackrel{\textrm{i.i.d.}}{\sim} \mathcal{N}(0,I_n)$ and $U_i \stackrel{\textrm{i.i.d.}}{\sim} 0.25\delta_{\sqrt{2}} + 0.25\delta_{-\sqrt{2}} + 0.5\delta_{0}$. The empirical risk curves are averaged over $50$ replications.}
	\label{fig:risk}
\end{figure}

Next we examine the distributional universality proved for Ridge and Lasso estimators in Theorems \ref{thm:ridge_dist} and \ref{thm:lasso_dist}. By the QQ plots in Figure \ref{fig:dist}, we see that such closeness holds all the way down to the very heavy-tailed situation where $(\sqrt{m}A, \xi)$ have i.i.d. entries following a $t$ distribution with only about $2$ dof. Here the simulation setup is similar to that used in Figure \ref{fig:risk} with the exception that the variance level of $\mu_0$ is enlarged to ensure a visible difference in the QQ plots. 

\begin{figure}
	\centering
	\begin{subfigure}[b]{0.48\textwidth}
		\centering
		\includegraphics[width=\textwidth]{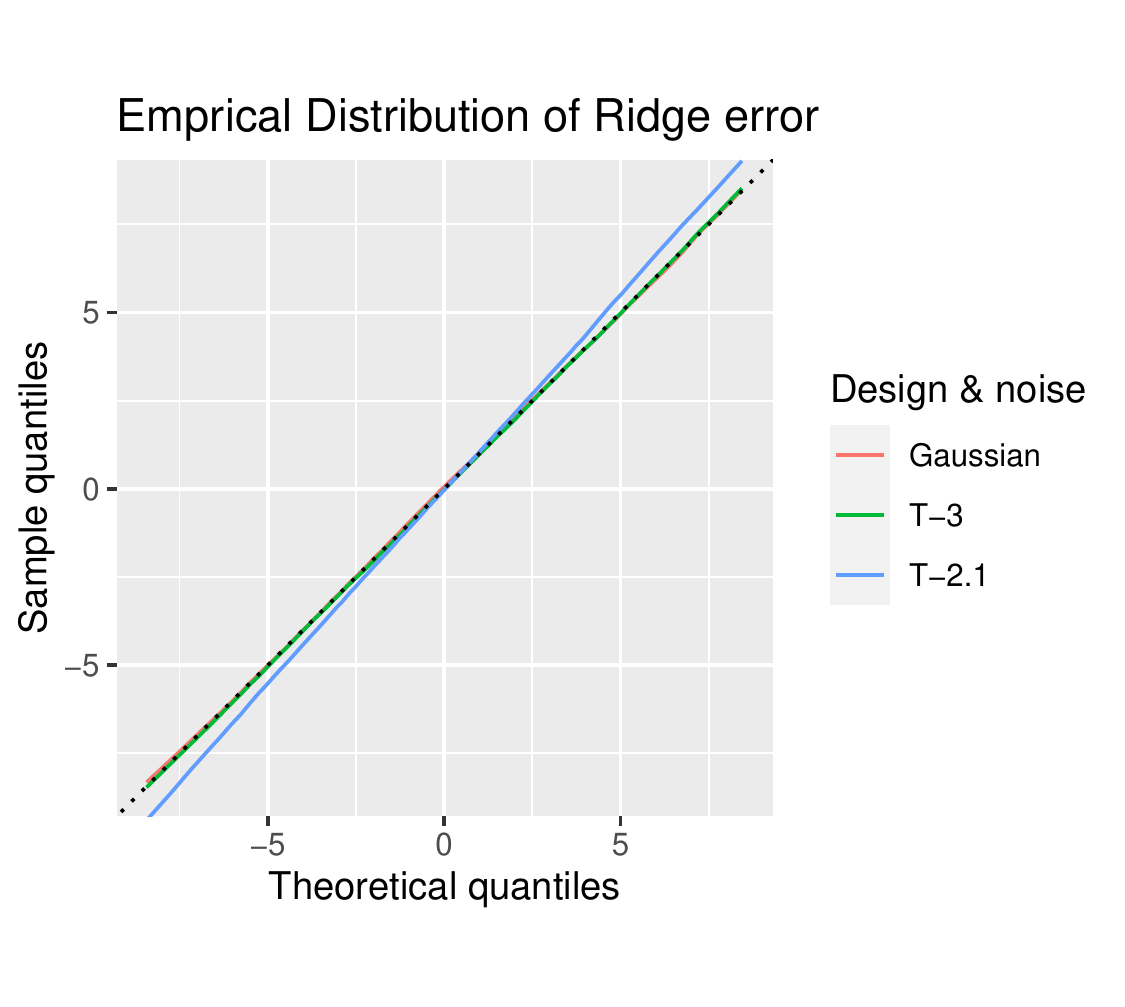}
		\caption{QQ plot of Ridge errors}
	\end{subfigure}
	\hfill
	\begin{subfigure}[b]{0.48\textwidth}
		\centering
		\includegraphics[width=\textwidth]{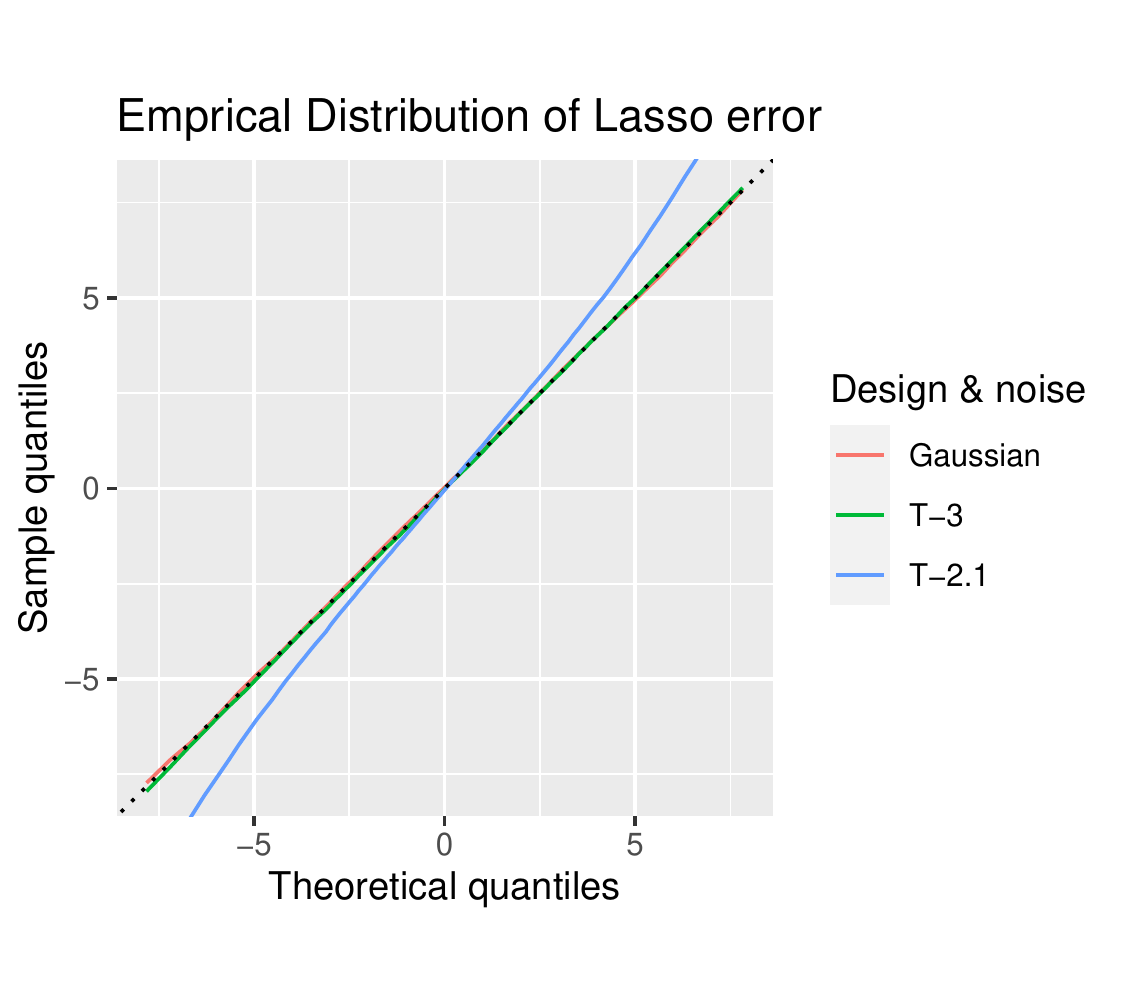}
		\caption{QQ plot of Lasso errors}
	\end{subfigure}
	\caption{Comparison of the empirical quantiles of the error with theoretical quantiles given by $w^{\mathsf{L}}_*$ and $w^{\mathsf{R}}_*$. Simulation parameters: $m = 1200$, $n = 1500$, $\mu_0\in\R^n$ are i.i.d. $\mathcal{N}(0,5^2)$, $(\sqrt{m}A, \xi)$ have i.i.d. entries following $\mathcal{N}(0,1)$ (red), $t$ distribution with df $3$ (green), $t$ distribution with df $2.1$ (blue), with proper normalization in the latter two cases so that the variance is $1$. The empirical quantiles are averaged over $50$ replications.}
	\label{fig:dist}
\end{figure}

\section{Proofs for Section \ref{section:theory}}\label{section:proof_theory}

\subsection{Notation on Hamiltonian}
Let $W$ be a finite set. For any Hamiltonian system $\{H_0(w,A): w \in W\}$ (indexed by any `other structure' $A$, which is typically a matrix in this paper), let $\ang{\cdot}{H_0}$ be the expectation under the Gibbs measure over $W$ induced by the Hamiltonian $H_0$, i.e., for any function $f(\cdot,A): W \to \R$, we have
\begin{align}\label{def:hamilton}
\ang{f}{H_0}\equiv \ang{f(\cdot,A)}{H_0}\equiv \frac{\sum_{w \in W}  f(w,A) e^{-H_0(w,A)}  }{  \sum_{w \in W} e^{-H_0(w,A)}}. 
\end{align}
It is easy to verify that for a generic differentiation operator $\partial$,
\begin{align}\label{eqn:hamilton_derivative}
\partial \ang{f}{H_0}=\ang{\partial f}{H_0}-\ang{f\cdot \partial H_0}{H_0}+\ang{f}{H_0}\ang{\partial H_0}{H_0}.
\end{align}
The above formula will often be used with $\partial\equiv \partial/\partial A_{ij}$. 

\subsection{Proof of Theorem \ref{thm:universality_smooth}}\label{section:proof_universality_smooth}
Let for $\beta>0$ and $N \in \N$, define the `soft-min' function
\begin{align}
F_\beta(x) \equiv F_{\beta;N}(x) \equiv -\frac{1}{\beta}\log \bigg(\sum_{j=1}^N \exp(-\beta x_j)\bigg),\quad x \in \R^N. 
\end{align}

Recall $H_\psi(w,A)$ defined in (\ref{def:H_psi}) for $\{\psi_i\}$ and its un-normalized version $\bar{H}_\psi = m\cdot H_{\psi}$. We will write $H_{\psi_\rho}(w,A)$ when $\{\psi_i\}$'s are replaced with their smoothed versions $\{\psi_{i;\rho}\}$'s (cf. Assumption \ref{assump:loss}) in (\ref{def:H_psi}) and $\bar{H}_{\psi_\rho} = m\cdot H_{\psi_\rho}$.

\begin{proposition}\label{prop:comparision_smooth}
	For any finite set $W$ and $\mathsf{g} \in C^1(\R)$, 
	\begin{align*}
	\biggabs{\E \mathsf{g}\Big(\min_{w \in W } H_{\psi_\rho}(w,A)\Big) - \E \mathsf{g}\Big[m^{-1}F_\beta\Big( \big(\bar{H}_{\psi_\rho}(w,A)\big)_{w \in W }   \Big)\Big]}\leq \frac{\pnorm{\mathsf{g}'}{\infty}\cdot \log \abs{W} }{m\beta}.
	\end{align*}
\end{proposition}
\begin{proof}
	The left hand side of the claimed inequality is bounded by
	\begin{align*}
	&\pnorm{\mathsf{g}'}{\infty} \cdot m^{-1} \E \biggabs{\min_{w \in W} \bar{H}_{\psi_\rho}(w,A) - F_\beta\Big( \big(\bar{H}_{\psi_\rho}(w,A)\big)_{w \in W }   \Big)  }\leq \frac{\pnorm{\mathsf{g}'}{\infty}\cdot \log \abs{W} }{m\beta}.
	\end{align*}
	Here in the last inequality we used the easy fact that $\min_j x_j- \log \abs{W}/\beta\leq F_{\beta;\abs{W}}(x)\leq \min_j x_j$ for all $x \in \R^{\abs{W}}$. 
\end{proof}

\begin{proposition}\label{prop:comparision_discrete}
	Suppose Assumption \ref{assump:loss} holds. Fix $L_n\geq 1$. There exists some $C_0=C_0(q)>0$, such that for any measurable $\mathcal{S}_n \subset [-L_n,L_n]^n$ and any $\delta \in (0,1)$, there exists a deterministic finite set $\mathcal{S}_{n,\delta}$ with $\abs{\mathcal{S}_{n,\delta}}\leq \ceil{2L_n/\delta}^n$ and the following hold: For any $\rho \in (0,\bar{\rho})$ and $\mathsf{g} \in C^1(\R)$, 
	\begin{align*}
	&\biggabs{\E \mathsf{g}\Big(\min_{w \in \mathcal{S}_{n,\delta}} H_{\psi_\rho}(w,A)\Big) - \E \mathsf{g}\Big( \min_{w \in \mathcal{S}_{n}} H_{\psi_\rho}(w,A) \Big)}\\
	&\leq C_0\cdot \pnorm{\mathsf{g}'}{\infty}\bigg( \mathscr{D}_{\psi}(\rho)\delta\cdot  L_n^{q_0q_1} \Av\big(\{L_{\psi_i}^{q_1}\}\big) \cdot \big(1+ \E \pnorm{A}{\infty}^{q_0q_1+1}\big) + \mathscr{M}_{\mathsf{f}}(L_n,\delta)\bigg),
	\end{align*}
	where  $\pnorm{A}{\infty}$ is the matrix norm induced by $\ell_\infty$. 
\end{proposition}
\begin{proof}
	Fix $\delta \in (0,1)$, we may construct $\mathcal{S}_{n,\delta}$ as follows: For all closed hyper-rectangles determined by the lattice $(\delta \mathbb{Z})^n\cap [-L_n,L_n]^n$ whose intersection with $\mathcal{S}_n$ is non-empty, pick one arbitrary element. Then collect all such elements to form $\mathcal{S}_{n,\delta}$. By construction, $\mathcal{S}_{n,\delta}\subset \mathcal{S}_n$ is a $\delta$-cover of $\mathcal{S}_n$ under $\ell_\infty$ with $\abs{\mathcal{S}_{n,\delta}}\leq \ceil{2L_n/\delta}^n$. Let $w_\ast \in \argmin_{w \in \mathcal{S}_n} H_{\psi_\rho}(w,A)$ (be well-defined without loss of generality), and $w_{\ast,\delta} \in \mathcal{S}_{n,\delta}$ be such that $\pnorm{w_\ast-w_{\ast,\delta}}{\infty}\leq \delta$. Then the left hand side of the desired inequality is bounded by 
	\begin{align*}
	&\pnorm{\mathsf{g}'}{\infty}\cdot \E \bigabs{\min_{w \in \mathcal{S}_{n,\delta}} H_{\psi_\rho}(w,A)-  \min_{w \in \mathcal{S}_{n}} H_{\psi_\rho}(w,A)}\\
	& \leq \pnorm{\mathsf{g}'}{\infty}\cdot   \E \big(H_{\psi_\rho}(w_{\ast,\delta},A)-H_{\psi_\rho}(w_\ast,A)\big)_+\\
	&\lesssim \pnorm{\mathsf{g}'}{\infty}\cdot  \bigg(\frac{1}{m}\E \biggabs{ \sum_{i=1}^m \Big[\psi_{i;\rho}\big((A w_{\ast,\delta})_i\big)- \psi_{i;\rho}\big((A w_{\ast})_i\big)\Big] }+ \sup_{\substack{w,w'\in [-L_n,L_n]^n,\\ \pnorm{w-w'}{\infty}\leq \delta}} \bigabs{\mathsf{f}(w)-\mathsf{f}(w') } \bigg)\\
	&\equiv \pnorm{\mathsf{g}'}{\infty}\cdot \big(T_1+ \mathscr{M}_{\mathsf{f}}(L_n,\delta)\big).
	\end{align*}
	By the assumption on $\{\psi_{i;\rho}\}$, we have
	\begin{align*}
	&\bigabs{\psi_{i;\rho}\big((A w_{\ast,\delta})_i\big)- \psi_{i;\rho}\big((A w_{\ast})_i\big)}\\
	&\lesssim_q  {\mathscr{D}_{\psi}}(\rho) \cdot \bigabs{ \big(A(w_\ast-w_{\ast,\delta})\big)_i}\cdot L_{\psi_i}^{q_1}\Big(1+\abs{(A w_{\ast})_i}^{q_0q_1}+ \abs{(A w_{\ast,\delta})_i}^{q_0q_1}\Big).
	\end{align*}
	This implies that 
	\begin{align*}
	T_1&\lesssim_q   m^{-1} \mathscr{D}_{\psi}(\rho) \E  \pnorm{A(w_\ast-w_{\ast,\delta})}{\infty} \sum_{i=1}^m L_{\psi_i}^{q_1} \Big(1+\abs{(A w_{\ast})_i}^{q_0q_1}+ \abs{(A w_{\ast,\delta})_i}^{q_0q_1}\Big)\\
	&\lesssim    \mathscr{D}_{\psi}(\rho)\cdot \delta \cdot \E \pnorm{A}{\infty}\bigg(\pnorm{A w_\ast}{\infty}^{q_0q_1}+ \pnorm{A w_{\ast,\delta}}{\infty}^{q_0q_1}+\frac{1}{m}\sum_{i=1}^m L_{\psi_i}^{q_1}\bigg)\\
	&\lesssim   \mathscr{D}_{\psi}(\rho)\cdot \delta\cdot  \Big(L_n^{q_0q_1}+ \Av\big(\{L_{\psi_i}^{q_1}\}\big)\Big) \big(1+ \E \pnorm{A}{\infty}^{q_0q_1+1}\big).
	\end{align*}
	The claim follows as $L_n\wedge L_{\psi_i}\geq 1$. 
\end{proof}

\begin{proposition}\label{prop:dim_free_smooth_min}
	Suppose Assumptions \ref{assump:setting} and \ref{assump:loss} hold. Let $A_0,B_0 \in \R^{m\times n}$ be two random matrices with independent components, such that $\E A_{0;ij}=\E B_{0;ij}=0$ and $\E A_{0;ij}^2 = \E B_{0;ij}^2$ for all $i \in [m], j \in [n]$. Further assume that 
	\begin{align*}
	M\equiv  \max_{i\in [m],j \in[n]}\big(\E \abs{A_{0;ij}}^{\bar{\mathsf{q}}}+\E \abs{B_{0;ij}}^{\bar{\mathsf{q}}}\big)<\infty.
	\end{align*}
	Let $A\equiv A_0/\sqrt{m}$ and $B\equiv B_0/\sqrt{m}$. Then there exists some $C_0=C_0(\tau,q,M)>0$ such that for any finite set $W\subset [-L_n,L_n]^n$ with $L_n\geq 1$, any $\rho \in (0,1)$, and any $\mathsf{g} \in C^3(\R)$, we have 
	\begin{align*}
	& \biggabs{\E \mathsf{g} \Big[m^{-1} F_\beta\Big( \big( \bar{H}_{\psi_\rho}(w,A)\big)_{w \in W }   \Big)\Big]-\E \mathsf{g} \Big[m^{-1} F_\beta\Big( \big(\bar{H}_{\psi_\rho}(w,B)\big)_{w \in W}   \Big)\Big] }\\
	&\leq C_0 \cdot K_{\mathsf{g}}\cdot \mathscr{D}_{\psi}^3(\rho)\cdot  L_n^{\bar{\mathsf{q}}} \Av\big(\{L_{\psi_i}^{\mathsf{q}}\}\big) \cdot  (1\vee\beta^2)\cdot n^{-1/2}.
	\end{align*}
	Here recall $K_{\mathsf{g}}=1+\max_{\ell=1,2,3} \pnorm{\mathsf{g}^{(\ell)}}{\infty}$, and $\mathsf{q},\bar{\mathsf{q}}$ are given in (\ref{def:q_bar}).
\end{proposition}
\begin{proof}	
	We will simply write $\bar{H}_{\psi_\rho}$ as $\bar{H}$, and work with the case $\beta\geq 1$ in the proof for notational simplicity. Recall the notation $\ang{\cdot}{H_0}$ in (\ref{def:hamilton}) for a generic Hamiltonian $H_0$. Let $\bar{H}_\beta \equiv \beta\cdot \bar{H}$, and $F_\beta(A)\equiv F_\beta\big( \big(\bar{H}(w,A)\big)_{w \in W}   \big)$. For $i \in [m], j\in [n], \ell \in \N$, write $\partial_{ij}^\ell\equiv \partial^\ell/\partial A_{ij}^\ell$ and repeatedly using the derivative formula (\ref{eqn:hamilton_derivative}), we have
	\begin{align*}
	\partial_{ij} \mathsf{g}\big(m^{-1}F_\beta(A)\big)
	&= m^{-1}\bigang{\partial_{ij} \bar{H} }{\bar{H}_\beta}\cdot \mathsf{g}'\big(m^{-1}F_\beta(A)\big) ,\\
	\partial_{ij}^2 \mathsf{g}\big(m^{-1}F_\beta(A)\big)& =  m^{-1}\bigg\{\bigang{\partial_{ij}^2 \bar{H}}{\bar{H}_\beta}- \beta\bigang{ \big(\partial_{ij} \bar{H}\big)^2}{\bar{H}_\beta}+\beta\bigang{\partial_{ij} \bar{H}}{\bar{H}_\beta}^2\bigg\} \cdot \mathsf{g}'\big(m^{-1}F_\beta(A)\big)\\
	&\qquad + m^{-2}\bigang{\partial_{ij} \bar{H} }{\bar{H}_\beta}^2 \mathsf{g}''\big(m^{-1}F_\beta(A)\big) ,
	\end{align*}
	and for $\beta\geq 1$,
	\begin{align}\label{ineq:smooth_error_1}
	\bigabs{\partial_{ij}^3 \mathsf{g}\big(m^{-1}F_\beta (A)\big)}&\lesssim K_{\mathsf{g}} m^{-1}\beta^2\bigg\{\bigang{ \abs{\partial^3_{ij} \bar{H}} }{\bar{H}_\beta} + \bigang{ \abs{\partial^2_{ij} \bar{H}}\cdot \abs{\partial_{ij}\bar{H}} }{\bar{H}_\beta}\nonumber\\
	&\qquad + \bigang{ \abs{\partial^2_{ij} \bar{H}} }{\bar{H}_\beta} \bigang{ \abs{\partial_{ij} \bar{H}} }{\bar{H}_\beta}+\bigang{\abs{\partial_{ij} \bar{H}}^3 }{\bar{H}_\beta}\bigg\}.
	\end{align}
	Let 
	\begin{align}\label{ineq:smooth_error_2}
	\bar{H}^{(i)}(w,A)\equiv \psi_{i;\rho} \big((Aw)_i\big)\geq 0,
	\end{align}
	so $\bar{H}(w,A) = \sum_{i=1}^m  \bar{H}^{(i)}(w,A)+m\cdot\mathsf{f}(w)$, and $\partial_{ij}^\ell \bar{H}(w,A)= w_j^\ell\cdot \partial^\ell \psi_{i;\rho}\big((Aw)_i\big)$ for $\ell=1,2,3$. Combined with (\ref{ineq:smooth_error_1}) and the assumption on $\psi_\rho$, we have 
	\begin{align}\label{ineq:smooth_error_3}
	&\bigabs{\partial_{ij}^3 \mathsf{g}\big(m^{-1}F_\beta (A)\big)} \lesssim K_{\mathsf{g}}m^{-1} \beta^2\cdot L_n^3 \mathscr{D}_{\psi}^3(\rho)  \Big[1+ \bigang{ \bar{H}^{(i)}(\cdot,A)^{\mathsf{q}} }{\bar{H}_\beta}\Big],
	\end{align}
	where recall $\mathsf{q}= \max\{q_3,q_1+q_2, 3q_1\}$. To apply Lindeberg's principle, consider the Lindeberg path $C_{ij}(t)$ between two random matrices $A,B$, defined by setting all elements in $C_{ij}(t)$ before (resp. after) the position $(i,j)$ as those of $A$ (resp. $B$) and $(C_{ij}(t))_{ij}=t$. Now we shall provide a bound for $\E  \bigang{\bar{H}^{(i)}(\cdot,C_{ij}(t))^{\mathsf{q}} }{\bar{H}_\beta}$. With
	\begin{align*}
	C\equiv C_{ij}(t),\quad \bar{H}_{\beta;-i}(w,C)\equiv \beta\big(\bar{H}(w,C)-\bar{H}^{(i)}(w,C)\big),
	\end{align*}
	we may continue as
	\begin{align}\label{ineq:smooth_error_4}
	\E  \bigang{\bar{H}^{(i)}(\cdot,C)^{\mathsf{q}} }{\bar{H}_\beta}& = \E \bigang{e^{-\beta \bar{H}^{(i)}(\cdot,C) }\cdot \bar{H}^{(i)}(\cdot,C)^{\mathsf{q}} }{\bar{H}_{\beta;-i}} \big/\bigang{e^{-\beta \bar{H}^{(i)}(\cdot,C) } }{\bar{H}_{\beta;-i}}   \nonumber\\
	&\leq \E \bigang{ \bar{H}^{(i)}(\cdot,C)^{\mathsf{q}} }{\bar{H}_{\beta;-i}},
	\end{align}
	where the last inequality follows by Chebyshev's association inequality, cf. \cite[Theorem 2.14]{boucheron2013concentration}. Using the definition (\ref{ineq:smooth_error_2}) and the polynomial growth assumption of $\{\psi_i\}$ along with $\max_i \pnorm{\psi_{i;\rho}-\psi_i}{\infty}\leq \mathscr{M}_{\psi}(\rho)\leq 1$, with $\mathsf{q}_0\equiv q_0\mathsf{q}$ and $\bar{\mathsf{q}}_0\equiv 2\ceil{\mathsf{q}_0/2}\leq \mathsf{q}_0+3=\bar{\mathsf{q}}$, we have
	\begin{align}\label{ineq:smooth_error_5}
	\E \bigang{ \bar{H}^{(i)}(\cdot,C)^{\mathsf{q}} }{\bar{H}_{\beta;-i}}&=\E \bigang{ \psi_{i;\rho}\big((C w)_i\big)^{\mathsf{q}} }{\bar{H}_{\beta;-i}}\lesssim_q 1+ \E \bigang{ \psi_{i}\big((C w)_i\big)^{\mathsf{q}} }{\bar{H}_{\beta;-i}} \nonumber\\
	&\lesssim L_{\psi_i}^{\mathsf{q}}\bigg\{1+ \E \bigang{(C_{ij} w_j)^{\mathsf{q}_0}}{\bar{H}_{\beta;-i}} + \E \biggang{\biggabs{\sum_{k\neq j} C_{ik}w_k}^{\mathsf{q}_0}}{\bar{H}_{\beta;-i}}\bigg\} \nonumber\\
	&\lesssim_{\tau,q,M} L_{\psi_i}^{\mathsf{q}}L_n^{\mathsf{q}_0}(1+t^{\mathsf{q}_0}).
	\end{align}
	Here the last inequality follows as
	\begin{align*}
	&\E\biggang{\bigg(\sum_{k\neq j} C_{ik}w_k\bigg)^{\bar{\mathsf{q}}_0}}{\bar{H}_{\beta;-i}}= \E \sum_{k_1,\cdots,k_{\bar{\mathsf{q}}_0}\neq j} C_{ik_1}\cdots C_{i k_{\bar{\mathsf{q}}_0}} \bigang{w_{k_1}\cdots w_{k_{\bar{\mathsf{q}}_0}}}{\bar{H}_{\beta;-i}} \\
	& \stackrel{(\ast)}{=} \sum_{k_1,\cdots,k_{\bar{\mathsf{q}}_0}\neq j} \E C_{ik_1}\cdots C_{i k_{\bar{\mathsf{q}}_0}} \E\bigang{w_{k_1}\cdots w_{k_{\bar{\mathsf{q}}_0}}}{\bar{H}_{\beta;-i}}\lesssim_{\tau,q,M} L_n^{\bar{\mathsf{q}}_0},
	\end{align*}
	and $(\ast)$ holds by independence between the Gibbs measure $\ang{\cdot}{\bar{H}_{\beta;-i}(\cdot,C)}$ and $C_{i\cdot}$. Combining (\ref{ineq:smooth_error_4})-(\ref{ineq:smooth_error_5}), we have
	\begin{align*}
	\E  \bigang{\bar{H}^{(i)}(\cdot,C)^{\mathsf{q}} }{\bar{H}_\beta}\lesssim_{\tau,M,q} L_{\psi_i}^{\mathsf{q}}L_n^{\mathsf{q}_0}(1+t^{\mathsf{q}_0}).
	\end{align*}
	Using (\ref{ineq:smooth_error_3}), it follows that
	\begin{align*}
	&\E \bigabs{\partial_{ij}^3 \mathsf{g}\big(m^{-1}F_\beta (C_{ij}(t))\big)}\lesssim_{\tau,M,q}K_{\mathsf{g}}\cdot m^{-1} \beta^2  L_n^{\mathsf{q}_0+3} \mathscr{D}_{\psi}^3(\rho)\cdot L_{\psi_i}^{\mathsf{q}} \cdot (1+t^{\mathsf{q}_0}).
	\end{align*}
	Now we may apply Chatterjee's Lindeberg principle (cf. Theorem \ref{thm:lindeberg}):
	\begin{align*}
	&\bigabs{\E \mathsf{g}\big(m^{-1}F_\beta(A)\big)-\E \mathsf{g}\big(m^{-1}F_\beta(B)\big) }\big/ K_{\mathsf{g}} \\
	&\lesssim_{\tau,M,q} m^{-1} \beta^2 L_n^{\mathsf{q}_0+3} \mathscr{D}_{\psi}^3(\rho)\cdot \sum_{i,j} L_{\psi_i}^{\mathsf{q}}\max_{D \in \{A_{ij},B_{ij}\}} \biggabs{\E \int_0^{D} (1+t^{\mathsf{q}_0})(D-t)^2\,\d{t}}\\
	&\lesssim_{\tau,M,q} \mathscr{D}_{\psi}^3(\rho)\cdot \beta^2 L_n^{\bar{\mathsf{q}}} \Av\big(\{L_{\psi_i}^{\mathsf{q}}\}\big) \cdot  n^{-1/2},
	\end{align*}
	where the last inequality follows as for $D \in \{A_{ij},B_{ij}\}$,
	\begin{align*}
	\biggabs{\E \int_0^{D} (1+t^{\mathsf{q}_0})(D-t)^2\,\d{t}}& = \biggabs{ \E D^3 \int_0^1 \big(1+D^{\mathsf{q}_0} u^{\mathsf{q}_0}\big) (1-u^2)\,\d{u}}\\
	& \lesssim \E \abs{D}^3+\E \abs{D}^{\mathsf{q}_0+3}\lesssim_{\tau,M,q} n^{-3/2}.
	\end{align*}
	The proof is complete. 
\end{proof}

We are now in position to prove Theorem \ref{thm:universality_smooth}.

\begin{proof}[Proof of Theorem \ref{thm:universality_smooth}]
	Fix $\delta \in (0,1)$, and let $\mathcal{S}_{n,\delta}$ be as in Proposition \ref{prop:comparision_discrete}. Then as $\abs{\mathcal{S}_{n,\delta}}\leq \big(1+2L_n/\delta\big)^n$, using the proceeding Propositions \ref{prop:comparision_smooth}-\ref{prop:dim_free_smooth_min},
	\begin{align*}
	&\biggabs{\E \mathsf{g}\Big(\min_{w \in \mathcal{S}_n } H_\psi(w,A)\Big) - \E  \mathsf{g}\Big(\min_{w \in \mathcal{S}_n } H_\psi(w,B)\Big)  }\\
	&\leq 2 \max_{D \in \{A,B\}} \biggabs{\E \mathsf{g}\Big(\min_{w \in \mathcal{S}_{n}} H_\psi(w,D)\Big) - \E \mathsf{g}\Big( \min_{w \in \mathcal{S}_{n}} H_{\psi_\rho}(w,D) \Big)} \\ 
	&\qquad + 2 \max_{D \in \{A,B\}} \biggabs{\E \mathsf{g}\Big(\min_{w \in \mathcal{S}_{n,\delta}} H_{\psi_\rho}(w,D)\Big) - \E \mathsf{g}\Big( \min_{w \in \mathcal{S}_{n}} H_{\psi_\rho}(w,D) \Big)}\\
	&\qquad + 2 \max_{D \in \{A,B\}}  \biggabs{\E \mathsf{g}\Big(\min_{w \in \mathcal{S}_{n,\delta} } H_{\psi_\rho}(w,D)\Big) - \E \mathsf{g}\Big[m^{-1}F_\beta\Big( \big(\bar{H}_{\psi_\rho}(w,D)\big)_{w \in \mathcal{S}_{n,\delta} }   \Big)\Big]}\\
	&\qquad + \biggabs{\E \mathsf{g} \Big[m^{-1} F_\beta\Big( \big(\bar{H}_{\psi_\rho}(w,A)\big)_{w \in \mathcal{S}_{n,\delta}}   \Big)\Big]-\E \mathsf{g} \Big[m^{-1}F_\beta\Big( \big(\bar{H}_{\psi_\rho}(w,B)\big)_{w \in \mathcal{S}_{n,\delta} }   \Big)\Big] }\\
	&\lesssim_{\tau,q,M} K_{\mathsf{g}}  \cdot \bigg(\mathscr{M}_\psi(\rho)+  \mathscr{D}_{\psi}^3(\rho)\\
	&\qquad\qquad  \times \Big\{  L_n^{q_0q_1} \Av\big(\{L_{\psi_i}^{q_1}\}\big) \cdot \big(1+ \E (\pnorm{A}{\infty}+\pnorm{B}{\infty})^{q_0q_1+1}\big)\cdot \delta+ \mathscr{M}_{\mathsf{f}}(L_n,\delta)\\
	&\qquad\qquad +\beta^{-1} \log_+(L_n/\delta)+ (1\vee\beta^2) L_n^{\bar{\mathsf{q}}} \Av\big(\{L_{\psi_i}^{\mathsf{q}}\}\big)  n^{-1/2}\Big\}\bigg).
	\end{align*}
	Using $\E \pnorm{A}{\infty}^{q_0q_1+1} = \E \big(\max_i \sum_j \abs{A_{ij}}\big)^{q_0q_1+1} \lesssim_{\tau,M} n\cdot (\sqrt{n})^{q_0q_1+1} = n^{(q_0q_1+3)/2}$, 
	in the regime $0<\delta\leq \omega_n=n^{-(q_0q_1+4)/2}$, the above bound reduces to
	\begin{align*}
	&\mathscr{M}_\psi(\rho)+  \mathscr{D}_{\psi}^3(\rho)  \Big\{   \mathscr{M}_{\mathsf{f}}(L_n,\delta) +\beta^{-1} \log_+(L_n/\delta)+ (1\vee\beta^2) L_n^{\bar{\mathsf{q}}} \Av\big(\{L_{\psi_i}^{\mathsf{q}}\}\big)   n^{-1/2}\Big\},
	\end{align*}
	modulo a multiplicative factor of $K_{\mathsf{g}}$. Now optimizing over $\beta>0$, the last two terms in the above bracket becomes $\mathfrak{a}^{2/3}\mathfrak{b}^{1/3}+\mathfrak{b}$, where $\mathfrak{a}\equiv \log_+(L_n/\delta)$ and $\mathfrak{b}\equiv L_n^{\bar{\mathsf{q}}} \Av\big(\{L_{\psi_i}^{\mathsf{q}}\}\big)n^{-1/2}$. The term $\mathfrak{b}$ can be dropped for free as it is effective only when $\mathfrak{b}\geq \mathfrak{a}\geq 1$---in this case the bound is trivial. This concludes the main inequality.  
	
	For the second claim, take any non-negative $g \in C^3(\R)$ such that $g(x)= 0$ for $x\leq 1$ and $g(x)=2$ for $x>3$. With $\phi(A)\equiv \min_{w \in \mathcal{S}_n } H_\psi(w,A)$, for any $z>0,\epsilon>0$, 
	\begin{align*}
	\Prob\big( \phi(A)-z>3\epsilon\big)&\leq \Prob \big[g\big(( \phi(A)-z)\big/\epsilon \big)\geq 2\big]\leq \frac{1}{2}\cdot \E g\big(( \phi(A)-z)\big/\epsilon \big)\\
	&\leq \frac{1}{2}\cdot \E g\big( (\phi(B)-z)\big/\epsilon \big)+C\cdot (1\vee \epsilon^{-3})  \mathsf{r}_{\mathsf{f}}(L_n)\\
	&\leq \Prob\big( \phi(B)-z>\epsilon\big)+ C\cdot (1\vee \epsilon^{-3})  \mathsf{r}_{\mathsf{f}}(L_n),
	\end{align*}
	as desired.
\end{proof}

\subsection{Proof of Theorem \ref{thm:universality_optimizer}}\label{section:proof_universality_optimizer}

	By (O1), the event $E_n\equiv \big\{\hat{w}_A \in [-L_n,L_n]^n\big\}$ satisfies $\Prob(E_n^c)\leq \epsilon_n$. Using that the event
	\begin{align*}
	&\bigg\{\min_{\pnorm{w}{\infty}\leq L_n} H_\psi(w,A)< z+3\rho_0\bigg\}\cap \bigg\{\min_{w \in \mathcal{S}_n\cap [-L_n,L_n]^n} H_\psi(w,A)> z+6\rho_0\bigg\}
	\end{align*}
	is included in $\big\{\hat{w}_A \notin \mathcal{S}_n \cap [-L_n,L_n]^n \big\}$, we have
	\begin{align*}
	&\Prob\big(\hat{w}_A \in \mathcal{S}_n\big)\leq \Prob\big(\hat{w}_A \in \mathcal{S}_n \cap [-L_n,L_n]^n\big)+\Prob(E_n^c)\\
	&\leq \Prob\bigg(\min_{\pnorm{w}{\infty}\leq L_n} H_\psi(w,A)\geq  z+3\rho_0\bigg)+ \Prob\bigg(\min_{w \in \mathcal{S}_n\cap [-L_n,L_n]^n} H_\psi(w,A)\leq z+6\rho_0\bigg)+\epsilon_n\\
	&\equiv \mathfrak{p}_1+\mathfrak{p}_2+\epsilon_n. 
	\end{align*}
	First we handle $\mathfrak{p}_1$:
	\begin{align*}
	\mathfrak{p}_1 &\leq \Prob\bigg(\min_{\pnorm{w}{\infty}\leq L_n} H_\psi(w,G)\geq z+\rho_0 \bigg)+ C_1(1\vee\rho_0^{-3}) \mathsf{r}_{\mathsf{f}}(L_n) \quad \hbox{(by Theorem \ref{thm:universality_smooth})}\\
	&\leq \Prob\bigg(\min_{w \in \R^n} H_\psi(w,G)\geq z+\rho_0 \bigg)+\epsilon_n+ C_1(1\vee\rho_0^{-3}) \mathsf{r}_{\mathsf{f}}(L_n) \quad \hbox{(by (O1))}\\
	&\leq  2\epsilon_n+ C_1(1\vee\rho_0^{-3}) \mathsf{r}_{\mathsf{f}}(L_n)\quad \hbox{(by (O2))}.
	\end{align*}
	As $\min_{w \in \mathcal{S}_n \cap [-L_n,L_n]^n} H_\psi(w,G)\geq \min_{w \in \mathcal{S}_n} H_\psi(w,G)$, (O2) then entails that
	\begin{align*}
	\Prob\bigg(\min_{w \in \mathcal{S}_n  \cap [-L_n,L_n]^n } H_\psi(w,G)\leq z+2\rho_0 \bigg) \leq \epsilon_n.
	\end{align*}
	Consequently, using Theorem \ref{thm:universality_smooth} we obtain, 
	\begin{align*}
	\mathfrak{p}_2\leq \epsilon_n+ C_1(1\vee \rho_0^{-3}) \mathsf{r}_{\mathsf{f}}(L_n).
	\end{align*}
	Collecting the estimates to conclude. \qed

\subsection{Proof of Theorem \ref{thm:min_max_universality}}\label{section:proof_universality_min_max}

	We introduce some further notation. Fix finite sets $U,W$, and $\beta_1,\beta_2 >0$. For $x=(x_{u,w})_{u \in U, w \in W}$, let
	\begin{align*}
	F_u(x)&\equiv F_{\beta_2}\big((x_{u,w})_{w \in W}\big) = -\beta_2^{-1} \log\bigg(\sum_{w \in W} e^{-\beta_2 x_{u,w}}\bigg),\\
	F(x)&\equiv \beta_1^{-1} \log\bigg(\sum_{u \in U} e^{\beta_1 F_u(x)}\bigg). 
	\end{align*}
	Then using $\min_w x_{u,w}-\log\abs{W}/\beta_2\leq F_u(x)\leq \min_w x_{u,w}$, and $\max_u F_u(x)\leq F(x)\leq \max_u F_u(x)+\log\abs{U}/\beta_1$, we have
	\begin{align}\label{ineq:min_max_universality_0}
	\max_{u \in U} \min_{w \in W} x_{u,w}-\frac{\log \abs{W}}{\beta_2}\leq F(x)\leq \max_{u \in U} \min_{w \in W} x_{u,w}+ \frac{\log \abs{U}}{\beta_1}.
	\end{align}
	For notational simplicity, we use the notation $x(A)$ when  $X(u,w;A)$ is viewed as a function of $A$:
	\begin{align*}
	x_{u,w} \equiv x_{u,w}(A)\equiv X(u,w;A)= u^\top A w +Q(u,w). 
	\end{align*}
	\noindent (\textbf{Step 1}). Using (\ref{ineq:min_max_universality_0}), for any finite sets $S_u\subset \R^m,S_w\subset \R^n$, we have
	\begin{align}\label{ineq:min_max_universality_1}
	&\biggabs{\E \mathsf{g}\Big(\max_{u \in S_u} \min_{w \in S_w} x_{u,w}(A)\Big)-\E \mathsf{g}\Big[F\Big((x_{u,w}(A))_{(u,w)\in S_u\times S_w}\Big)\Big]     }\nonumber\\
	&\leq \pnorm{\mathsf{g}'}{\infty}\bigg(\frac{\log \abs{S_u}}{\beta_1}+\frac{\log \abs{S_w}}{\beta_2}\bigg).
	\end{align}
	
	\noindent (\textbf{Step 2}). In this step, we show the following. Take any $\mathcal{S}_u \subset [-L_u,L_u]^m,\mathcal{S}_w \subset [-L_w,L_w]^n$. Let $\mathcal{S}_{u,\delta}, \mathcal{S}_{w,\delta}$ be $\delta$-covers of $\mathcal{S}_u,\mathcal{S}_w$ under $\ell_\infty$ as constructed in the proof of Proposition \ref{prop:comparision_discrete} (so $\abs{\mathcal{S}_{u,\delta}}\leq \ceil{2L_u/\delta}^m$, $\abs{\mathcal{S}_{w,\delta}}\leq \ceil{2L_w/\delta}^n$). Then,
	\begin{align}\label{ineq:min_max_universality_2}
	&\biggabs{\E \mathsf{g}\Big(\max_{u \in \mathcal{S}_{u,\delta}} \min_{w \in \mathcal{S}_{w,\delta}} x_{u,w}(A)\Big)-\E \mathsf{g}\Big(\max_{u \in \mathcal{S}_u} \min_{w \in \mathcal{S}_w} x_{u,w}(A)\Big) }\nonumber\\
	&\leq \pnorm{\mathsf{g}'}{\infty}\bigg(\sum_{i,j}\E \abs{A_{ij}}\cdot (L_u+L_w)\delta+ \mathscr{M}_Q(L,\delta)\bigg).
	\end{align}
	We first claim that
	\begin{align}\label{ineq:min_max_universality_3}
	\biggabs{\max_{u \in \mathcal{S}_{u,\delta}} \min_{w \in \mathcal{S}_{w,\delta}} x_{u,w}-\max_{u \in \mathcal{S}_u} \min_{w \in \mathcal{S}_w} x_{u,w}}\leq \sum_{i,j}\abs{A_{ij}}\cdot (L_u+L_w)\delta+ \mathscr{M}_Q(L,\delta).
	\end{align}
	Let $u_{\ast,\delta}\in \mathcal{S}_{u,\delta}$ be the maximizer for $u\mapsto \min_{w \in \mathcal{S}_{w,\delta}} x_{u,w}$, and $u_\ast \in \mathcal{S}_u$ be such that $\pnorm{u_{\ast,\delta}-u_\ast}{\infty}\leq \delta$. Let $w_\ast \in \mathcal{S}_w$ be the minimizer for $w \mapsto x_{u_\ast,w}$, and $w_{\ast,\delta} \in \mathcal{S}_{w,\delta}$ be such that $\pnorm{w_{\ast,\delta}-w_\ast}{\infty}\leq \delta$. Then 
	\begin{align*}
	&\max_{u \in \mathcal{S}_{u,\delta}} \min_{w \in \mathcal{S}_{w,\delta}} x_{u,w}-\max_{u \in \mathcal{S}_u} \min_{w \in \mathcal{S}_w} x_{u,w}\\
	&\leq \min_{w \in \mathcal{S}_{w,\delta}} x_{u_{\ast,\delta},w}- \min_{w \in \mathcal{S}_w} x_{u_\ast,w}\leq x_{u_{\ast,\delta},w_{\ast,\delta}} - x_{u_\ast,w_\ast}.
	\end{align*}
	Arguing the other direction in a similar way, it is now easy to see that
	\begin{align*}
	&\biggabs{\max_{u \in \mathcal{S}_{u,\delta}} \min_{w \in \mathcal{S}_{w,\delta}} x_{u,w}-\max_{u \in \mathcal{S}_u} \min_{w \in \mathcal{S}_w} x_{u,w}}\\
	&\leq \sup_{\cdots}\abs{x_{u_\delta,w_\delta}-x_{u,w}}\leq \sup_{\cdots}\bigabs{u_\delta^\top A w_\delta-u^\top A w}+ \mathscr{M}_Q(L,\delta),
	\end{align*}
	where the supremum are taken over all $u,u_\delta \in [-L_u,L_u]^m$, $w,w_\delta \in [-L_w,L_w]^n$ with $\pnorm{u-u_\delta}{\infty}\vee \pnorm{w-w_\delta}{\infty}\leq \delta$. Furthermore, the first term on the right hand side above can be further bounded by 
	\begin{align*}
	&\sup_{\cdots}\big(\abs{u^\top A w-u^\top A w_\delta}+\abs{u^\top A w_\delta-u_\delta^\top A w_\delta}\big)\\
	&\leq \sup_{\cdots}\big(\pnorm{A^\top u}{1}+\pnorm{Aw_\delta}{1}\big)\delta \leq \sum_{i,j}\abs{A_{ij}}\cdot (L_u+L_w)\delta.
	\end{align*}
	The claim (\ref{ineq:min_max_universality_3}) follows by combining the two above displays. Now (\ref{ineq:min_max_universality_2}) follows by a simple Taylor expansion.

	\noindent (\textbf{Step 3}). In this step, we show the following: There exists some universal constant $K>0$ such that for any finite set $U\subset \R^m,W\subset \R^n$ with $\max_{u \in U}\pnorm{u}{\infty}\leq L_u, \max_{w \in W}\pnorm{w}{\infty}\leq L_w$, we have for $\mathsf{g} \in C^3(\R)$,
	\begin{align}\label{ineq:min_max_universality_4}
	&\bigabs{\E \mathsf{g}\big(F(x(A))-\E \mathsf{g}\big(F(x(B))\big)}\nonumber\\
	&\leq K\cdot  \sum_{\ell=1}^3 \pnorm{\mathsf{g}^{(\ell)}}{\infty}(\beta_1+\beta_2)^{\ell-1} L_u^\ell L_w^\ell \cdot \sum_{i,j} \big(\E \abs{A_{ij}}^3+\E \abs{B_{ij}}^3\big).
	\end{align}
	To this end, define the Hamiltonian
	\begin{align*}
	H_1(u)&\equiv \frac{e^{\beta_1 F_u(x)} }{ \sum_{u \in U} e^{\beta_1 F_u(x)}   }, \quad H_{|1}(u,w)\equiv \frac{e^{-\beta_2 x_{u,w}} }{ \sum_{w \in W} e^{-\beta_2 x_{u,w}}  },
	\end{align*}
	and $H(u,w)\equiv H_1(u)\cdot H_{|1}(u,w)$. Recall the notation (\ref{def:hamilton}). Then with $\partial_{ij}\equiv \partial/\partial A_{ij}$, some calculations show that $\partial_{ij} F_u(x) = \ang{u_iw_j}{H_{|1}}$ and
	\begin{align*}
	\partial_{ij}H_1(u)&=\beta_1 H_1(u)\Big(\ang{u_iw_j}{H_{|1}}-\ang{u_iw_j}{H}\Big)\\
	\partial_{ij} H_{|1}(u,w)& = \beta_2 H_{|1}(u,w)\Big(-u_iw_j+\ang{u_iw_j}{H_{|1}}\Big)\\
	\partial_{ij} H(u,w)& = H(u,w)\Big[\beta_1\Big(\ang{u_iw_j}{H_{|1}}-\ang{u_iw_j}{H}\Big) +\beta_2 \Big(-u_iw_j+\ang{u_iw_j}{H_{|1}}\Big)\Big].
	\end{align*}
	Here $\ang{f(u,w)}{H}, \ang{f(u,w)}{H_{|1}}, \ang{f(u)}{H_1}$ are understood as total expectation over the Gibbs measure $H$ on $(u,w)$, conditional expectation over $H_{|1}(u,\cdot)$ on $w$, and marginal expectation over $H_1$ on $u$. Consequently:
	\begin{enumerate}
		\item As $\partial_{ij} F(x)=\ang{\partial_{ij} F_u(x)}{H_1}= \ang{u_i w_j}{H}$, we have $\abs{\partial_{ij} F(x)}\leq L_u L_w$. 
		\item Using (1) and the generic derivative formula (\ref{eqn:hamilton_derivative}), we find $\partial_{ij}^2 F(x)= -\ang{u_iw_j \partial_{ij}H}{H}+\ang{u_iw_j}{H}\ang{\partial_{ij}H}{H}$.  By the derivative formula for $\partial_{ij} H$ above and the fact that $\abs{H}\vee \abs{H_1}\vee \abs{H_{|1}}\leq 1$, we have $\abs{\partial_{ij}^2 F(x)}\lesssim (\beta_1+\beta_2)L_u^2 L_w^2$.
		\item The formula for the third derivative $\partial_{ij}^3 F(x)$ is rather tedious, but its exact form is immaterial and we only need $\abs{\partial_{ij}^3 F(x)}\lesssim (\beta_1^2+\beta_2^2)L_u^3 L_w^3$.
	\end{enumerate}
	Now using the formula $\partial_{ij}^3 \mathsf{g}(F(x))=\sum_{\ell=1}^3\binom{2}{\ell-1}\mathsf{g}^{(\ell)}(F(x))\cdot \partial_{ij}^{4-\ell} F(x)$, we have 
	\begin{align*}
	\sup_{t \in [0,1]}\bigabs{\partial_{ij}^3 \mathsf{g}\big(F(x(C_{ij}(t)))\big)}\leq K\cdot \sum_{\ell=1}^3 \pnorm{\mathsf{g}^{(\ell)}}{\infty}(\beta_1+\beta_2)^{\ell-1} L_u^\ell L_w^\ell
	\end{align*}
	for some universal constant $K>0$. Here $\{C_{ij}(t)\}_{t \in [0,1]}$ is the Lindeberg interpolation between $A$ and $B$ as defined in the proof of Proposition \ref{prop:dim_free_smooth_min}. So by Chatterjee's Lindeberg principle (cf. Theorem \ref{thm:lindeberg}),
	\begin{align*}
	&\bigabs{\E \mathsf{g}\big(F(x(A))-\E \mathsf{g}\big(F(x(B))\big)}\bigg/ \sum_{\ell=1}^3 \pnorm{\mathsf{g}^{(\ell)}}{\infty}(\beta_1+\beta_2)^{\ell-1} L_u^\ell L_w^\ell\\
	&\lesssim \sum_{i,j} \max_{D \in \{A_{ij},B_{ij}\}}\biggabs{\E \int_0^D (D-t)^2\,\d{t}} \asymp \sum_{i,j} \big(\E \abs{A_{ij}}^3+\E \abs{B_{ij}}^3\big),
	\end{align*}
	proving (\ref{ineq:min_max_universality_4}).
	
	\noindent (\textbf{Step 4}). With the claims proved in Steps 1-3, and noting that $\abs{\mathcal{S}_{u,\delta}}\leq (1+2L_u/\delta)^m$ and $\abs{\mathcal{S}_{w,\delta}}\leq (1+2L_w/\delta)^n$, we have
	\begin{align*}
	&\biggabs{\E \mathsf{g}\Big(\max_{u \in \mathcal{S}_u} \min_{w \in \mathcal{S}_w} x_{u,w}(A)\Big)-\E \mathsf{g}\Big(\max_{u \in \mathcal{S}_u} \min_{w \in \mathcal{S}_w} x_{u,w}(B)\Big)}\\
	&\leq 2\max_{D \in \{A,B\}}\biggabs{\E \mathsf{g}\Big(\max_{u \in \mathcal{S}_u} \min_{w \in \mathcal{S}_w} x_{u,w}(D)\Big)-\E \mathsf{g}\Big(\max_{u \in {S}_u(\delta)} \min_{w \in {S}_w(\delta)} x_{u,w}(D)\Big)}\\
	&\quad + 2\max_{D \in \{A,B\}}\biggabs{\E \mathsf{g}\Big(\max_{u \in {S}_u(\delta)} \min_{w \in {S}_w(\delta)} x_{u,w}(D)\Big)- \E \mathsf{g}\Big[F\Big((x_{u,w}(D))_{(u,w)\in \mathcal{S}_{u,\delta}\times \mathcal{S}_{w,\delta}}\Big)\Big] }\\
	&\quad + \biggabs{\E \mathsf{g}\Big[F\Big((x_{u,w}(A))_{(u,w)\in \mathcal{S}_{u,\delta}\times \mathcal{S}_{w,\delta}}\Big)\Big]-\E \mathsf{g}\Big[F\Big((x_{u,w}(B))_{(u,w)\in \mathcal{S}_{u,\delta}\times \mathcal{S}_{w,\delta}}\Big)\Big]}\\
	&\lesssim \pnorm{\mathsf{g}'}{\infty}\bigg\{ M_1 (L_u+L_w)\delta+ \mathscr{M}_Q(L,\delta)+ \bigg(\frac{m\log_+(L_u/\delta)}{\beta_1}+\frac{n \log_+(L_w/\delta)}{\beta_2}\bigg)\bigg\}\\
	&\qquad +M_3\cdot \sum_{\ell=1}^3 \pnorm{\mathsf{g}^{(\ell)}}{\infty}(\beta_1+\beta_2)^{\ell-1} L_u^\ell L_w^\ell.
	\end{align*}
	By setting $\beta\equiv \beta_1=\beta_2$, the above display can be further bounded, up to a multiplicative factor of $K_{\mathsf{g}}$, by
	\begin{align*}
	M_1 L\delta+\mathscr{M}_Q(L,\delta)+ \beta^{-1}(m+n)\log_+(L/\delta)+  (1\vee \beta^2)\cdot M_3 L^6.
	\end{align*}
	Optimizing over $\beta>0$ to conclude. \qed

\subsection{Proof of Corollary \ref{cor:min_max_universality}}\label{section:proof_universality_min_max_cor}

	Using Theorem \ref{thm:min_max_universality} with the random matrix replaced by $A/m$, we have $M_\ell\leq K n^{2-3\ell/2} M_{0}^{\ell/3}$ for $\ell=1,2,3$, the bound in Theorem \ref{thm:min_max_universality} is bounded, up to a constant factor of $K_{\mathsf{g}}$ that depends on $(\tau, M_0)$, by
	\begin{align*}
	n^{1/2} L \delta+ \mathscr{M}_{Q_n}(L,\delta)+\frac{L^2 \log_+^{2/3}(L/\delta)}{n^{1/6}}.
	\end{align*}
	The first term can be assimilated into the last term in the regime $\delta\leq n^{-1}$. The second claim follows from the same argument as in the proof of Theorem \ref{thm:universality_smooth}. Generalizations to the other two cases are immediate. \qed

\section{Proofs for Section \ref{section:examples}: Ridge}\label{section:proof_ridge}

\noindent \emph{Convention}: We shall write
\begin{align*}
\bar{H}^{\ridge}(w) \equiv \bar{H}^{\ridge}(w,A)\equiv \bar{H}^{\ridge}(w,A,\xi),
\end{align*}
and will usually omit the superscript $(\cdot)^{\ridge}$ if no confusion could arise. We also usually omit the subscript $A$ that indicates the design matrix, but we will use the subscript $G$ for Gaussian designs when needed. For notational simplicity, we write
\begin{align*}
r_n\equiv \sqrt{\log n/n},\quad s_n\equiv n^{-1/6}\log^2 n.
\end{align*}
The constants $C,K>0$, typically depending on $\sigma,\lambda,\tau,M_2$, will vary from line to line. All notation will be local in this section.

\subsection{Proof of Proposition \ref{prop:ridge_risk}}

Define the (column) leave-one-out Ridge version
\begin{align}\label{def:ridge_loo}
\hat{w}^{(s)}\equiv \argmin_{w\in\R^n: w_s = 0} H(w,A,\xi)= \argmin_{w\in\R^n: w_s = 0} \frac{1}{2}\pnorm{Aw-\xi}{}^2 + \frac{\lambda}{2} \pnorm{w + \mu_0}{}^2,
\end{align}
and the (row) leave-one-out Ridge version
\begin{align}\label{def:ridge_loo_row}
\hat{w}^{[t]}\equiv  \argmin_{w\in\R^n} \frac{1}{2}\pnorm{A_{[-t]}w-\xi_{-t}}{}^2 + \frac{\lambda}{2} \pnorm{w + \mu_0}{}^2,
\end{align}
where $A_{[-t]} \in \R^{(m-1)\times n}$ is $A$ minus its $t$-th row. Recall that $A_s$ and $a_t$ denote the $s$-th column and the $t$-th row of $A$ respectively.

\begin{lemma}\label{lem:ridge_loo}
	The following deterministic inequalities hold.
	\begin{enumerate}
		\item $\pnorm{  A\big(\hat{w}^{(s)}-\hat{w}\big)  }{}^2+ \lambda\cdot \hat{w}_s^2
		\leq 4\lambda^{-1} \big( \abs{A_s^\top (A_{-s} \hat{w}^{(s)}_{-s}-\xi)}+\lambda \abs{\mu_{0,s}}  \big)^2$.
		\item $\pnorm{ A(\hat{w}^{[t]}-\hat{w}) }{}\leq 2 \big(\abs{\xi_t}+\abs{a_t^\top \hat{w}^{[t]}}\big)$. 
	\end{enumerate}
\end{lemma}
\begin{proof}
	Write $\mathsf{f}(w)\equiv (\lambda/2)\pnorm{w+\mu_0}{}^2$ (without normalization $1/m$). 
	
	\noindent (1). By the cost optimality of $\bar{H}$,
	\begin{align*}
	0&\leq \bar{H}(\hat{w}^{(s)})- \bar{H}(\hat{w})\\
	& = \frac{1}{2} \bigpnorm{A\big(\hat{w}^{(s)}-\hat{w}\big)+A\hat{w}-\xi}{}^2 + \mathsf{f}(\hat{w}^{(s)}) - \frac{1}{2} \pnorm{A\hat{w}-\xi}{}^2-\mathsf{f}(\hat{w})\\
	& = \frac{1}{2}\bigpnorm{  A\big(\hat{w}^{(s)}-\hat{w}\big)  }{}^2+\iprod{ A\big(\hat{w}^{(s)}-\hat{w}\big)  }{  A\hat{w}-\xi } + \mathsf{f}(\hat{w}^{(s)})-\mathsf{f}(\hat{w})\\
	& = -\frac{1}{2}\bigpnorm{  A\big(\hat{w}^{(s)}-\hat{w}\big)  }{}^2+\iprod{ A\big(\hat{w}^{(s)}-\hat{w}\big)  }{  A\hat{w}^{(s)}-\xi } + \mathsf{f}(\hat{w}^{(s)})-\mathsf{f}(\hat{w}).
	\end{align*}
	Now using the decomposition
	\begin{align*}
	&\iprod{ A\big(\hat{w}^{(s)}-\hat{w}\big)  }{  A\hat{w}^{(s)}-\xi } = - \hat{w}_s\cdot A_s^\top (A_{-s} \hat{w}^{(s)}_{-s}-\xi)    +\iprod{ \hat{w}^{(s)}_{-s}-\hat{w}_{-s}  }{  A_{-s}^\top \big(A_{-s}\hat{w}^{(s)}_{-s}-\xi\big) },
	\end{align*}
	we arrive at
	\begin{align*}
	0&\leq -\frac{1}{2}\bigpnorm{  A\big(\hat{w}^{(s)}-\hat{w}\big)  }{}^2- \hat{w}_s\cdot A_s^\top (A_{-s} \hat{w}^{(s)}_{-s}-\xi)\nonumber\\
	&\qquad +\iprod{ \hat{w}^{(s)}_{-s}-\hat{w}_{-s}  }{  A_{-s}^\top \big(A_{-s}\hat{w}^{(s)}_{-s}-\xi\big) } + \mathsf{f}(\hat{w}^{(s)})-\mathsf{f}(\hat{w}).
	\end{align*}
	Using KKT condition for $\hat{w}^{(s)}$ which reads $A_{-s}^\top \big(A_{-s}\hat{w}^{(s)}_{-s}-\xi\big) = -\lambda (\mu_{0,-s}+\hat{w}^{(s)}_{-s})$, 
	\begin{align*}
	0&\leq  -\frac{1}{2}\bigpnorm{  A\big(\hat{w}^{(s)}-\hat{w}\big)  }{}^2- \hat{w}_s\cdot A_s^\top (A_{-s} \hat{w}^{(s)}_{-s}-\xi)\\
	&\qquad + \frac{\lambda}{2}\bigg( \pnorm{\hat{w}^{(s)}+\mu_0}{}^2- \pnorm{\hat{w}+\mu_0}{}^2-2\iprod{  \hat{w}^{(s)}_{-s}-\hat{w}_{-s} }{ \mu_{0,-s}+\hat{w}^{(s)}_{-s} } \bigg)\\
	& \leq  -\frac{1}{2}\bigpnorm{  A\big(\hat{w}^{(s)}-\hat{w}\big)  }{}^2- \hat{w}_s\cdot A_s^\top (A_{-s} \hat{w}^{(s)}_{-s}-\xi)+ \frac{\lambda}{2}\big( (\mu_{0,s})^2- (\hat{w}_s+\mu_{0,s})^2\big).
	\end{align*}
	Here the last inequality follows by the convexity of $\mathsf{f}$. This implies 
	\begin{align*}
	\frac{1}{2}\bigpnorm{  A\big(\hat{w}^{(s)}-\hat{w}\big)  }{}^2+ \frac{\lambda}{2}\cdot \hat{w}_s^2
	&\leq \abs{\hat{w}_s}\cdot \big( \abs{A_s^\top (A_{-s} \hat{w}^{(s)}_{-s}-\xi)}+\lambda \abs{\mu_{0,s}}  \big)\\
	&\leq \frac{2}{\lambda} \big( \abs{A_s^\top (A_{-s} \hat{w}^{(s)}_{-s}-\xi)}+\lambda \abs{\mu_{0,s}}  \big)^2,
	\end{align*}
	as desired.
	
	\noindent (2). For any $w \in \R^n$, by expanding the Ridge cost $\bar{H}$ at $\hat{w}$, 
	\begin{align*}
	\bar{H}(w)& = \bar{H}(\hat{w}) + \frac{1}{2}\pnorm{A(w - \hat{w})}{}^2 +\iprod{A(w-\hat{w})}{A\hat{w}-\xi}+\mathsf{f}(w)-\mathsf{f}(\hat{w}).
	\end{align*}
	Using the KKT condition for $\hat{w}$, the last two terms of the above display becomes $
	\mathsf{f}(w)-\mathsf{f}(\hat{w})-\iprod{w-\hat{w}}{v(\hat{w})}\geq 0$, 
	where $v(\hat{w})$ is a subgradient of $\mathsf{f}$ at $\hat{w}$ (here $\mathsf{f}$ is smooth so it is the gradient, but we will maintain this generality for convenient generalization to Lasso later on). Consequently, for any $w \in \R^n$,
	\begin{align*}
	\frac{1}{2}\pnorm{A(w - \hat{w})}{}^2&\leq \bar{H}(w)- \bar{H}(\hat{w}).
	\end{align*}
	Let $a_t$ be the $t$-th row of $A$, $\hat{w}^{[t]}=\argmin_{w \in \R^n}\big(\bar{H}(w)-\abs{a_t^\top w-\xi_t}^2/2\big)$, so with $w=\hat{w}^{[t]}$, the above display further yields
	\begin{align*}
	\frac{1}{2}\pnorm{A(\hat{w}^{[t]} - \hat{w})}{}^2&\leq  \frac{1}{2}\Big(\abs{a_t^\top \hat{w}^{[t]}-\xi_t}^2- \abs{a_t^\top \hat{w}-\xi_t}^2\Big)\\
	&= -\frac{1}{2}\abs{a_t^\top(\hat{w}-\hat{w}^{[t]})}^2+ a_t^\top(\hat{w}^{[t]}-\hat{w})\big(a_t^\top \hat{w}^{[t]}-\xi_t\big)\\
	&\leq \pnorm{ A(\hat{w}^{[t]}-\hat{w}) }{}\cdot \big(\abs{\xi_t}+\abs{a_t^\top \hat{w}^{[t]}}\big),
	\end{align*}
	as desired.
\end{proof}

\begin{proof}[Proof of Proposition \ref{prop:ridge_risk}]	
	\noindent (1). By cost optimality of $H$, we have
	\begin{align*}
	\frac{1}{2}\pnorm{A\hat{w}-\xi}{}^2+\frac{\lambda}{2} \pnorm{\hat{w}+\mu_0}{}^2\leq \frac{1}{2}\pnorm{\xi}{}^2+ \frac{\lambda}{2}\pnorm{\mu_0}{}^2,
	\end{align*}
	so $\pnorm{A\hat{w}}{}^2\lesssim \pnorm{\xi}{}^2+\lambda \pnorm{\mu_0}{}^2$. The claim follows by the assumptions. 
	
	\noindent (2). Recall the column leave-one-out Ridge version $\hat{w}^{(s)}$ defined in (\ref{def:ridge_loo}). By Lemma \ref{lem:ridge_loo}, for any $s \in [n]$, we have
	\begin{align*}
	\abs{\hat{w}_s}&\leq 2\lambda^{-1}  \abs{A_s^\top (A_{-s} \hat{w}^{(s)}_{-s}-\xi)}+ 2\abs{\mu_{0,s}}\\
	&\leq 2\lambda^{-1}  \abs{ A_s^\top A_{-s} \hat{w}^{(s)}_{-s}  }+2\lambda^{-1} \abs{A_s^\top \xi}+ 2\abs{\mu_{0,s}}. 
	\end{align*}
	Using independence between $A_s$ and $A_{-s}\hat{w}_{-s}^{(s)}$, $\abs{ A_s^\top A_{-s} \hat{w}^{(s)}_{-s}  }\lesssim \lambda \sqrt{\log n}$ uniformly in $s$ with probability at least $1-C n^{-100}$. The second term $\abs{A_s^\top \xi}\lesssim \sqrt{\log n}$ uniformly in $s$ with the same probability. 
	
	\noindent (3). Recall the row leave-one-out Ridge version $\hat{w}^{[t]}$ defined in (\ref{def:ridge_loo_row}). Recall $a_t$ denotes the $t$-th row of $A$, 
	\begin{align*}
	\abs{(A\hat{w})_t} = \abs{a_t^\top \hat{w}}\leq \abs{a_t^\top \hat{w}^{[t]}}+ \abs{a_t^\top(\hat{w}-\hat{w}^{[t]})}\leq \abs{a_t^\top \hat{w}^{[t]}}+ \pnorm{A(\hat{w}-\hat{w}^{[t]})}{}.
	\end{align*}
	Using Lemma \ref{lem:ridge_loo}-(2), we have
	\begin{align*}
	\abs{(A\hat{w})_t}\lesssim \abs{\xi_t}+ \abs{a_t^\top \hat{w}^{[t]}}.
	\end{align*}
	As $\hat{w}^{[t]}$ is independent of $a_t$, so on an event with probability at least $1-C n^{-100}$, uniformly for all $t \in [m]$, $\abs{(A\hat{w})_t}\leq K \sqrt{\log n}$.
\end{proof}

\subsection{Proof of Theorem \ref{thm:ridge_universality_generic}}

	By Proposition \ref{prop:ridge_risk}, we may take $L_n=K (\sqrt{\log n}+\pnorm{\mu_0}{\infty})\leq n$ for $n$ large (as $\pnorm{\mu_0}{\infty}\leq \sqrt{M_2 n}$) and $\epsilon_n=C n^{-100}$. Now with $\mathsf{f}(w)\equiv(\lambda/2m)\pnorm{w+\mu_0}{}^2$, (\ref{cond:f_moduli}) can be easily verified, and we may apply Theorem \ref{thm:universality_reg} to conclude. \qed

\subsection{The Gaussian design problem}

Consider the Gaussian design problem where the entries of $A=G$ are i.i.d. $\mathcal{N}(0,1/m)$. Now let $h,{\ell}: \R^n\times \R^m\to \R$ be defined by 
\begin{align}\label{def:h_l_ridge}
h(w,u)& = \frac{1}{m}u^\top G w-\frac{1}{m}u^\top \xi-\frac{1}{2m} \pnorm{u}{}^2 + \frac{\lambda}{2m} \big(\pnorm{w+\mu_0}{}^2-\pnorm{\mu_0}{}^2\big),\nonumber\\
{\ell}(w,u)& = -\frac{1}{m^{3/2}} \pnorm{u}{}g^\top w +\frac{\pnorm{w}{}}{m^{1/2}}\cdot \frac{h^\top u}{m}-\frac{1}{m}u^\top \xi\nonumber\\
&\qquad\qquad -\frac{1}{2m}\pnorm{u}{}^2 + \frac{\lambda}{2m} \big(\pnorm{w+\mu_0}{}^2-\pnorm{\mu_0}{}^2\big).
\end{align}
Here $g \in \R^n, h \in \R^m$ are independent standard Gaussian vectors. The Ridge cost function in the Gaussian design case can be realized as $H(w,G)=\max_u h(w,u)$. Let $L(w)\equiv \max_u {\ell}(w,u)$ be the associated Gordon cost, and 
\begin{align}\label{def:ridge_psi_emp}
\psi_n(\beta,\gamma)&\equiv  \bigg(\frac{\sigma^2}{\gamma}+\gamma\bigg) \frac{\beta}{2} - \frac{\beta^2}{2} -\frac{1}{m/n}\cdot \frac{\gamma \lambda^2}{2(\beta+\lambda \gamma)}\cdot \frac{1}{n}\sum_{j=1}^n \bigg(\mu_{0,j}-\frac{\beta}{\lambda}g_j\bigg)^2.
\end{align}
Let $\psi$ be the `population version' of $\psi_n$ defined by
\begin{align}\label{def:ridge_psi_pop}
\psi(\beta,\gamma)
&\equiv \bigg(\frac{\sigma^2}{\gamma}+\gamma\bigg)\frac{\beta}{2}-\frac{\beta^2}{2}+\frac{1}{m/n}\E \min_w \bigg\{\frac{w^2}{2\gamma} \beta-\beta Z w+\frac{\lambda}{2}\Big((w+\Pi_{\mu_0})^2-\Pi_{\mu_0}^2\Big)\bigg\}\nonumber\\
&= \bigg(\frac{\sigma^2}{\gamma}+\gamma\bigg) \frac{\beta}{2} - \frac{\beta^2}{2}-\frac{1}{m/n}\cdot \frac{\gamma \lambda^2}{2(\beta+\lambda \gamma)} \E\bigg(\Pi_{\mu_0}-\frac{\beta}{\lambda} Z\bigg)^2,
\end{align}
where $\Pi_{\mu_0}\otimes Z \equiv (n^{-1}\sum_{j=1}^n \delta_{\mu_{0,j}})\otimes \mathcal{N}(0,1)$.

\begin{proposition}\label{prop:ridge_empirical_char}
	Suppose (R1)-(R2) hold, and the entries of $\xi_0$ are independent, mean $0$, variance $1$ and uniformly sub-Gaussian. There exist constants $C,K>0$ depending only on $\sigma,\lambda,\tau,M_2$ such that the following hold with probability at least $1-C n^{-100}$.
	\begin{enumerate}
		\item The max-min problem $\max_{\beta\geq 0}\min_{\gamma>0} \psi_n(\beta,\gamma)$ admits a unique saddle point $(\beta_{\ast,n},\gamma_{*,n}) \in [1/K,K]^2$ that can be characterized by
		\begin{align*}
		\gamma_{*,n}^2& = \sigma^2 + \frac{1}{m/n}\cdot \Prob_n\bigg[\eta_2\bigg(\mu_0+\gamma_{*,n} \cdot g; \frac{\gamma_{*,n} \lambda }{\beta_{\ast,n} }\bigg)-\mu_0\bigg]^2,\\
		\beta_{\ast,n}& = \gamma_{*,n}-\frac{1}{m/n} \Prob_n \bigg\{g\cdot \bigg[\eta_2\bigg(\mu_0+\gamma_{*,n} \cdot g; \frac{\gamma_{*,n} \lambda }{\beta_{\ast,n} }\bigg)-\mu_0\bigg]\bigg\}.
		\end{align*}
		Here $\Prob_n\equiv n^{-1}\sum_{j=1}^n\delta_{(\mu_{0,j},g_j)}$. 
		\item There exists some constant $K= K(\sigma,\lambda,\tau,M_2)>1$ such that the max-min problem $\max_{\beta\geq 0}\min_{\gamma>0} \psi(\beta,\gamma)$ admits a unique saddle point $(\beta_\ast,\gamma_\ast) \in [1/K,K]^2$ that can be characterized by (\ref{eqn:ridge_fpe}).	
		\item $(\beta_{\ast,n},\gamma_{\ast,n})$ is close to $(\beta_{\ast},\gamma_{\ast})$ in the sense that
		\begin{align*}
		\abs{\beta_{\ast,n}-\beta_\ast}\vee \abs{\gamma_{*,n}-\gamma_\ast}\leq K r_n.
		\end{align*} 
		\item We have 
		\begin{align*}
		\max_{Q \in \{L,H(\cdot,G)\}}\bigabs{\min_{w \in \R^n} Q(w)- \psi(\beta_\ast,\gamma_\ast)}\leq K r_n.
		\end{align*}
		\item Let $w_{\ast,n}=\argmin_{w} L(w)$, and
		\begin{align*}
		\bar{w}_{\ast,n}&\equiv \frac{-\lambda \mu_0+ \beta_{\ast,n} g}{(\beta_{\ast,n}/\gamma_{\ast,n})+\lambda} = \eta_2\bigg(\mu_0+\gamma_{*,n}\cdot g; \frac{\gamma_{*,n} \lambda }{\beta_{\ast,n} }\bigg)-\mu_0.
		\end{align*}
		Then $\pnorm{w_{\ast,n}-\bar{w}_{\ast,n}}{}/\sqrt{m}\leq K r_n^{1/2}$. 
	\end{enumerate}
\end{proposition}
\begin{proof}
We introduce some further notation. Let
\begin{align}\label{ineq:ridge_psi_empirical_notation}
e_g\equiv \frac{\pnorm{g}{}}{\sqrt{n}},\quad e_h\equiv \frac{\pnorm{h}{}}{\sqrt{m}},\quad  \Delta_m\equiv\frac{h^\top \xi}{m},\quad  \sigma_m^2\equiv \frac{\pnorm{\xi}{}^2}{m}.
\end{align}
Consider the event 
\begin{align}\label{ineq:ridge_psi_empirical_E}
E&\equiv \big\{\abs{e_h-1}\vee \abs{e_g-1}\vee \abs{\Delta_m}\vee \abs{\sigma_m^2-\sigma^2}\leq K_0r_n\big\}\nonumber\\
&\qquad \cap \big\{\abs{\iprod{\mu_0}{g}}\leq K_0 \pnorm{\mu_0}{} \sqrt{\log n}  \big\},
\end{align}
where $K_0>0$ is a large enough constant so that $\Prob(E)\geq 1- Cn^{-100}$. 

\noindent (\textbf{Preliminary Step 1}). We rewrite the Gordon cost that suits our purposes: with $\mathsf{f}(w)\equiv (\lambda/2m)\big(\pnorm{w+\mu_0}{}^2-\pnorm{\mu_0}{}^2\big)$, 
\begin{align}\label{ineq:ridge_psi_empirical_pm_1}
&L(w) = \max_u \ell (w,u)\nonumber\\
&= \max_{\beta> 0} \max_{\pnorm{u}{}=\beta }\bigg\{-\frac{1}{m^{3/2}} \pnorm{u}{}g^\top w+\frac{\pnorm{w}{}}{m^{1/2}}\cdot \frac{h^\top u}{m}-\frac{1}{m}u^\top \xi-\frac{1}{2m} \pnorm{u}{}^2 + \mathsf{f}(w) \bigg\}\nonumber\\
& = \max_{\beta> 0} \bigg\{ \frac{\beta}{\sqrt{m}}\cdot \bigg(\biggpnorm{\frac{ \pnorm{w}{}}{m^{1/2}}\cdot \frac{h}{\sqrt{m}}-\frac{\xi}{\sqrt{m}} }{}  - \frac{g^\top w}{m} \bigg)- \frac{(\beta/\sqrt{m})^2}{2}   +\mathsf{f}(w)\bigg\}.
\end{align}
Using $\pnorm{t}{}=\min_{\gamma} (\pnorm{t}{}^2/\gamma+\gamma)/2$ and replacing $\beta/\sqrt{m}$ by $\beta$, we may write $L(w)$ as
\begin{align}\label{ineq:ridge_psi_empirical_pm_1.1}
L(w)
& = \max_\beta \min_\gamma\bigg\{ \bigg(\frac{\sigma_m^2}{\gamma}+\gamma\bigg) \frac{\beta}{2} - \frac{\beta^2}{2}\nonumber\\
&\qquad +\bigg[\frac{\beta}{2\gamma}\frac{\pnorm{w}{}^2}{m}e_h^2 - \beta\cdot \frac{g^\top w}{m}+\mathsf{f}(w)\bigg]- \frac{\beta}{\gamma}\frac{\pnorm{w}{}}{\sqrt{m}} \Delta_m\bigg\}.
\end{align}
Now for $L_w>0$ to be determined later on, let
\begin{align}\label{ineq:ridge_psi_empirical_pm_1.2}
&\bar{\psi}_n^\pm(w,\beta,\gamma;L_w) \equiv  \bigg(\frac{(\sigma_m^\pm)^2}{\gamma}+\gamma\bigg) \frac{\beta}{2} - \frac{\beta^2}{2} +\bigg[\frac{\beta}{2\gamma}\frac{\pnorm{w}{}^2}{m}e_h^2 - \beta\cdot \frac{g^\top w}{m}+\mathsf{f}(w)\bigg],
\end{align}
where $(\sigma_m^\pm)^2\equiv (\sigma_m^2\pm 2L_w  \abs{\Delta_m})_+$. 
Further let 
\begin{align}\label{def:L_pm_ridge}
L^\pm(w)&\equiv \max_\beta \min_\gamma \bar{\psi}_n^\pm(w,\beta,\gamma;L_w) = \bigg(\sqrt{ \frac{\pnorm{w}{}^2}{m}e_h^2+(\sigma_m^\pm)^2  }-\frac{g^\top w}{m}\bigg)_+ +\mathsf{f}(w).
\end{align}
As the map $w\mapsto \sqrt{a\pnorm{w}{}^2+b}$ is convex for $a,b\geq 0$, $L^\pm$ is $1/m$-strongly convex. Furthermore, $L^-\leq L\leq L^+$ with $\pnorm{L^+-L^-}{\infty}\leq \abs{\sigma_m^+-\sigma_m^-}\leq 2L_w \abs{\Delta_m}/\sigma_m$.

\noindent (\textbf{Preliminary Step 2}). We will prove localization of $w$. We will show that there exists some constant $K>0$ depending on $(\sigma,\lambda,\tau,M_2)$ such that the global minimizers 
$w_{\ast,n}=\argmin_w L(w)$, $w_{\ast,n}^\pm=\argmin_w \max_\beta \min_\gamma \bar{\psi}_n^\pm(w,\beta,\gamma;L_w)$ satisfy
\begin{align*}
m^{-1/2}\max\{\pnorm{w_{\ast,n}}{},\pnorm{w_{\ast,n}^\pm}{}\}\leq K
\end{align*}
with probability at least $1-Cn^{-100}$. We will only prove this localization claim for $w_{\ast,n}$ as the claims for $w_{\ast,n}^\pm$ follow from the same arguments. First note that by (\ref{ineq:ridge_psi_empirical_pm_1}), $
\min_w L(w)\leq L(0) = \max_\beta \big\{ \sigma_m \beta - \beta^2/2\big\} = {\sigma_m^2}/{2}$. 
On the other hand, using (\ref{ineq:ridge_psi_empirical_pm_1}) again, for any $w\in \R^n$ with $\pnorm{w}{}/\sqrt{m}=\alpha$, we have $L(w)\geq \mathsf{f}(w)\geq  (\lambda/2)\big(\alpha^2-2(\pnorm{\mu_0}{}/\sqrt{m})\alpha\big)$. Combining the two inequalities,
\begin{align*}
\frac{\lambda}{2}\bigg(\alpha^2-2\frac{\pnorm{\mu_0}{}}{\sqrt{m}}\alpha\bigg)\leq \frac{\sigma_m^2}{2}.
\end{align*}
Solving the above inequality gives $\pnorm{w_{\ast,n}}{}/\sqrt{m}\leq K(1\vee \{\sigma_m/\sigma\})$ for some $K=K(\sigma,\tau,\lambda,M_2)>0$. Similarly $\pnorm{w_{\ast,n}^\pm}{}/\sqrt{m}\leq K(1\vee \{\sigma_m^\pm/\sigma\})$. Now we choose $L_w\equiv 2K$ and write $\bar{\psi}_n^\pm(w,\beta,\gamma;L_w)$ simply as $\bar{\psi}_n^\pm(w,\beta,\gamma)$. Then on $E$, $m^{-1/2}\max\{\pnorm{w_{\ast,n}}{},\pnorm{w_{\ast,n}^\pm}{}\}\leq 2K$ for $n$ large, and so
\begin{align}\label{ineq:ridge_psi_empirical_pm_2}
\min_{w \in \R^n} \max_{\beta>0} \min_{\gamma>0}\bar{\psi}_n^-(w,\beta,\gamma)\leq \min_{w \in \R^n} L(w)\leq \min_{w \in \R^n} \max_{\beta>0} \min_{\gamma>0}\bar{\psi}_n^+(w,\beta,\gamma).
\end{align}

\noindent (\textbf{Preliminary Step 3}). We continue rewriting Gordon's cost based on (\ref{ineq:ridge_psi_empirical_pm_2}). Clearly on an event with probability $1$, for any $w \in \R^n$, the saddle points to the max-min problems $\max_{\beta>0} \min_{\gamma>0}\bar{\psi}_n^\pm (w,\beta,\gamma)$ in (\ref{ineq:ridge_psi_empirical_pm_2}) do not reach boundary, so we may interchange the order of max-min by Sion's min-max theorem to obtain $\max_{\beta>0} \min_{\gamma>0}\bar{\psi}_n^\pm (w,\beta,\gamma)=\min_{\gamma>0}\max_{\beta>0} \bar{\psi}_n^\pm (w,\beta,\gamma)$ for all $w\in \R^n$ on a full probability event. Using Sion's min-max theorem again in view of the joint convexity of $\bar{\psi}_n^\pm$ in $(w,\gamma)$, we have  
\begin{align*}
\max_{\beta} \min_\gamma \bigg(\min_w \bar{\psi}_n^-(w,\beta,\gamma)\bigg)\leq \min_w L(w) \leq \max_\beta \min_\gamma \bigg(\min_w \bar{\psi}_n^+(w,\beta,\gamma)\bigg).
\end{align*}
The inner most minimum with respect to $w$ in $\psi_n^\pm$ takes the same form: 
\begin{align*}
&\min_w\bigg\{\frac{\beta}{2\gamma}\frac{ \pnorm{w}{}^2}{m} e_h^2 - \beta\cdot \frac{g^\top w}{m}+\mathsf{f}(w)\bigg\}\\
& = \frac{1}{m} \min_w \bigg\{ \frac{\beta e_h^2}{2\gamma}\pnorm{w-(\gamma/e_h^2)g}{}^2+ \frac{\lambda}{2}\pnorm{w+\mu_0}{}^2 \bigg\}- \frac{\beta \gamma}{2m e_h^2}\pnorm{g}{}^2-\frac{\lambda}{2m}\pnorm{\mu_0}{}^2\\
& = \frac{\lambda}{m} \env_{\pnorm{\cdot}{}^2/2}\bigg(\mu_0+\frac{\gamma}{e_h^2}  g; \frac{\gamma \lambda }{\beta e_h^2 }\bigg)- \frac{\beta \gamma}{2m e_h^2}\pnorm{g}{}^2-\frac{\lambda}{2m}\pnorm{\mu_0}{}^2,
\end{align*}
where $\env_{\pnorm{\cdot}{}^2/2}$ is the Moreau envelope associated with the function $x\mapsto \pnorm{x}{}^2/2$. Some further calculations lead to
\begin{align}\label{ineq:ridge_psi_empirical_pm_3}
&\min_w \bar{\psi}_n^\pm(w,\beta,\gamma) = \bigg(\frac{(\sigma_m^\pm)^2}{\gamma}+\gamma\bigg) \frac{\beta}{2} - \frac{\beta^2}{2} \nonumber\\
&\quad + \frac{\lambda}{m} \env_{\pnorm{\cdot}{}^2/2}\bigg(\mu_0+\frac{\gamma}{e_h^2} g; \frac{\gamma \lambda }{\beta e_h^2 }\bigg)- \frac{\beta \gamma}{2m e_h^2}\pnorm{g}{}^2-\frac{\lambda}{2m}\pnorm{\mu_0}{}^2\nonumber \\
& = \bigg(\frac{(\sigma_m^\pm)^2}{\gamma}+\gamma\bigg) \frac{\beta}{2}- \frac{\beta^2}{2} -\frac{1}{m/n}\cdot \frac{\gamma \lambda^2}{2(\beta e_h^2+\lambda \gamma)}\cdot \frac{1}{n}\sum_{j=1}^n \bigg(\mu_{0,j}-\frac{\beta}{\lambda}g_j\bigg)^2\nonumber\\
& \equiv \psi_n^{\pm}(\beta,\gamma),
\end{align}
where, for every $(\beta,\gamma)$, the minimum in the above display is attained at 
\begin{align*}
{w}&= \eta_2\bigg(\mu_0+\frac{\gamma}{e_h^2} \cdot g; \frac{\gamma \lambda }{\beta e_h^2 }\bigg)-\mu_0=\frac{-\lambda \mu_{0}+\beta g}{(\beta/\gamma)e_h^2+\lambda}.
\end{align*}
Summarizing, we have shown that with probability at least $1-Cn^{-100}$,
\begin{align}\label{ineq:ridge_psi_empirical_pm_4}
\max_{\beta>0} \min_{\gamma>0}\psi_n^-(\beta,\gamma)\leq  \min_{w\in \R^n} L(w)\leq \max_{\beta>0} \min_{\gamma>0}\psi_n^+(\beta,\gamma).
\end{align}

\noindent (\textbf{Step 1}). We prove the localization claim in (1). We shall show that the range of maximum and minimum in the max-min problems $\max_{\beta>0} \min_{\gamma>0}\psi_n^\pm(\beta,\gamma)$ and $\max_{\beta>0} \min_{\gamma>0}\psi_n(\beta,\gamma)$  can be localized to $[1/K,K]^2$ with probability at least $1-Cn^{-100}$. We will do so only for $\psi_n$, as the claims for $\psi_n^\pm$ are similar.

Let $\Psi_n(\beta)\equiv \min_\gamma \psi_n(\beta,\gamma)$. Note that map $\Psi_n(\cdot)= \min_\gamma \psi_n(\cdot,\gamma)$ is the infimum of $1$-strongly concave functions, so itself is also a.s. strongly concave. Furthermore, it is clear that a.s. $\Psi_n(\beta)\to -\infty$ as $\beta \to \infty$ and $\Psi_n(0) = -\lambda \pnorm{\mu_0}{}^2/(2m)$, so $\beta_{\ast,n}=\argmax_{\beta\geq 0} \Psi_n(\beta)$ is a.s. well-defined. 

First we obtain an upper bound for $\beta_{\ast,n}$ on $E$. As
\begin{align}\label{ineq:ridge_psi_empirical_pm_5}
\Psi_n(\beta)\leq  2\sigma \beta-\frac{\beta^2}{2},
\end{align}
so for $\beta\geq \bar{B}$ where $\bar{B}= \bar{B}(\sigma,\lambda,\tau,M_2)>0$, we have $\Psi_n(\beta)<\Psi_n(0)$. In other words, $\beta_{\ast,n}\leq  \bar{B}$ on $E$. Next we obtain a lower bound for $\beta_{\ast,n}$ on $E$:
\begin{align*}
\Psi_n(\beta)&\geq \min_{\gamma>0} \bigg\{\bigg(\frac{\sigma^2}{\gamma}+\gamma\bigg) \frac{\beta}{2}-\frac{\beta^2}{2}- \frac{\lambda}{2m}\bigg(\pnorm{\mu_0}{}^2-\frac{2\beta}{\lambda}\iprod{\mu_0}{g}+  \frac{\beta^2 \pnorm{g}{}^2}{\lambda^2}\bigg) \bigg\}\\
& = \Psi_n(0)+\beta \bigg(\sigma+\frac{ \iprod{\mu_0}{g} }{m}\bigg)-\frac{\beta^2}{2}\bigg(1+\frac{\pnorm{g}{}^2}{m\lambda} \bigg)\\
&\geq \Psi_n(0)+ \beta (\sigma/2)- K \beta^2.
\end{align*}
This means on $E$, $\Psi_n(\beta)>\Psi_n(0)$ for $\beta\leq \underline{B}=\underline{B}(\sigma,\lambda,\tau,M_2)$, and therefore $\beta_{\ast,n} \geq \underline{B}$. This proves that on the event $E$, $0<\underline{B}\leq \beta_{\ast,n}\leq \bar{B}<\infty$. 

For $\beta_{\ast,n} \in (0,\infty)$, clearly $\gamma_{\ast,n}=\argmin_{\gamma>0} \psi_n(\beta_{\ast,n},\gamma)$ is a.s. well-defined with $\gamma_{*,n} \in (0,\infty)$.  On $E$, by (\ref{ineq:ridge_psi_empirical_pm_5}),
\begin{align*}
C_1(\sigma,\bar{B})\geq \Psi_n(\beta_{\ast,n})&=\psi_n(\beta_{\ast,n},\gamma_{*,n})\geq \bigg(\frac{\sigma^2}{\gamma_{*,n}}+\gamma_{*,n}\bigg) \frac{\underline{B}}{2}-C_2(\bar{B},\tau,\lambda,M_2),
\end{align*}
which gives both lower and upper bounds for $\gamma_{*,n}$. This proves the desired high probability localization claim. Consequently, with probability at least $1- Cn^{-100}$, 
\begin{align}\label{ineq:ridge_psi_empirical_pm_6}
\bigabs{\min_{w\in \R^n} L(w)-\max_{\beta>0} \min_{\gamma>0}\psi_n(\beta,\gamma)}\leq K r_n.
\end{align}
The term $\max_{\beta>0} \min_{\gamma>0}\psi_n(\beta,\gamma)$ can be replaced by $\max_{\beta>0} \min_{\gamma>0}\psi_n^{\pm}(\beta,\gamma)$ in the above display.

\noindent (\textbf{Step 2}). We prove the fixed point equation characterization claim in (1). Recall the relationship (see e.g., \cite[Lemma D.1-(iii)]{thrampoulidis2018precise})
\begin{align}\label{ineq:ridge_psi_empirical_pm_7}
\nabla_z \env_{\pnorm{\cdot}{}^2/2}(z;\rho)& = \rho^{-1}\big(z-\prox_{\pnorm{\cdot}{}^2/2}(z;\rho)\big) = \rho^{-1}\big(z-\eta_2(z;\rho)\big),\nonumber\\
\frac{\partial}{\partial \rho} \env_{\pnorm{\cdot}{}^2/2}(z;\rho) &= -(2\rho^{2})^{-1}\pnorm{z-\eta_2(z;\rho)}{}^2.
\end{align}
We will evaluate the partial derivatives of $\psi_n$ using the representation
\begin{align*}
\psi_n(\beta,\gamma)=\bigg(\frac{\sigma^2}{\gamma}+\gamma\bigg) \frac{\beta}{2} - \frac{\beta^2}{2}  + \frac{\lambda}{m} \env_{\pnorm{\cdot}{}^2/2}\bigg(\mu_0+\gamma g; \frac{\gamma \lambda }{\beta}\bigg)- \frac{\beta \gamma}{2m}\pnorm{g}{}^2-\frac{\lambda}{2m}\pnorm{\mu_0}{}^2,
\end{align*}
and the identities in (\ref{ineq:ridge_psi_empirical_pm_7}). We first evaluate with some calculations that
\begin{align*}
\frac{\partial}{\partial \beta}\env_{\pnorm{\cdot}{}^2/2}\bigg(\mu_0+\gamma g; \frac{\gamma \lambda }{\beta}\bigg)& = \frac{1}{2\gamma \lambda}\biggpnorm{\mu_0+\gamma g-\eta_2\bigg(\mu_0+\gamma  g; \frac{\gamma \lambda }{\beta }\bigg)}{}^2,\\
\frac{\partial}{\partial \gamma}\env_{\pnorm{\cdot}{}^2/2}\bigg(\mu_0+\gamma g; \frac{\gamma \lambda }{\beta }\bigg)&=\frac{\beta}{2\lambda}
\pnorm{g}{}^2- \frac{\beta}{2\gamma^2 \lambda}\biggpnorm{\mu_0-\eta_2\bigg(\mu_0+\gamma g; \frac{\gamma \lambda }{\beta}\bigg)}{}^2.
\end{align*}
This means
\begin{align}\label{ineq:ridge_psi_empirical_2}
\frac{\partial \psi_n}{\partial \beta}
& = \bigg(\frac{\sigma^2}{\gamma}+\gamma\bigg)\frac{1}{2}-\beta+\frac{1}{2\gamma m} \biggpnorm{\mu_0-\eta_2\bigg(\mu_0+\gamma g; \frac{\gamma \lambda }{\beta }\bigg)}{}^2\nonumber\\
&\qquad\qquad  + \frac{1}{m}\iprod{g}{\mu_0-\eta_2\bigg(\mu_0+\gamma  g; \frac{\gamma \lambda }{\beta }\bigg)},
\end{align}
and
\begin{align}\label{ineq:ridge_psi_empirical_3}
\frac{\partial \psi_n}{\partial \gamma}
& = \bigg(-\frac{\sigma^2}{\gamma^2}+1\bigg) \frac{\beta}{2}- \frac{\beta}{2m\gamma^2} \biggpnorm{\mu_0-\eta_2\bigg(\mu_0+\gamma  g; \frac{\gamma \lambda }{\beta }\bigg)}{}^2.
\end{align}
Setting the RHS of (\ref{ineq:ridge_psi_empirical_2})-(\ref{ineq:ridge_psi_empirical_3}) to be 0 yields the fixed point equation.

\noindent (\textbf{Step 3}). We prove the claim in (2). The same proof as in Steps 1-2 can be used to reach the conclusion with the first equation in (\ref{eqn:ridge_fpe}) as stated, and the second equation in (\ref{eqn:ridge_fpe}) reading
\begin{align*}
\beta_\ast = \gamma_\ast-\frac{1}{m/n} \E \bigg\{ Z\cdot \bigg[\eta_2\bigg(\Pi_{\mu_0}+\gamma_\ast Z; \frac{\gamma_\ast \lambda }{\beta_\ast }\bigg)-\Pi_{\mu_0}\bigg]\bigg\}.
\end{align*}
Now we may apply Stein's identity to conclude the fixed point equation characterization at the population level.

\noindent (\textbf{Step 4}). We prove the claim in (3). Using the closed form expression for $\eta_2(z;\lambda)=z/(1+\lambda)$, on the event $E$, the second equation in (1) becomes
\begin{align*}
\beta_{\ast,n} &= \gamma_{*,n} - \frac{1}{m/n}\cdot \Prob_n \bigg[g \cdot \frac{ \gamma_{*,n} g}{1+ (\gamma_{*,n} \lambda/\beta_{\ast,n} ) }\bigg] + \mathcal{O}(r_n)\\
& = \gamma_{*,n} \bigg(1- \frac{1}{m/n}\cdot \frac{1}{1+\lambda (\gamma_{*,n} /\beta_{\ast,n})   } \bigg)+\mathcal{O}(r_n).
\end{align*}
Consequently, on the event $E$, 
\begin{align}\label{ineq:emp_pop_diff_1}
\frac{1}{\gamma_{*,n} /\beta_{\ast,n}} = 1- \frac{1}{m/n}\cdot \frac{1}{1+\lambda (\gamma_{*,n} /\beta_{\ast,n})   }+ \mathcal{O}(r_n). 
\end{align}
The above display is a quadratic equation in $\gamma_{*,n}/\beta_{\ast,n}$, and $\gamma_\ast/\beta_\ast$ verifies the equation exactly without $\mathcal{O}(r_n)$, so 
\begin{align}\label{ineq:emp_pop_diff_2}
\biggabs{\frac{\gamma_{*,n}}{\beta_{\ast,n}}-\frac{\gamma_\ast}{\beta_\ast}  } = \mathcal{O}(r_n). 
\end{align}
Now with $a_n \equiv  \gamma_{*,n}$, $b_n \equiv  \gamma_{*,n} \lambda/(\beta_{\ast,n})$ and $a\equiv \gamma_\ast$, $b\equiv \gamma_\ast \lambda/\beta_\ast$, the above display entails $\abs{b_n-b}=\mathcal{O}(r_n)$ and both $b_n,b$ are bounded away from $0$ and $\infty$. This means 
\begin{align*}
&\Prob_n\Big(\eta_2(\mu_0+a_n g;b_n)-\mu_0\Big)^2 
= \Prob_n \bigg(\frac{-\mu_0+(a_n/b) g}{1/b+1}\bigg)^2 +\mathcal{O}(r_n)\\
& = \Prob_n \bigg(\frac{-b\mu_0+a g}{1+b}\bigg)^2+\frac{a_n^2-a^2}{(1+b)^2}\Prob_n g^2+\bigo\big(r_n(1\vee\abs{a_n-a})\big)\\
& = \E \bigg(\frac{-b\mu_0+a Z}{1+b}\bigg)^2+\frac{a_n^2-a^2}{(1+b)^2}+\bigo\big(r_n(1\vee\abs{a_n-a})\big)\\
& = (m/n)\Big((\gamma_\ast)^2-\sigma^2\Big) +\frac{(\gamma_{*,n})^2-(\gamma_\ast)^2}{(1+b)^2}+\bigo\big(r_n(1\vee\abs{\gamma_{*,n}-\gamma_\ast})\big).
\end{align*}
Now using the first equation in (1), we arrive at
\begin{align}\label{ineq:emp_pop_diff_3}
(\gamma_{*,n})^2-(\gamma_\ast)^2&=\frac{ (\gamma_{*,n})^2-(\gamma_\ast)^2 }{(m/n)(1+b)^2}+ \bigo\big(r_n(1\vee\abs{\gamma_{*,n}-\gamma_\ast})\big).
\end{align}
By the population version of (\ref{ineq:emp_pop_diff_1}) and the definition of $b$, we have 
\begin{align*}
\frac{\lambda}{b} = 1-\frac{1}{(m/n)(1+b)}\,\Rightarrow \frac{1}{(m/n)(1+b)^2} = \frac{1}{1+b}\bigg(1-\frac{\lambda}{b}\bigg) \in [\epsilon,1-\epsilon]
\end{align*}
for some $\epsilon=\epsilon(\sigma,\lambda,\tau,M_2) \in (0,1/2)$. Now by  (\ref{ineq:emp_pop_diff_3}), we have
\begin{align*}
\epsilon\cdot \abs{(\gamma_{*,n})^2-(\gamma_\ast)^2}\leq \bigo\big(r_n(1\vee\abs{\gamma_{*,n}-\gamma_\ast})\big).
\end{align*}
Using the boundedness of $\gamma_{\ast,n},\gamma_\ast$ and solving the inequality yield that $\abs{\gamma_{*,n}-\gamma_\ast}=\bigo(r_n)$. The claimed bounds for $\abs{\beta_{\ast,n}-\beta_\ast}\vee \abs{\gamma_{*,n}-\gamma_\ast}$ follows by combining (\ref{ineq:emp_pop_diff_2}). 

\noindent (\textbf{Step 5}). We prove the claim in (4). By Step 4 and standard concentration arguments, with probability at least $1-Cn^{-100}$, we have $\abs{\psi_n(\beta_{\ast,n},\gamma_{\ast,n})-\psi(\beta_\ast,\gamma_\ast)}\leq K r_n$. Combined with (\ref{ineq:ridge_psi_empirical_pm_6}), we have proved the inequality in (4) for the Gordon cost $\min_w L(w)$. The inequality involving $\min_w H(w)$ follows further by an application of the CGMT; details are omitted.

\noindent (\textbf{Step 6}). We prove the claim in (5). Note that the claim (3) proved in Step 4 also holds when $(\beta_{\ast,n},\gamma_{\ast,n})$ is replaced by $(\beta_{\ast,n}^{\pm},\gamma_{\ast,n}^{\pm})$ defined as the saddle point for $\max_{\beta>0} \min_{\gamma>0}\psi_n^\pm(\beta,\gamma)$. So with probability at least $1-Cn^{-100}$, $\pnorm{w_{\ast,n}^\pm-\bar{w}_{\ast,n}}{}/\sqrt{m}\leq K r_n$ (and so $\pnorm{w_{\ast,n}^+-w_{\ast,n}^-}{}/\sqrt{m}\leq K r_n$). Recall $L^\pm$ defined in (\ref{def:L_pm_ridge}) are $1/m$-strongly convex with $L^-\leq L\leq L^+$, and $\pnorm{L^+-L^-}{\infty}\leq Kr_n$ with probability at least $1-Cn^{-100}$.  This means we may control the distance of the minimizer $w_{\ast,n}$ for $L$ and the minimizers $w_{\ast,n}^\pm$ of $L^\pm$ in that $\pnorm{w^\pm_{\ast,n}-w_{\ast,n}}{}/\sqrt{m}\leq K r_n^{1/2}$.  Combining we find $\pnorm{\bar{w}_{\ast,n}-w_{\ast,n}}{}/\sqrt{m}\leq K r_n^{1/2}$. This completes the proof for all the desired claims.
\end{proof}

\subsection{Proof of Theorem \ref{thm:ridge_dist}, distribution of $\hat{w}^{\ridge}_A$}

\begin{proposition}\label{prop:ridge_W2_gap}
Suppose (R1)-(R2) hold, and the entries of $\xi_0$ are independent, mean $0$, variance $1$ and uniformly sub-Gaussian. Let $\mathsf{g}:\R^n\to \R$ be $1$-Lipschitz. Then there exist constants $C,K>0$ depending only on $\sigma,\lambda,\tau,M_2$ such that for all $\epsilon \geq Kr_n$, with probability at least $1-C n^{-100}$,
\begin{align*}
\min_{w \in D_\epsilon(\mathsf{g})} H(w,G)\geq \psi(\beta_\ast,\gamma_\ast)+ K^{-1}\epsilon.
\end{align*} 
Here $
D_\epsilon(\mathsf{g})\equiv \big\{ w \in \R^n: \abs{ \mathsf{g}\big(w/\sqrt{n}\big) - \E \mathsf{g}\big(w_\ast/\sqrt{n}\big)}\geq \epsilon^{1/2} \big\}$, where $w_\ast$ is defined in (\ref{def:w_ridge}).
\end{proposition}
\begin{proof}
By Gaussian concentration for Lipschitz function of Gaussian random variables, with probability at least $1-C n^{-100}$, we have $\abs{\mathsf{g}\big(w_\ast/\sqrt{n}\big)-\E \mathsf{g}\big(w_\ast/\sqrt{n}\big)}\leq K r_n$. This means that with probability at least $1-C n^{-100}$, 
\begin{align*}
\epsilon^{1/2}&\leq \bigabs{\mathsf{g}\big(w/\sqrt{n}\big) - \E \mathsf{g}\big(w_\ast/\sqrt{n}\big)}\leq \bigabs{\mathsf{g}\big(w/\sqrt{n}\big) -  \mathsf{g}\big(w_\ast/\sqrt{n}\big)}+ Kr_n\\
&\leq n^{-1/2}\pnorm{w-w_\ast}{}+Kr_n\leq n^{-1/2}\pnorm{w-w_{\ast,n} }{}+Kr_n^{1/2}
\end{align*}
holds uniformly in  $w \in D_\epsilon(\mathsf{g})$. Here the last inequality follows as 
\begin{align*}
n^{-1/2}\pnorm{w-w_\ast}{} &\leq n^{-1/2}\big(\pnorm{w-w_{\ast,n}}{}+ \pnorm{w_\ast-\bar{w}_{\ast,n}}{}+\pnorm{\bar{w}_{\ast,n}-w_{\ast,n}}{}\big)\\
& \leq n^{1/2}\pnorm{w-w_{\ast,n}}{}+ K r_n^{1/2},
\end{align*}
by using Proposition \ref{prop:ridge_empirical_char}-(3)(5). So for any $\epsilon>K r_n$, with probability at least $1-C n^{-100}$,
\begin{align}\label{ineq:ridge_W2_gap_1}
n^{-1}\pnorm{w-w_{\ast,n}}{}^2\geq \epsilon/2,\quad \forall w \in D_\epsilon(\mathsf{g}).
\end{align}
As $L$ is globally $(\lambda/m)$-strongly convex, we conclude that for any $\epsilon \geq Kr_n$, with probability at least $1-C n^{-100}$,
\begin{align*}
 \min_{w \in D_\epsilon(\mathsf{g})} L(w) \geq \min_w L(w)+ K^{-1}\epsilon.
\end{align*} 
By Proposition \ref{prop:ridge_empirical_char}-(4), $\min_w L(w)$ in the above display may be replaced by the deterministic quantity $\psi(\beta_\ast,\gamma_\ast)$, by possibly changing $C,K>0$ accordingly. Now the claim follows by (a Ridge modified form of) the CGMT in the form given by \cite[Corollary 5.1-(1)]{miolane2021distribution}.
\end{proof}

\begin{proof}[Proof of Theorem \ref{thm:ridge_dist}, distribution of $\hat{w}^{\ridge}_A$]
By Proposition \ref{prop:ridge_empirical_char}-(4), we have
\begin{align*}
\Prob\bigg(\min_w H(w,G)\geq \psi(\beta_\ast,\gamma_\ast) +K r_n \bigg)\leq C n^{-100}. 
\end{align*}
By Proposition \ref{prop:ridge_W2_gap}, we have for any $\epsilon \geq K r_n$, 
\begin{align*}
\Prob\bigg(\min_{w \in D_\epsilon(\mathsf{g}) } H(w,G)\leq  \psi(\beta_\ast,\gamma_\ast)+ K^{-1}\epsilon\bigg)\leq C n^{-100}. 
\end{align*}
Now let $\mathcal{S}_n\equiv D_\epsilon(\mathsf{g})$, $z\equiv \psi(\beta_\ast,\gamma_\ast)$, $\rho_0\equiv \epsilon/K_1$ for some large $K_1>0$, we may apply Theorem \ref{thm:ridge_universality_generic} to conclude. The restriction $\epsilon \geq K r_n$ can be dropped for free as the probability bound becomes trivial otherwise. 
\end{proof}

\subsection{Proof of Theorem \ref{thm:ridge_dist}, distribution of $\hat{r}^{\ridge}_A$}

For a general design matrix $A$, let
\begin{align}\label{def:h_ridge_general}
h(w,u;A)& \equiv \frac{1}{m} u^\top A w-\frac{1}{m} u^\top\xi-\frac{1}{2m} \pnorm{u}{}^2 + \frac{\lambda}{2m} \big(\pnorm{w+\mu_0}{}^2-\pnorm{w}{}^2\big).
\end{align}
It is easy to see that
\begin{align}\label{def:ridge_u}
\hat{u}\equiv \hat{u}_A\equiv \argmax_{u \in \R^m} \min_{w \in \R^n} h(w,u;A) = A \hat{w}-\xi.
\end{align}
We define
\begin{align}\label{def:u_ridge}
u_\ast\equiv 
 \frac{\beta_\ast}{\gamma_\ast} \bigg(\sqrt{ \gamma_\ast^2-\sigma^2}\cdot h-\xi\bigg),
\end{align}
as the `population version' of $\hat{u}$ in the Gordon problem. Recall $\ell(\cdot,\cdot)$ defined in (\ref{def:h_l_ridge}), and $w_\ast$ defined in (\ref{def:w_ridge}).

\begin{proposition}\label{prop:ridge_u_map}
Suppose (R1)-(R2) hold, and the entries of $\xi_0$ are independent, mean $0$, variance $1$ and uniformly sub-Gaussian. There exist constants $C,K>0$ depending only on $\sigma,\lambda,\tau, M_2$ such that the following hold with probability $1-Cn^{-100}$. 
\begin{enumerate}
	\item $u\mapsto \ell(w_\ast,u)$ is $1/m$-strongly concave with a unique maximizer $u_{\ast,n}$:
	\begin{align*}
	u_{\ast,n}& =\bigg[1-\frac{(g^\top w_\ast)/m^{1/2}}{ \pnorm{(\pnorm{w_\ast}{}/m^{1/2})h-\xi}{}  }\bigg]\cdot \Big[ (\pnorm{w_\ast}{}/m^{1/2})h-\xi \Big]. 
	\end{align*}
	\item $\abs{\max_u \ell(w_\ast,u)-\psi(\beta_\ast,\gamma_\ast)}\leq Kr_n^{1/2}$.
	\item $m^{-1/2}\pnorm{u_{\ast,n}-u_\ast}{}\leq Kr_n$.
\end{enumerate}
\end{proposition}

We need a simple lemma before the proof of Proposition \ref{prop:ridge_u_map}.
\begin{lemma}\label{lem:concentration_w_ridge}
	Suppose (R1)-(R2) hold. Recall $w_\ast$ defined in (\ref{def:w_ridge}). Then there exist constants $C,K>0$ depending only on $\sigma,\lambda,\tau, M_2$ such that with probability at least $1- Cn^{-100}$, the term
	\begin{align*}
	n^{-1}\max\bigg\{\bigabs{g^\top w_\ast-\E g^\top w_\ast}, \bigabs{\pnorm{ w_\ast}{}^2-\E \pnorm{w_\ast}{}^2  }, \bigabs{\pnorm{ w_\ast+\mu_0}{}^2-\E \pnorm{w_\ast+\mu_0}{}^2  }\bigg\}
	\end{align*}
	is bounded by $K r_n$. 
\end{lemma}
\begin{proof}
	Note that $
	w_\ast= (-\lambda_\ast \mu_0+ \beta_\ast g)/\big((\beta_\ast/\gamma_\ast)+\lambda\big)$, so the concentration properties follows from standard arguments along with the boundedness of $\beta_\ast,\gamma_\ast$ proved in Proposition \ref{prop:ridge_empirical_char}-(2). 
\end{proof}

\begin{proof}[Proof of Proposition \ref{prop:ridge_u_map}]
\noindent (1). By Lemma \ref{lem:concentration_w_ridge}, with probability at least $1- Cn^{-100}$, $
n^{-1} g^\top w_\ast \geq  n^{-1}\E g^\top w_\ast-Kr_n \geq 1/K >0$ 
for $n$ large. This means that $u\mapsto \ell(w_\ast,u)$ is the sum of a $1/m$-strongly concave function and a concave function, and therefore again a $1/m$-strongly concave function. The desired expression for $u_{\ast,n}$ follows from the calculations in (\ref{ineq:ridge_psi_empirical_pm_1}), which gives 
\begin{align*}
\frac{\beta(w_\ast)}{\sqrt{m}}&=\frac{\pnorm{u_{\ast,n}}{}}{\sqrt{m}}=\biggpnorm{\frac{ \pnorm{w_\ast}{}}{m^{1/2}}\cdot \frac{h}{\sqrt{m}}-\frac{\xi}{\sqrt{m}} }{}  - \frac{g^\top w_\ast}{m},
\end{align*}
and
\begin{align*}
u_{\ast,n}& =\beta(w_\ast)\cdot \frac{ (\pnorm{w_\ast}{}/m^{1/2})h-\xi }{  \pnorm{(\pnorm{w_\ast}{}/m^{1/2})h-\xi}{}  }\\
&=\bigg[1-\frac{(g^\top w_\ast)/m^{1/2}}{ \pnorm{(\pnorm{w_\ast}{}/m^{1/2})h-\xi}{}  }\bigg]\cdot \Big[ (\pnorm{w_\ast}{}/m^{1/2})h-\xi \Big],
\end{align*}
as claimed.

\noindent (2). Note that from the calculations in (\ref{ineq:ridge_psi_empirical_pm_1}), both the range of maximum over $u$ in $\min_w L(w)=L(w_{\ast,n})=\max_u \ell(w_{\ast,n},u)$ and $L(w_{\ast})=\max_u \ell(w_{\ast},u)$ can be restricted to $\{\pnorm{u}{}^2/m\leq K\}$ for some large enough $K>0$ on an event $E_1$ with probability at least $1- Cn^{-100}$.  Now on the intersection of $E_1$, the event $\{\pnorm{g}{}\vee \pnorm{h}{}\leq K n^{1/2}\}$ and the event on which Proposition \ref{prop:ridge_empirical_char} is valid---which holds with probability at least $1- Cn^{-100}$---we have
\begin{align*}
&L(w_\ast)-\min_w L(w) = \max_{\pnorm{u}{}^2/m\leq K} \ell(w_\ast,u)-\max_{\pnorm{u}{}^2/m\leq K} \ell(w_{\ast,n},u)\\
&\leq K\max_{\pnorm{u}{}^2/m\leq K}\bigg\{m^{-3/2}\pnorm{u}{}\pnorm{g}{}\cdot \pnorm{w_\ast-w_{\ast,n}}{}+ m^{-1}\pnorm{h}{}\pnorm{u}{}\cdot m^{-1/2}\bigabs{\pnorm{w_\ast}{}-\pnorm{w_{\ast,n}}{}}\\
&\qquad \qquad + m^{-1}\bigabs{\pnorm{w_\ast+\mu_0}{}^2-\pnorm{w_{\ast,n}+\mu_0}{}^2} \bigg\}\\
&\leq K\cdot m^{-1/2}\pnorm{w_\ast-w_{\ast,n}}{}\leq Km^{-1/2}\big(\pnorm{w_\ast-\bar{w}_{\ast,n}}{}+\pnorm{w_{\ast,n}-\bar{w}_{\ast,n}}{}\big)\leq K r_n^{1/2}. 
\end{align*}
The last inequality follows from the explicit formula for $w_\ast,\bar{w}_{\ast,n}$ and Proposition \ref{prop:ridge_empirical_char}-(5). Now using Proposition \ref{prop:ridge_empirical_char}-(4) yields the claim.

\noindent (3). The claim follows by using Lemma \ref{lem:concentration_w_ridge} and the fact that $m^{-1}\E \pnorm{w_\ast}{}^2=\gamma_\ast^2-\sigma^2$ and $m^{-1}\E g^\top w_\ast =\gamma_\ast-\beta_\ast$ (this form can be most easily seen by that of the fixed point equations defined in Lemma \ref{prop:ridge_empirical_char}).
\end{proof}

\begin{proof}[Proof of Theorem \ref{thm:ridge_dist}: distribution of $\hat{r}^{\ridge}_A$]
	Without loss of generality, we assume $\pnorm{\mu_0}{\infty}\geq 1$. Recall $r_n=\sqrt{\log n/n}$ and $s_n\equiv n^{-1/6}\log^2 n$. As $\hat{u}= A\hat{\mu}-Y=-\hat{r}$ by (\ref{def:ridge_u}), we only need to study $\hat{u}$. Fix any $\epsilon>0$, and any $\mathsf{h}:\R^m \to \R$, let
	\begin{align*}
	D_\epsilon\equiv D_\epsilon(\mathsf{h})&\equiv \bigg\{ u \in \R^m : \bigabs{ \mathsf{h}(u/\sqrt{n}) - \E_h \mathsf{h}(u_\ast/\sqrt{n})}\geq \epsilon^{1/2} \bigg\},
	\end{align*}
	where recall $u_\ast$ is defined in (\ref{def:u_ridge}).
	Consider the event $E_1$ on which $\pnorm{\hat{u}}{\infty}\leq K(\sqrt{\log n}+\pnorm{\mu_0}{\infty})\equiv L_n$ (which is $\leq n$ for $n$ large), and $E_{2}$ on which $\pnorm{\hat{w}}{\infty}\leq K\sqrt{\log n}$ for some large enough $K>0$. By Proposition \ref{prop:ridge_risk}-(2)(3), $\Prob(E_1)\wedge \Prob(E_{2})\geq 1-Cn^{-100}$. Let
	\begin{align}\label{ineq:ridge_u_dist_1}
	\rho_0\equiv \epsilon/K,\quad z_0\equiv \psi(\beta_\ast,\gamma_\ast)-3\rho_0
	\end{align}
	for some large enough $K>0$. Recall $h(\cdot,\cdot;A)$ defined in (\ref{def:h_ridge_general}). Note that the set inclusion
	\begin{align*}
	& \bigg\{\max_{u \in [-L_n,L_n]^m}\min_{w \in \R^n} h(w,u;A)\geq z_0+6\rho_0\bigg\}\\
	&\quad \cap \bigg\{\max_{u \in [-L_n,L_n]^m \cap D_\epsilon}\min_{w \in \R^n} h(w,u;A)\leq z_0+3\rho_0\bigg\} \cap E_1 \subset \big\{\hat{u} \notin D_\epsilon\big\}
	\end{align*}
	implies that
	\begin{align}\label{ineq:ridge_u_dist_2}
	\Prob\big(\hat{u} \in D_\epsilon\big)&\leq \Prob\bigg(\max_{u \in [-L_n,L_n]^m}\min_{w \in \R^n} h(w,u;A)< z_0+6\rho_0\bigg)\nonumber\\
	&\qquad + \Prob\bigg(\max_{u \in [-L_n,L_n]^m \cap D_\epsilon}\min_{w \in \R^n} h(w,u;A)> z_0+3\rho_0\bigg)+\Prob(E_1^c)\nonumber\\
	&\equiv \mathfrak{p}_1+\mathfrak{p}_2+\Prob(E_1^c).
	\end{align}
	For $\mathfrak{p}_1$, on the event $E_1$, $\max_{u \in [-L_n,L_n]^m}$ can be replaced by the global maximum, and we may then exchange the order of max and min to obtain
	\begin{align*}
	\mathfrak{p}_1&\leq  \Prob\bigg(\min_{w \in \R^n} H(w,A)< z_0+6\rho_0\bigg)+ \Prob(E_1^c)\\
	&\leq \Prob\bigg(\min_{w \in [-L_n,L_n]^n} H(w,A)< z_0+6\rho_0\bigg)+ \Prob(E_1^c)+\Prob(E_{2}^c). 
	\end{align*}
	Now apply Theorem \ref{thm:universality_smooth} for the first term in the above display, we have 
	\begin{align*}
	\mathfrak{p}_1&\leq  \Prob\bigg(\min_{w \in [-L_n,L_n]^n} H(w,G)< z_0+2\rho_0\bigg)+ C\cdot\pnorm{\mu_0}{\infty}^2(1\vee \rho_0^{-3})s_n\\
	&\leq \Prob\bigg(\min_{w \in \R^n} H(w,G)< \psi(\beta_\ast,\gamma_\ast)-\rho_0\bigg)+ C\cdot \pnorm{\mu_0}{\infty}^2(1\vee \rho_0^{-3})s_n.
	\end{align*}
	By (\ref{ineq:ridge_u_dist_1}) and Proposition \ref{prop:ridge_empirical_char}-(4), for $\epsilon>Kr_n$, 
	\begin{align}\label{ineq:ridge_u_dist_3}
	\mathfrak{p}_1\leq  C \cdot \pnorm{\mu_0}{\infty}^2(1\vee \rho_0^{-3})s_n.
	\end{align}
	Next we handle $\mathfrak{p}_2$. To use Corollary \ref{cor:min_max_universality}, note that here $L_u=L_w\equiv L_n=K(\sqrt{\log n}+\pnorm{\mu_0}{\infty})$ (which is $\leq n$ for $n$ large), and the function $Q_n(u,w)$ is defined via
	\begin{align*}
	Q_n(u,w)& = -\frac{1}{m} u^\top\xi-\frac{1}{2m} \pnorm{u}{}^2 + \frac{\lambda}{2m} \big(\pnorm{w+\mu_0}{}^2-\pnorm{w}{}^2\big).
	\end{align*}
	Then $\mathscr{M}_{Q_n}(L_n,\delta)\leq K(L_n+\pnorm{\xi}{1}/m)\delta\leq n\delta$ on an event with probability at least $1-Cn^{-100}$, and therefore applying Corollary \ref{cor:min_max_universality} conditionally on $\xi$ first and then taking expectation yield that
	\begin{align*}
	\mathfrak{p}_2&\leq  \Prob\bigg(\max_{u \in [-L_n,L_n]^m \cap D_\epsilon}\min_{w \in [-L_n,L_n]^n} h(w,u;A)> z_0+3\rho_0\bigg)\\
	&\leq \Prob\bigg(\max_{u \in [-L_n,L_n]^m \cap D_\epsilon}\min_{w \in [-L_n,L_n]^n} h(w,u;G)> z_0+\rho_0\bigg)+ C\cdot \pnorm{\mu_0}{\infty}^2 (1\vee\rho_0^{-3}) s_n\\
	&\leq 2\Prob\bigg(\max_{u \in [-L_n,L_n]^m \cap D_\epsilon}\min_{w \in [-L_n,L_n]^n} {\ell}(w,u)> \psi(\beta_\ast,\gamma_\ast)-2\rho_0\bigg)+ C\cdot \pnorm{\mu_0}{\infty}^2 (1\vee\rho_0^{-3}) s_n.
	\end{align*}
	Here the last inequality follows by the CGMT (for max-min and inequality $>$, the set in the max only need be closed, and again we first condition on $\xi$ and then take expectation) and uses the definition of $z_0$ in (\ref{ineq:ridge_u_dist_1}). Recall the definition of $w_\ast$ in (\ref{def:w_ridge}). Clearly $\pnorm{w_\ast}{\infty}\leq L_n$ with probability at least $1- C n^{-100}$, so we may continue bounding the above display as follows:
	\begin{align}\label{ineq:ridge_u_dist_4}
	\mathfrak{p}_2&\leq 2\Prob\bigg(\max_{u \in [-L_n,L_n]^m \cap D_\epsilon} {\ell}(w_\ast,u)> \psi(\beta_\ast,\gamma_\ast)-2\rho_0\bigg)   + C \pnorm{\mu_0}{\infty}^2(1\vee\rho_0^{-3}) s_n+ C n^{-100}\nonumber\\
	&\leq 2\Prob\bigg(\max_{u \in  D_\epsilon} {\ell}(w_\ast,u)> \psi(\beta_\ast,\gamma_\ast)-2\rho_0\bigg)+ C \pnorm{\mu_0}{\infty}^2(1\vee\rho_0^{-3}) s_n .
	\end{align}
	By Proposition \ref{prop:ridge_u_map}, with probability at least $1- Cn^{-100}$, (i) the function $u\mapsto {\ell}(w_\ast,u)$ is $1/m$-strongly concave with a unique minimizer $u_{\ast,n}$, (ii) $\abs{\max_u {\ell}(w_\ast,u)- \psi(\beta_\ast,\gamma_\ast)}\leq Kr_n^{1/2}$, and (iii) $m^{-1/2}\pnorm{u_{\ast,n}-u_\ast}{}\leq Kr_n$, where recall $u_\ast$ is defined in (\ref{def:u_ridge}). Now using the Lipschitz property of $\mathsf{h}$ and Gaussian concentration conditionally on $\xi$ (in similar spirit to the argument in (\ref{ineq:ridge_W2_gap_1})), for any $\epsilon \geq K r_n^2$, with unconditional probability at least $1-C n^{-100}$, 
	\begin{align*}
	m^{-1}\pnorm{u-u_\ast}{}^2\geq \epsilon/2,\quad \forall u \in D_\epsilon. 
	\end{align*}
	Combined with (iii) above, we conclude that for any $\epsilon \geq Kr_n^2$, with unconditional probability at least $1- C n^{-100}$, 
	\begin{align*}
	m^{-1}\pnorm{u-u_{\ast,n}}{}^2\geq \epsilon/4, \quad \forall u \in D_\epsilon. 
	\end{align*}
	Using the high probability strong concavity of ${\ell}(w_\ast,\cdot)$, we now see that for all $\epsilon\geq Kr_n^2$, with probability at least $1- C n^{-100}$, 
	\begin{align*}
	\max_{u \in D_\epsilon} {\ell}(w_\ast,u)\leq \max_{u} {\ell}(w_\ast,u) - 2\epsilon/K\leq \psi(\beta_\ast,\gamma_\ast)-\epsilon/K.
	\end{align*}
	Combined with (\ref{ineq:ridge_u_dist_4}) and the definition of $\rho_0$ in (\ref{ineq:ridge_u_dist_1}), for all $\epsilon\geq Kr_n^2$,
	\begin{align}\label{ineq:ridge_u_dist_5}
	\mathfrak{p}_2\leq  C \cdot \pnorm{\mu_0}{\infty}^2(1\vee \epsilon^{-3}) s_n.
	\end{align}
	Now combining (\ref{ineq:ridge_u_dist_2}), (\ref{ineq:ridge_u_dist_3}) and (\ref{ineq:ridge_u_dist_5}), we find that for all $\epsilon\geq Kr_n$,
	\begin{align*}
	\Prob\big(\hat{u} \in D_\epsilon\big)&\leq C \cdot \pnorm{\mu_0}{\infty}^2(1\vee \epsilon^{-3}) s_n.
	\end{align*}
	The condition on $\epsilon$ can be dropped for free, so the proof is complete. 
\end{proof}

\section{Proofs for Section \ref{section:examples}: Lasso}\label{section:proof_lasso}

\noindent \emph{Convention}: We shall write
\begin{align*}
\bar{H}^{\lasso}(w) \equiv \bar{H}^{\lasso}(w,A)\equiv \bar{H}^{\lasso}(w,A,\xi),
\end{align*}
and will usually omit the superscript $(\cdot)^{\lasso}$ if no confusion could arise. We also usually omit the subscript $A$ that indicates the design matrix, but we will use the subscript $G$ for Gaussian designs when needed. Recall $r_n\equiv \sqrt{\log n/n}$ and $s_n\equiv n^{-1/6}\log^2 n$. The constants $C,K>0$, typically depending on $\sigma,\lambda,\tau,M_2$, will vary from line to line. All notation will be local in this section. 

\subsection{Proof of Proposition \ref{prop:lasso_risk}}

Define the (column) leave-one-out Lasso version
\begin{align}\label{def:lasso_loo}
\hat{w}^{(s)} \equiv \argmin_{w\in\R^n: w_s = 0} H(w,A,\xi) = \argmin_{w\in\R^n: w_s = 0} \frac{1}{2}\pnorm{Aw-\xi}{}^2 + \lambda \pnorm{w + \mu_0}{1},
\end{align}
and the (row) leave-one-out Lasso version
\begin{align}\label{def:lasso_loo_row}
\hat{w}^{[t]} \equiv  \argmin_{w\in\R^n} \frac{1}{2}\pnorm{A_{[-t]}w-\xi_{-t}}{}^2 + \lambda \pnorm{w + \mu_0}{1},
\end{align}
where $A_{[-t]} \in \R^{(m-1)\times n}$ is $A$ minus its $t$-th row. The following perturbation lemma controls the difference between the original Lasso solution $\hat{w}$ and its leave-one-out versions $\hat{w}^{(s)},\hat{w}^{[t]}$.
\begin{lemma}\label{lem:loo_perturb}
	Let the column and row leave-one-out Lasso versions $\hat{w}^{(s)},\hat{w}^{[t]}$ be defined as in (\ref{def:lasso_loo})-(\ref{def:lasso_loo_row}). Assume the same conditions as in Theorem \ref{thm:lasso_universality_generic}. Suppose $\lambda \geq K_0 (1\vee \sigma)\sqrt{(n/m)\log_+(n/m)}$ for some $K_0=K_0(M_2)>0$. Then with probability at least $1-C n^{-100}$, we have
	\begin{align*}
	\max_{s\in[n]} \pnorm{\hat{w}_{-s} - \hat{w}^{(s)}_{-s}}{}\vee \max_{t\in[m]} \pnorm{\hat{w} - \hat{w}^{[t]}}{}\leq K \sqrt{\log n},
	\end{align*}
	where $C,K>0$ depend on $(\sigma,\lambda,\tau,M_2)$. Furthermore,
	\begin{align*}
	\pnorm{A(\hat{w}^{[t]}-\hat{w}) }{}\leq 2 \big(\abs{\xi_t}+\abs{a_t^\top \hat{w}^{[t]}}\big)
	\end{align*}
	 holds. The lower bound on $\lambda$ can be eliminated when $m/n\geq 1+\epsilon$ for some $\epsilon>0$ at the cost of possibly enlarged constants $C,K$ depending further on $\epsilon$.
\end{lemma}
\begin{proof}
	The proof is inspired by that of \cite[Lemma 6.3]{javanmard2018debiasing}.

	\noindent (\textbf{Bounds for $\hat{w}^{(s)}$}).  Recall the definition of $H(w)=H(w,A)$ in (\ref{def:lasso_cost}), and $
	\hat{w} = \argmin_{w\in\R^{n}} H(w)$, $\hat{w}^{(s)} = \argmin_{w\in\R^{n}: w_s = 0} H(w)$. 
	The KKT condition for $\hat{w}^{(s)}$ yields that
	\begin{align}\label{ineq:perturb_basic_0}
	A_{-s}^\top(A_{-s}\hat{w}^{(s)}_{-s} - \xi) + \lambda v_{-s} = 0,
	\end{align}
	where $v_{-s}\in\R^{n-1}$ is a sub-gradient of $w_{-s}\mapsto\pnorm{w_{-s}}{1}$ at $(\hat{w}^{(s)} + \mu_0)_{-s}$. 
	On the other hand, for any $w_{-s} \in \R^{n-1}$, let $(0,w_{-s})$ denote the vector in $\R^n$ with its $s$-th entry equal to $0$ (the notation is slightly incorrect for $s\neq 1$, but for simplicity we shall abuse this notation or regard $s$ as $1$). 
	Then expanding the cost $\bar{H}((0,w_{-s}))$ around $\hat{w}^{(s)}$ yields that
	\begin{align*}
	&\bar{H}((0,w_{-s})) =\bar{H}(\hat{w}^{(s)}) + \frac{1}{2}\pnorm{A_{-s}(w_{-s} - \hat{w}^{(s)}_{-s})}{}^2\\
	& \qquad+ \iprod{A_{-s}\hat{w}^{(s)}_{-s}-\xi}{A_{-s}(w_{-s} - \hat{w}^{(s)}_{-s})} + \lambda\pnorm{(0, w_{-s}) + \mu_0}{1}- \lambda\pnorm{\hat{w}^{(s)} + \mu_0}{1}\\
	&\stackrel{(*)}{=}  \bar{H}(\hat{w}^{(s)}) + \frac{1}{2}\pnorm{A_{-s}(w_{-s} - \hat{w}^{(s)}_{-s})}{}^2\\
	&\qquad+ \lambda\pnorm{(0, w_{-s}) + \mu_0}{1}- \lambda\pnorm{\hat{w}^{(s)} + \mu_0}{1} - \lambda \iprod{v_{-s}}{w_{-s} - \hat{w}^{(s)}_{-s}}\\
	&\stackrel{(**)}{\geq}  \bar{H}(\hat{w}^{(s)}) + \frac{1}{2}\pnorm{A_{-s}(w_{-s} - \hat{w}^{(s)}_{-s})}{}^2.
	\end{align*}
	Here $(*)$ follows from the KKT condition of $\hat{w}^{(s)}$ in (\ref{ineq:perturb_basic_0}), and $(**)$ follows from the convexity of $\pnorm{\cdot}{1}$. Hence by choosing $w_{-s}$ to be $\hat{w}_{-s}$, we have
	\begin{align}\label{ineq:perturb_basic}
	\frac{1}{2}\pnorm{A_{-s}(\hat{w}_{-s} - \hat{w}^{(s)}_{-s})}{}^2 \leq \bar{H}\big((0,\hat{w}_{-s})\big) - \bar{H}(\hat{w}^{(s)}).
	\end{align}
	Next, for any $x \in \R$ and $w_{-s}\in\R^{n-1}$, 
	\begin{align*}
	\bar{H}((x,w_{-s})) &= \frac{1}{2}\pnorm{\xi - xA_s - A_{-s}w_{-s}}{}^2 + \lambda\pnorm{(x,w_{-s}) + \mu_0}{1}\\
	&= \lambda|x + (\mu_0)_s| + \frac{\mathsf{c}}{2}(x - \mathsf{u})^2 + \mathsf{b},
	\end{align*}
	where $\mathsf{c} \equiv \pnorm{A_s}{}^2$,
	\begin{align*}
	\mathsf{u} &\equiv \mathsf{u}(w_{-s}) \equiv \frac{\iprod{A_s}{\xi - A_{-s}w_{-s}}}{\pnorm{A_s}{}^2},\\
	\mathsf{b} &\equiv \mathsf{b}(w_{-s}) \equiv -\frac{\mathsf{c}}{2}\mathsf{u}^2 + H((0,w_{-s}))-\lambda \abs{(\mu_0)_s}.
	\end{align*}
	So we have
	\begin{align*}
	\bar{H}((x,\hat{w}_{-s})) = \lambda|x + (\mu_0)_s| + \frac{\mathsf{c}}{2}(x - \mathsf{u}_1)^2 + \mathsf{b}_1 \equiv f_1(x),\\
	\bar{H}((x,\hat{w}^{(s)}_{-s})) = \lambda|x + (\mu_0)_s| + \frac{\mathsf{c}}{2}(x - \mathsf{u}_2)^2 + \mathsf{b}_2 \equiv f_2(x),
	\end{align*}
	where $(\mathsf{u}_1, \mathsf{b}_1) = (\mathsf{u}(\hat{w}_{-s}), \mathsf{b}(\hat{w}_{-s}))$ and  $(\mathsf{u}_2, \mathsf{b}_2) = (\mathsf{u}(\hat{w}^{(s)}_{-s}), \mathsf{b}(\hat{w}^{(s)}_{-s}))$. 
	Since 
	\begin{align*}
	\min_{x \in \R} f_1(x)&=\min_{x\in\R} \bar{H}((x,\hat{w}_{-s})) = \min_{w\in\R^n} \bar{H}(w)\\
	& \leq \min_{x\in\R} \bar{H}((x,\hat{w}^{(s)}_{-s}))=\min_{x \in \R} f_2(x),
	\end{align*}
	by Lemma \ref{lem:compare_quad_form}, we have
	\begin{align*}
	&\bar{H}((0,\hat{w}_{-s})) - \bar{H}(\hat{w}^{(s)}) \leq \big(f_1(0) - \min_x f_1(x)\big) - \big(f_2(0) - \min_x f_2(x)\big)\\
	&\leq \frac{1}{2}\frac{\iprod{A_s}{A_{-s}(\hat{w}_{-s} - \hat{w}^{(s)}_{-s})}^2}{\pnorm{A_s}{}^2} + \Big(\big|\iprod{A_s}{\xi - A_{-s}\hat{w}^{(s)}_{-s}}\big| + \lambda\Big)\frac{\big|\iprod{A_s}{A_{-s}(\hat{w}_{-s} - \hat{w}^{(s)}_{-s})}\big|}{\pnorm{A_s}{}^2}.
	\end{align*}
	Let $P_{A_s} \equiv A_sA_s^\top/\pnorm{A_s}{}^2$ be the projection matrix onto $A_s$, and $P_{A_s}^\perp \equiv I_{m} - P_{A_s}$ be its orthogonal complement. Combining the above estimate with (\ref{ineq:perturb_basic}) yields that
	\begin{align*}
	\frac{1}{2}\pnorm{P_{A_s}^\perp\big(A_{-s}(\hat{w}_{-s} - \hat{w}^{(s)}_{-s})\big)}{}^2 &\leq \Big(\big|\iprod{A_s}{\xi - A_{-s}\hat{w}^{(s)}_{-s}}\big| + \lambda\Big)\frac{\big|\iprod{A_s}{A_{-s}(\hat{w}_{-s} - \hat{w}^{(s)}_{-s})}\big|}{\pnorm{A_s}{}^2}\\
	&\leq \Big(\big|\iprod{A_s}{\xi - A_{-s}\hat{w}^{(s)}_{-s}}\big| + \lambda\Big)\frac{\pnorm{A_{-s}}{\op}\pnorm{\hat{w}_{-s} - \hat{w}^{(s)}_{-s}}{}}{\pnorm{A_s}{}}.
	\end{align*}
	Note that the left hand side of the above display can be written as
	\begin{align*}
	P_{A_s}^\perp\big(A_{-s}(\hat{w}_{-s} - \hat{w}^{(s)}_{-s})\big) = A_s \cdot \bigg(-\frac{1}{\pnorm{A_s}{}^2}A_s^\top A_{-s}(\hat{w}_{-s} - \hat{w}^{(s)}_{-s})\bigg) + A_{-s} (\hat{w}_{-s} - \hat{w}^{(s)}_{-s}) \equiv A\bar{w}, 
	\end{align*}
	and by Lemma \ref{lem:lasso_sparsity}, $\bar{w}$ with $\bar{w}_{-s} = \hat{w}_{-s} - \hat{w}^{(s)}_{-s} = \hat{\mu}_{-s} - \hat{\mu}^{(s)}$ is at most $(2c_0m + 1)$-sparse for some sufficiently small $c_0 > 0$ (recall that $\hat{\mu}^{(s)} = \hat{w}^{(s)} + \mu_0$). Hence by the sparse eigenvalue Lemma \ref{lem:sparse_eigenvalue}, we have with high probability $\pnorm{A\bar{w}}{}^2 \geq \pnorm{\bar{w}}{}^2/2 \geq \pnorm{\hat{w}_{-s} - \hat{w}^{(s)}_{-s}}{}^2/2$. This leads to
	\begin{align}\label{ineq:perturb_basic_1}
	\pnorm{\hat{w}_{-s} - \hat{w}^{(s)}_{-s}}{}^2 &\lesssim  \Big(\big|\iprod{A_s}{\xi - A_{-s}\hat{w}^{(s)}_{-s}}\big| + \lambda\Big)\frac{\pnorm{A_{-s}}{\op}\pnorm{\hat{w}_{-s} - \hat{w}^{(s)}_{-s}}{}}{\pnorm{A_s}{}}\nonumber\\
	&\leq K \sqrt{\log n}\cdot \pnorm{\hat{w}_{-s} - \hat{w}^{(s)}_{-s}}{}.
	\end{align}
	Here in the last inequality we used: (i) $\big|\iprod{A_s}{\xi - A_{-s}\hat{w}^{(s)}_{-s}}\big| \leq K  \sqrt{\log n}$ with probability at least $1-C n^{-100}$, by the independence between $A_s$ and $(\xi, A_{-s}, \hat{w}^{(s)}_{-s})$; (ii) $\pnorm{A_{-s}}{\op}\lesssim 1$ and $\pnorm{A_s}{} \gtrsim 1$ with high probability. If $m/n\geq 1+\epsilon$ for some $\epsilon>0$, $\pnorm{A\bar{w}}{}^2\geq c_\epsilon \pnorm{\bar{w}}{}^2$ holds with probability at least $1-C_\epsilon n^{-100}$, so (\ref{ineq:perturb_basic_1}) holds by enlarging the constant that may depend further on $\epsilon$. 
	
	\noindent (\textbf{Bounds for $\hat{w}^{[t]}$}). Verbatim following the proof of Lemma \ref{lem:ridge_loo}-(2), we have 
	\begin{align*}
	\pnorm{A(\hat{w}^{[t]} - \hat{w})}{}\leq 2 \big(\abs{\xi_t}+\abs{a_t^\top \hat{w}^{[t]}}\big).
	\end{align*}
	By the independence between $a_t$ and $\hat{w}^{[t]}$, with probability at least $1-C n^{-100}$, uniformly in $t \in [m]$, $\pnorm{A(\hat{w}^{[t]} - \hat{w})}{}^2\leq K\sqrt{\log n}\cdot \pnorm{ \hat{w}^{[t]}-\hat{w} }{}$. Now apply the sparse eigenvalue argument as in (\ref{ineq:perturb_basic_1}) with the help of Lemma \ref{lem:lasso_sparsity} to conclude. 
\end{proof}

The following two lemmas are used in the proof of Lemma \ref{lem:loo_perturb} above. 

\begin{lemma}\label{lem:compare_quad_form}
	Fix $c>0, (\mu_0)_s, b_1,b_2 \in \R$. Let $f_k(x) \equiv (c/2)(x - u_k)^2 + \lambda |x + (\mu_0)_s| + b_k$ for $k=1,2$. Then 
	\begin{align*}
	\Big|\big(f_1(0) - \min_x f_1(x)\big) - \big(f_2(0) - \min_x f_2(x)\big)\Big| \leq \frac{c}{2}(u_1 - u_2)^2 + (c|u_2| + \lambda)|u_1 - u_2|.
	\end{align*}
\end{lemma}
\begin{proof}
	For any $\tau > 0$, let $H(\cdot;\tau)$ be the Huber function, i.e., $H(x;\tau) = x^2/(2\tau)$ for $|x|\leq \tau$ and $H(x;\tau) = |x| - \tau/2$ for $|x| \geq \tau$. Then $\min_x f_k(x) = \lambda H((\mu_0)_s + u_k;\lambda/c) + b_k$, hence
	\begin{align*}
	\Delta_k \equiv f_k(0) - \min_x f_k(x) = \frac{c}{2}u_k^2 + \lambda|(\mu_0)_s| - \lambda H((\mu_0)_s + u_k;\lambda/c).
	\end{align*}
	Using that $H(\cdot;\tau)$ is 1-Lipschitz, we have
	\begin{align*}
	|\Delta_1 - \Delta_2| \leq c|u_2||u_1 - u_2| + \frac{c}{2}(u_1 - u_2)^2 + \lambda |u_1 - u_2|,
	\end{align*}
	as desired. 
\end{proof}

\begin{lemma}\label{lem:lasso_sparsity}
	Assume the same conditions as in Theorem \ref{thm:lasso_universality_generic}.  Then there exists some $K=K(M_2)>0$ such that for any $c_0 \in (0,1)$, if the tuning parameter $\lambda$ is chosen such that
	\begin{align*}
	\lambda\geq K(1\vee \sigma)\sqrt{\bigg(\frac{n}{m}+1\bigg)\bigg(\frac{1}{c_0}+\log \bigg(\frac{en}{c_0m}\bigg)\bigg)},
	\end{align*}
	then $\hat{\mu} = \hat{w} + \mu_0$ is $(c_0 m)$-sparse with probability $1 - C\exp\big(-\frac{c_0\lambda^2m}{C\sigma^2}\big)$, where $C>0$ is universal.
\end{lemma}
\begin{proof}
	The KKT condition for Lasso yields that $
	A^\top(A\hat{w} - \xi) + \lambda v = 0$, 
	where $v$ is a sub-gradient of $\pnorm{\cdot}{1}$ at $\hat{\mu}=\hat{w}+\mu_0$.  Let $\hat{S}_\pm \equiv \{i\in[n]: (\hat{\mu}_i)_\pm > 0\}$, so that $\hat{\mu}$ has sparsity $|\hat{S}_+| + |\hat{S}_-|$. Let $z\equiv A^\top \xi\in\R^n$. Then we have
	\begin{align*}
	A_i^\top (A\hat{w}) &= z_i \pm \lambda, \quad i\in \hat{S}_\pm.
	\end{align*}
	Taking the square for both sides and summing over $i\in\hat{S}_+$ yield that
	\begin{align*}
	\sum_{i\in\hat{S}_+}z_i^2 + \lambda^2|\hat{S}_+| - 2\lambda\sum_{i\in\hat{S}_+}z_i &= (A\hat{w})^\top (A_{\hat{S}_+}A_{\hat{S}_+}^\top) (A\hat{w})\leq \pnorm{A\hat{w}}{}^2\pnorm{A_{\hat{S}_+}A_{\hat{S}_+}^\top}{\op},
	\end{align*}
	where $A_S$ denotes the columns of $A$ in $S$. By Lemma \ref{lem:sparse_eigenvalue}, we have 
	\begin{align*}
	\pnorm{A_{\hat{S}_+}A_{\hat{S}_+}^\top}{\op}=\pnorm{A_{\hat{S}_+}^\top A_{\hat{S}_+}}{\op} \leq K_2 (1 + |\hat{S}_+|\log(en/|\hat{S}_+|)/m)
	\end{align*}
	with probability $1 - \exp(-m/K_2)$ for some universal $K_2>0$. Using the prediction bound in (\ref{ineq:lasso_prediction}) below, there exists some $K_3=K_3(M_2)\geq 1$,
	\begin{align}\label{ineq:sparsity_bound}
	\lambda^2|\hat{S}_+| \leq 2\lambda\Big|\sum_{i\in\hat{S}_+}z_i\Big| + K_3\bigg(\sigma^2+\frac{n}{m}\bigg) \cdot m\cdot \Big(1 + \frac{|\hat{S}_+|\log(en/|\hat{S}_+|)}{m}\Big).
	\end{align}
	Fix any $c_0 > 0$, and let $E$ denote the event that (\ref{ineq:sparsity_bound}) holds. By choosing
	\begin{align}\label{def:lambda_lower}
	\lambda_0 = K\cdot \bigg\{K_3 \bigg(\sigma^2+\frac{n}{m}+1\bigg) \bigg(\frac{1}{c_0} +  \log\bigg(\frac{en}{c_0 m}\bigg)\bigg)\bigg\}^{1/2},
	\end{align}
	where $K>0$ is a large enough universal constant, the second term on the right most side of (\ref{ineq:sparsity_bound}) is bounded by $\lambda^2\abs{\hat{S}_+}/2$ for every $\lambda\geq \lambda_0$ on the event $\{|\hat{S}_+| \geq c_0m\}$. This means for every $\lambda\geq \lambda_0$, 
	\begin{align*}
	\Prob\Big(\Big\{|\hat{S}_+| \geq c_0m\Big\}\cap E\Big)
	&\leq \Prob\bigg(\Big\{|\hat{S}_+| \geq c_0m\Big\} \cap \bigg\{\lambda^2|\hat{S}_+|/2 \leq 2\lambda\Big|\sum_{i\in\hat{S}_+}z_i\Big|\bigg\}\bigg).
	\end{align*}
	Now using an easy union bound, the probability on the right hand side above can be further bounded by
	\begin{align*}
	&\sum_{s = \ceil{c_0m} }^n \sum_{S\subset[n]: |S| = s} \Prob\bigg(\Big|\sum_{i\in S}z_i\Big| \geq \lambda s/4\bigg)\leq 2 \sum_{s = \ceil{c_0m}}^n {n \choose s}\exp\Big(-\frac{\lambda^2(s/4)^2}{2s\sigma^2}\Big) \\
	&\leq 2 \sum_{s = \ceil{c_0m} }^n \exp\Big(s\log(en/s) - \frac{\lambda^2 s}{32\sigma^2}\Big) \leq 2\sum_{s = \ceil{c_0m} }^n \exp\Big(-\frac{\lambda^2 s}{64\sigma^2}\Big)  \leq 4\exp\Big(-\frac{c_0\lambda^2m}{64\sigma^2}\Big).
	\end{align*}
	A similar argument applies to $\hat{S}_-$. The proof is complete as $\Prob(E^c) \leq \exp(-Cm)$ for some universal $C > 0$ that can be assimilated into the above probability bound by adjusting constants.
\end{proof}

Now we are in position to prove Proposition \ref{prop:lasso_risk}.
\begin{proof}[Proof of Proposition \ref{prop:lasso_risk}]
	
	\noindent{(1). (Prediction risk)} By optimality of $\hat{w}$ versus $w = -\mu_0$, 
	\begin{align*}
	\frac{1}{2}\pnorm{A\hat{w} - \xi}{}^2 + \lambda\pnorm{\hat{w} + \mu_0}{1} \leq \frac{1}{2}\pnorm{A\mu_0+\xi}{}^2.
	\end{align*}
	This implies that with probability at least $1-Cn^{-100}$, 
	\begin{align}\label{ineq:lasso_prediction}
	\pnorm{A\hat{w}}{}^2 \leq 2\big(\pnorm{A\hat{w} - \xi}{}^2 + \pnorm{\xi}{}^2\big) \lesssim \pnorm{A\mu_0}{}^2+ \pnorm{\xi}{}^2 \leq K_1\cdot (n M_2+\sigma^2 m)
	\end{align}
	for some universal $K_1 >0$.

	\noindent{(2). ($\ell_\infty$ risk)}
	Recall that the Lasso solution is $\hat{\mu} = \hat{w} +\mu_0$, where $\hat{w}=\argmin_{w\in\R^n} H(w)$. The KKT condition of this optimization problem yields that $
	A^\top (A\hat{w} - \xi) + \lambda v = 0$, where $v$ is a sub-gradient of $\pnorm{\cdot}{1}$ at $\hat{w} + \mu_0$. Hence with $\hat{\Sigma} \equiv A^\top A\in\R^{n\times n}$ denoting the sample covariance, we have
	\begin{align*}
	\hat{\Sigma}\hat{w} = A^\top\xi - \lambda v.
	\end{align*}
	Fix $s\in[n]$. It is clear that $\pnorm{A^\top \xi}{\infty} \vee \pnorm{\lambda v}{\infty} \leq K \sqrt{\log n}$ with probability at least $1-Cn^{-100}$, so taking the $s$-th component of the above display yields that
	\begin{align*}
	\Big|\hat{\Sigma}_{s,s}\hat{w}_s + \hat{\Sigma}_{s,-s}\hat{w}_{-s}\Big| = \Big|\pnorm{A_s}{}^2\hat{w}_s + A_s^\top A_{-s} \hat{w}_{-s}\Big| \leq K \sqrt{\log n}.
	\end{align*}
	It is clear that $\inf_{s\in[n]}\pnorm{A_s}{}^2 \geq 1/2$ with probability at least $1-Cn^{-100}$, so 
	\begin{align}\label{eq:linfty_decomp}
	|\hat{w}_s|/K \leq \bigabs{A_s^\top A_{-s} \hat{w}_{-s}} + \sqrt{\log n}
	\end{align}
	with the prescribed probability. Recall the column leave-one-out Lasso version $\hat{w}^{(s)}$ defined in (\ref{def:lasso_loo}). Then
	\begin{align*}
	\bigabs{A_s^\top A_{-s} \hat{w}_{-s}}&\leq \bigabs{A_s^\top A_{-s} \hat{w}^{(s)}_{-s}} + \bigabs{A_s^\top A_{-s} (\hat{w}_{-s} - \hat{w}^{(s)}_{-s})}\\
	&\leq \bigabs{A_s^\top A_{-s} \hat{w}^{(s)}_{-s}} + \pnorm{A_s}{}\pnorm{A_{-s}}{\op}\pnorm{(\hat{w}_{-s} - \hat{w}^{(s)}_{-s})}{}\equiv \mathsf{w}_{1;s} + \mathsf{w}_{2;s}.
	\end{align*}
	For $\mathsf{w}_{1;s}$, using the independence between $A_s$ and $A_{-s}\hat{w}^{(s)}_{-s}$, and the prediction risk proved in (1), with probability at least $1-C n^{-100}$, uniformly in $s$ we have
	\begin{align*}
	\mathsf{w}_{1;s} \lesssim \sqrt{\log n}\cdot \frac{1}{\sqrt{n}}\pnorm{A_{-s}\hat{w}^{(s)}_{-s}}{} \leq K  \sqrt{\log n}.
	\end{align*}
	Moreover, it is easy to see that $\pnorm{A_s}{}\vee\pnorm{A_{-s}}{\op} \leq K$ holds with probability at least $1-Cn^{-100}$, and by Lemma \ref{lem:loo_perturb}, $\mathsf{w}_{2;s} \leq K \sqrt{\log n}$ holds uniformly in $s$ with probability at least $1-Cn^{-100}$. We therefore have the bound $ \sqrt{\log n}$ for the first term in (\ref{eq:linfty_decomp}).
	
	\noindent (3). (Prediction $\ell_\infty$ risk) The proof follows verbatim as that of Proposition \ref{prop:ridge_risk}-(3) upon using the row leave-one-out Lasso version defined in (\ref{def:lasso_loo_row}) with Lemma \ref{lem:loo_perturb}.	
\end{proof}

\subsection{Proof of Theorem \ref{thm:lasso_universality_generic}}

	The proof is almost the same as that of Theorem \ref{thm:ridge_universality_generic}, but now using Proposition \ref{prop:lasso_risk}, and then we may take $L_n=K\sqrt{\log n}$ and $\epsilon_n=C n^{-100}$. It is easy to verify (\ref{cond:f_moduli}) so we may apply Theorem \ref{thm:universality_reg} to conclude.\qed

\subsection{Relating the Gordon cost functions for Gaussian and non-Gaussian errors}

In the Gaussian design case, we write $A=G$. Let $h,\ell: \R^n\times \R^m\to \R$ be defined similarly as in (\ref{def:h_l_ridge}) with the penalty $(\pnorm{w+\mu_0}{}^2-\pnorm{\mu_0}{}^2)/2$ replaced by 
$(\pnorm{w+\mu_0}{1}-\pnorm{\mu_0}{1})$:
\begin{align}\label{def:h_l_lasso}
h(w,u)& = \frac{1}{m}u^\top G w-\frac{1}{m}u^\top \xi-\frac{1}{2m} \pnorm{u}{}^2 + \frac{\lambda}{m} \big(\pnorm{w+\mu_0}{1}-\pnorm{\mu_0}{1}\big),\nonumber\\
{\ell}(w,u)& = -\frac{1}{m^{3/2}} \pnorm{u}{}g^\top w +\frac{\pnorm{w}{}}{m^{1/2}}\cdot \frac{h^\top u}{m}-\frac{1}{m}u^\top \xi\nonumber\\
&\qquad\qquad -\frac{1}{2m} \pnorm{u}{}^2 + \frac{\lambda}{m}\big(\pnorm{w+\mu_0}{1}-\pnorm{\mu_0}{1}\big).
\end{align}
Here $g \in \R^n, h \in \R^m$ are independent standard Gaussian vectors. The Lasso cost function in the Gaussian design case can be realized as $H(w,G)=\max_u h(w,u)$, and let the associated Gordon cost $L$ be defined by $L(w)=\max_u \ell (w,u)$. 

When the error vector $\xi$ is also Gaussian, we write $\xi=\sigma z$, where $z \sim \mathcal{N}(0,I_m)$. Let the `Gaussian error version' of $h,\ell$ in (\ref{def:h_l_lasso}) be defined by
\begin{align}\label{def:h_l_lasso_errg}
h^{\errg}(w,u)& = \frac{1}{m} u^\top\big(G w-\sigma z\big)-\frac{1}{2m} \pnorm{u}{}^2 + \frac{\lambda}{m} \big(\pnorm{w+\mu_0}{1}-\pnorm{\mu_0}{1}\big),\nonumber\\
{\ell}^{\errg}(w,u)& = -\frac{1}{m^{3/2}} \pnorm{u}{}g^\top w+\frac{1}{m}\pnorm{u}{}\cdot g'\sigma +\sqrt{\frac{\pnorm{w}{}^2}{m}+\sigma^2}\cdot \frac{h^\top u}{m}\nonumber\\
&\qquad\qquad -\frac{1}{2m} \pnorm{u}{}^2 + \frac{\lambda}{m}\big(\pnorm{w+\mu_0}{1}-\pnorm{\mu_0}{1}\big).
\end{align}
Here $g' \in \R$ is a standard normal independent of $g,h$ and the original Gaussian noise vector $\xi=\sigma z$. Similarly, the Lasso cost function in this Gaussian design and Gaussian error case $H^{\errg}(w,G)=\max_u h^{\errg}(w,u)$, and let the associated Gordon cost with Gaussian error $L^{\errg}$ be defined by $L^{\errg}(w)=\max_u \ell^{\errg} (w,u)$.

Finally, similar to (\ref{def:ridge_psi_pop}), let
\begin{align*}
\psi(\beta,\gamma)&\equiv \bigg(\frac{\sigma^2}{\gamma}+\gamma\bigg)\frac{\beta}{2}-\frac{\beta^2}{2}+\frac{1}{m/n}\E \min_w \bigg\{\frac{w^2}{2\gamma}-\beta Z w+\lambda\Big(\abs{w+\Pi_{\mu_0} }-\abs{\Pi_{\mu_0}}\Big)\bigg\},
\end{align*}
where $\Pi_{\mu_0}\otimes Z \equiv \big(n^{-1}\sum_{j=1}^n \delta_{\mu_{0,j}}\big)\otimes \mathcal{N}(0,1)$. $(\beta_\ast,\gamma_\ast)$ defined in (\ref{eqn:lasso_fpe}) is the unique saddle point for the max-min problem $\max_\beta \min_\gamma \psi(\beta,\gamma)$ under (R1)-(R2). 

The following proposition relates the Lasso Gordon cost functions with Gaussian and non-Gaussian errors. We will work with the probability space that Gaussian random variables $g,h,g',z$, the possibly non-Gaussian noise vector $\xi$, and the Gaussian design matrix $G$ are all independent. 

\begin{proposition}\label{prop:lasso_gordon_cost_com}
Suppose (R1)-(R2) hold, and the entries of $\xi_0$ are independent, mean $0$, variance $1$ and uniformly sub-Gaussian. There exist constants $C,K>0$ depending only on $\sigma,\lambda,\tau,M_2$ such that with probability at least $1-C n^{-100}$,
\begin{align*}
\sup_{1\leq R\leq 1/(Kr_n)}\sup_{w\in \R^n: \pnorm{w}{}/\sqrt{m}\leq R}\abs{ L^{\errg}(w) -L(w)   }\big/R^2\leq K r_n.
\end{align*}
\end{proposition}
\begin{proof}
We continue using the notation in (\ref{ineq:ridge_psi_empirical_notation}), and consider the event $E$ :
\begin{align*}
E&\equiv \big\{\abs{e_h-1}\vee \abs{e_g-1}\vee \abs{\Delta_m}\vee \abs{\sigma_m^2-\sigma^2}\leq K_0r_n\big\}\cap \big\{ \abs{g'}\leq\sqrt{\log n}  \big\}.
\end{align*}
Clearly $\Prob(E)\geq 1- Cn^{-100}$. By exactly the same calculations in (\ref{ineq:ridge_psi_empirical_pm_1})-(\ref{ineq:ridge_psi_empirical_pm_1.1}), now with $\mathsf{f}(w)=(\lambda/m)\big(\pnorm{w+\mu_0}{1}-\pnorm{\mu_0}{1}\big)$, we have
\begin{align}\label{ineq:lasso_errg_1}
&L(w) = \max_{\beta> 0} \bigg\{ \beta \cdot \bigg(\biggpnorm{\frac{ \pnorm{w}{}}{m^{1/2}}\cdot \frac{h}{\sqrt{m}}-\frac{\xi}{\sqrt{m}} }{}  - \frac{g^\top w}{m} \bigg)- \frac{\beta^2}{2} + \mathsf{f}(w)\bigg\}\nonumber\\
& = \max_\beta \min_\gamma\bigg\{ \bigg(\frac{\sigma_m^2}{\gamma}+\gamma\bigg) \frac{\beta}{2} - \frac{\beta^2}{2} +\bigg[\frac{\beta}{2\gamma}\frac{\pnorm{w}{}^2}{m}e_h^2 - \beta\cdot \frac{g^\top w}{m}+\mathsf{f}(w)\bigg] - \frac{\beta}{\gamma}\frac{\pnorm{w}{}}{\sqrt{m}} \Delta_m\bigg\} \nonumber\\
&\equiv \max_{\beta}\min_{\gamma}\psi_n(\beta,\gamma,w),
\end{align}
and
\begin{align}\label{ineq:lasso_errg_2}
&L^{\errg}(w)  = \max_{\beta}\bigg\{\beta\cdot \bigg(\sqrt{ \frac{\pnorm{w}{}^2}{m}+\sigma^2}\cdot e_h-\frac{g^\top w}{m}+\frac{g' \sigma}{\sqrt{m}}\bigg)-\frac{\beta^2}{2}+ \mathsf{f}(w)\bigg\}\nonumber\\
&= \max_\beta\min_\gamma \bigg\{ \bigg(\frac{\sigma^2 e_h^2}{\gamma}+\gamma\bigg) \frac{\beta}{2}- \frac{\beta^2}{2} +\bigg[\frac{\beta}{2\gamma}\frac{\pnorm{w}{}^2}{m}e_h^2 - \beta\cdot \frac{g^\top w}{m}+\mathsf{f}(w)\bigg]+ \beta \cdot \frac{g'\sigma}{\sqrt{m}}\bigg\}\nonumber\\
&\equiv \max_\beta \min_\gamma \psi_n^{\errg}(\beta,\gamma,w).
\end{align}
The inner minimum with respect to $\gamma$ in (\ref{ineq:lasso_errg_1})-(\ref{ineq:lasso_errg_2}), denoted $\gamma_{\ast,n}(w),\gamma_{\ast,n}^{\errg}(w)$, can be computed exactly: $\gamma_{\ast,n}(w)=\bigpnorm{\frac{ \pnorm{w}{}}{m^{1/2}}\cdot \frac{h}{\sqrt{m}}-\frac{\xi}{\sqrt{m}} }{}$, $\gamma_{\ast,n}^{\errg}(w)=e_h\big(\sigma^2+\pnorm{w}{}^2/m\big)^{1/2}$. So on the event $E$, for any $w\in \R^n$ such that $\pnorm{w}{}/\sqrt{m}\leq R$ with $1\leq R\leq 1/(Kr_n)$,
\begin{align}\label{ineq:lasso_errg_3}
\frac{\sigma}{2}\leq \gamma_{\ast,n}^{\errg}(w)\wedge \gamma_{\ast,n}(w)\leq \gamma_{\ast,n}^{\errg}(w)\vee \gamma_{\ast,n}(w)\leq 2(\sigma+R).
\end{align}
On the other hand, the maximizers with respect to $\beta$ in (\ref{ineq:lasso_errg_1})-(\ref{ineq:lasso_errg_2}), denoted $\beta_{\ast,n}^{\errg}(w),\beta_{\ast,n}(w)$, are solutions to a quadratic form. It then easily follows that on the event $E$, for any $w\in \R^n$ such that $\pnorm{w}{}/\sqrt{m}\leq R$, 
\begin{align}\label{ineq:lasso_errg_4}
\beta_{\ast,n}^{\errg}(w)\vee \beta_{\ast,n}(w)\leq 2R+K_\beta
\end{align}
holds for some $K_\beta=K_\beta(\sigma)>0$. Now using the final lines in (\ref{ineq:lasso_errg_1})-(\ref{ineq:lasso_errg_2}) with the range estimates in (\ref{ineq:lasso_errg_3})-(\ref{ineq:lasso_errg_4}), on the event $E$, for any $w\in \R^n$ such that $\pnorm{w}{}/\sqrt{m}\leq R$ with $1\leq R\leq 1/(Kr_n)$, 
\begin{align*}
&\bigabs{ L^{\errg}(w) -L(w)   }=\bigabs{\psi_n^{\errg}(\beta_{\ast,n}^{\errg},\gamma_{\ast,n}^{\errg},w)-\psi_n(\beta_{\ast,n},\gamma_{\ast,n},w)}\\
&\leq \max_{0\leq \beta \leq 2R+K_\beta}\max_{\sigma/2\leq \gamma\leq 2(\sigma+R)}\bigabs{\psi_n^{\errg}(\beta,\gamma,w)-\psi_n(\beta,\gamma,w)}\\
&\leq \max_{0\leq \beta \leq 2R+K_\beta}\max_{\sigma/2\leq \gamma\leq 2(\sigma+R)}\bigg\{\bigg(\bigabs{\sigma_m^2-\sigma^2e_h^2}+2R\abs{\Delta_m}\bigg)\frac{\beta}{2\gamma}+\beta\cdot \frac{g'\sigma}{\sqrt{m}} \bigg\}\\
&\leq K_1\cdot R^2 r_n
\end{align*}
for some $K_1=K(\sigma,\tau)>0$. The proof is thus complete.
\end{proof}

\subsection{Proof of Theorem \ref{thm:lasso_dist}, distribution of $\hat{w}^{\lasso}_A$}

\begin{proposition}\label{prop:lasso_est_gap}
Suppose (R1)-(R2) hold, and the entries of $\xi_0$ are independent, mean $0$, variance $1$ and uniformly sub-Gaussian. Let $\mathsf{g}:\R^n\to \R$ be $1$-Lipschitz, and $1\leq R\leq n^{0.49}$ be a real number. Then there exist constants $C,K>0$ depending only on $\sigma,\lambda,\tau,M_2$ such that for all $\epsilon \geq K \cdot R^2 r_n$, with probability at least $1-C n^{-100}$, 
\begin{align*}
\min_{w \in D_\epsilon(\mathsf{g})\cap B_n(\sqrt{m}R)} H(w,G)\geq \psi(\beta_\ast,\gamma_\ast)+ K^{-1}\epsilon.
\end{align*} 
Here $
D_\epsilon(\mathsf{g})\equiv \big\{ w \in \R^n: \abs{ \mathsf{g}\big(w/\sqrt{n}\big) - \E \mathsf{g}\big(w_\ast/\sqrt{n}\big)}\geq \epsilon^{1/2} \big\}$.
\end{proposition}

\begin{proof}
By Gaussian concentration, with probability at least $1-C n^{-100}$, $
\bigabs{\mathsf{g}(w_\ast/\sqrt{n})-\E \mathsf{g}(w_\ast/\sqrt{n})}\leq Kr_n$. This means on an event $E_1$ with probability at least $1-C n^{-100}$,
\begin{align*}
\epsilon^{1/2}&\leq \bigabs{ \mathsf{g}(w/\sqrt{n}) - \E \mathsf{g}(w_\ast/\sqrt{n})}\\
&\leq \bigabs{ \mathsf{g}(w/\sqrt{n}) - \mathsf{g}(w_\ast/\sqrt{n})}+ Kr_n\leq n^{-1/2}\pnorm{w-w_\ast}{}+Kr_n
\end{align*}
holds uniformly in $w \in D_\epsilon(\mathsf{g})$. Consequently, for $\epsilon\geq K r_n^2$, on the event $E_1$,
\begin{align}\label{ineq:lasso_w_gap_1}
n^{-1}\pnorm{w-w_{\ast}}{}^2\geq \epsilon/2,\quad \forall w \in D_\epsilon(\mathsf{g}).
\end{align}
By \cite[Theorem B.1, Corollary B.1]{miolane2021distribution}, the following holds on an event $E_2$ with probability at least $1-(C/\epsilon) e^{-n\epsilon^2/C}$: For any $w \in \R^n$ such that $n^{-1}\pnorm{w-w_\ast}{}^2> \epsilon$, $L^{\errg}(w)\geq \psi(\beta_\ast,\lambda_\ast)+\epsilon/K$. Consequently, for any $\epsilon\geq K r_n^2$, on the event $E_1\cap E_2$, $L^{\errg}(w)\geq \psi(\beta_\ast,\lambda_\ast)+\epsilon/K$ holds uniformly in $w \in D_\epsilon(\mathsf{g})$. By Proposition \ref{prop:lasso_gordon_cost_com}, $L^{\errg}(w)$ may be replaced by $L(w)$ uniformly in $w \in \R^n$ with $\pnorm{w}{}/\sqrt{m}\leq R$ with the additional constraint that $\epsilon\geq K (r_n^2+R^2 r_n)$. The claim now follows by an application of an obviously modified version of the CGMT as in \cite[Corollary 5.1]{miolane2021distribution} that holds for non-Gaussian errors.
\end{proof}

\begin{proof}[Proof of Theorem \ref{thm:lasso_dist}: distribution of $\hat{w}^{\lasso}_A$]

Let $R_n\equiv K_1\sqrt{\log n}$ for a large enough constant $K_1>0$. By Proposition \ref{prop:lasso_risk}-(2), we have $\Prob(\hat{w}_A \notin B_n(\sqrt{m}R_n))\leq C n^{-100}$. 	
	
First, by \cite[Corollary B.1]{miolane2021distribution}, we have
\begin{align*}
\Prob\bigg(\min_w L^{\errg}(w)\geq \psi(\beta_\ast,\gamma_\ast) +K r_n \bigg)\leq C n^{-100}. 
\end{align*}
We will replace $\min_w L^{\errg}(w)$ by $\min_w H(w,G)$ as follows: (1) By \cite[Theorem B.1]{miolane2021distribution} and subsequent remarks, it can be replaced by $\min_{w \in B_n (\sqrt{m}R_n)} L^{\errg}(w)$; (2) using Proposition \ref{prop:lasso_gordon_cost_com}, it can be further replaced by $\min_{w \in B_n(\sqrt{m}R_n)} L(w)$; (3) using CGMT, it can be then replaced by $\min_{w \in B_n(\sqrt{m}R_n)} H(w,G)$; (4) using the choice of $R_n$, it can be finally replaced by $\min_w H(w,G)$. In summary, 
\begin{align*}
\Prob\bigg(\min_{w} H(w,G)\geq \psi(\beta_\ast,\gamma_\ast) +K r_n \bigg)\leq C n^{-100}. 
\end{align*}
On the other hand, by Proposition \ref{prop:lasso_est_gap}, for any $\epsilon \geq K\cdot R_n^2 r_n$, 
\begin{align*}
\Prob\bigg(\min_{w \in D_\epsilon(\mathsf{g})\cap B_n(\sqrt{m}R_n) } H(w,G)\leq  \psi(\beta_\ast,\gamma_\ast)+ K^{-1}\epsilon\bigg)\leq C n^{-100}. 
\end{align*}
Now let $\mathcal{S}_n\equiv D_\epsilon(\mathsf{g})\cap B_n(\sqrt{m}R_n)$, $z\equiv \psi(\beta_\ast,\gamma_\ast)$, $\rho_0\equiv \epsilon/K_1$ for some large $K_1>0$, we may apply Theorem \ref{thm:lasso_universality_generic} to conclude that  for $\epsilon\geq K\cdot R_n^2 r_n$, 
\begin{align*}
\Prob\big(\hat{w}_A \in D_\epsilon(\mathsf{g})\cap B_n(\sqrt{m}R_n)\big)\leq C (1\vee \epsilon^{-3}) s_n. 
\end{align*}
Combined with the above display,  
\begin{align*}
\Prob\big(\hat{w}_A \in D_\epsilon(\mathsf{g})\big)\leq C (1\vee \epsilon^{-3}) s_n
\end{align*}
holds for $\epsilon\geq K r_n\log n$. The restriction on $\epsilon$ can be dropped for free as otherwise the bound becomes trivial. 
\end{proof}

\subsection{Proof of Theorem \ref{thm:lasso_dist}, distribution of $\hat{r}^{\lasso}_A$}
For a general design matrix $A$, let
\begin{align*}
h(w,u;A)& \equiv \frac{1}{m} u^\top A w -\frac{1}{m}u^\top \xi-\frac{1}{2m} \pnorm{u}{}^2 + \frac{\lambda}{m} \big(\pnorm{w+\mu_0}{1}-\pnorm{w}{1}\big).
\end{align*}
It is easy to see that
\begin{align*}
\hat{u}\equiv \hat{u}_A\equiv \argmax_{u \in \R^m} \min_{w \in \R^n} h(w,u;A) = A \hat{w}-\xi.
\end{align*}
Similar to (\ref{def:u_ridge}), we define its Lasso analogue
\begin{align}\label{def:u_lasso}
u_\ast\equiv \frac{\beta_\ast}{\gamma_\ast} \bigg(\sqrt{ (\gamma_\ast)^2-\sigma^2}\cdot h-\xi\bigg),
\end{align}
as the `population version' of $\hat{u}$ in the Gordon problem.

\begin{proof}[Proof of Theorem \ref{thm:lasso_dist}: distribution of $\hat{r}^{\lasso}_A$]
The proof follows a similar strategy to that of the second part of Theorem \ref{thm:ridge_dist}, so we only outline key steps. 
For any $\epsilon>0$, and any $\mathsf{h}:\R^m \to \R$, let
\begin{align*}
D_\epsilon\equiv D_\epsilon(\mathsf{h})&\equiv \bigg\{ u \in \R^m : \bigabs{ \mathsf{h}(u/\sqrt{n}) - \E_h \mathsf{h}(u_\ast/\sqrt{n})}\geq \epsilon^{1/2} \bigg\}.
\end{align*}
Consider the event $E_1$ on which $\pnorm{\hat{u}}{\infty}\leq K\sqrt{\log n}\equiv L_n$, and $E_{2}$ on which $\pnorm{\hat{w}}{\infty}\leq K\sqrt{\log n}$ for some large enough $K>0$. By Proposition \ref{prop:lasso_risk}-(2)(3), $\Prob(E_1)\wedge \Prob(E_{2})\geq 1-Cn^{-100}$. Let 
\begin{align}\label{ineq:lasso_u_dist_1}
\rho_0\equiv \epsilon/K,\quad z_0\equiv \psi(\beta_\ast,\gamma_\ast)-3\rho_0
\end{align}
for some large enough $K>0$. Then arguing as in the proof of the second part of Theorem \ref{thm:ridge_dist},
\begin{align}\label{ineq:lasso_u_dist_2}
\Prob\big(\hat{u} \in D_\epsilon\big)&\leq \Prob\bigg(\max_{u \in [-L_n,L_n]^m}\min_{w \in \R^n} h(w,u;A)< z_0+6\rho_0\bigg)\nonumber\\
&\qquad + \Prob\bigg(\max_{u \in [-L_n,L_n]^m \cap D_\epsilon}\min_{w \in \R^n} h(w,u;A)> z_0+3\rho_0\bigg)+\Prob(E_1^c)\nonumber\\
&\equiv \mathfrak{p}_1+\mathfrak{p}_2+\Prob(E_1^c).
\end{align}
For $\mathfrak{p}_1$, arguing as in the proof of the second part of Theorem \ref{thm:ridge_dist} that replaces $A$ by its Gaussian counterpart $G$, 
\begin{align*}
\mathfrak{p}_1
&\leq \Prob\bigg(\min_{w \in \R^n} H(w;G)< z_0+2\rho_0\bigg)+ C (1\vee \rho_0^{-3})s_n.
\end{align*}
By (a non-Gaussian error modified version of) the CGMT as in \cite[Corollary 5.1]{miolane2021distribution}, we have
\begin{align*}
\mathfrak{p}_1
&\leq 2\Prob\bigg(\min_{w \in \R^n} L(w)< z_0+2\rho_0\bigg)+ C (1\vee \rho_0^{-3})s_n.
\end{align*}
With $R_n\equiv K'\sqrt{\log n}$ for some large enough $K'>0$, using Proposition \ref{prop:lasso_gordon_cost_com} we may replace the above Gordon cost $L$ with $L^{\errg}$: for $\rho_0 \geq K r_n$,
\begin{align*}
\mathfrak{p}_1
&\leq 2\Prob\bigg(\min_{w \in B_n(\sqrt{m}R_n)} L^{\errg}(w)< z_0+\rho_0\bigg)+ C(1\vee \rho_0^{-3})s_n.
\end{align*}
In view of the definition of $z_0$ in (\ref{ineq:lasso_u_dist_1}), by \cite[Corollary B.1]{miolane2021distribution} that applies to $L^{\errg}$, we have for $\rho_0 \geq K r_n$, 
\begin{align}\label{ineq:lasso_u_dist_3}
\mathfrak{p}_1\leq C\rho_0^{-1} e^{-n\rho_0^2/C}+ C \cdot (1\vee \rho_0^{-3})s_n\leq C (1\vee \rho_0^{-3}) s_n.
\end{align}
Next we handle $\mathfrak{p}_2$. Arguing as in the proof of the second part of Theorem \ref{thm:ridge_dist},
\begin{align}\label{ineq:lasso_u_dist_4}
\mathfrak{p}_2
&\leq 2\Prob\bigg(\max_{u \in  D_\epsilon} {\ell}(w_\ast,u)> \psi(\beta_\ast,\gamma_\ast)-2\rho_0\bigg) + C (1\vee \rho_0^{-3})s_n.
\end{align}
The first term above has been almost exactly studied in \cite[Lemma D.1]{miolane2021distribution}, which holds as long as the components of the error vector is uniformly sub-Gaussian. In particular, the function $u\mapsto {\ell}(w_\ast,u)$ is $1/m$-strongly concave with a unique minimizer $u_{\ast,n}$ with probability at least $1- C e^{-n/C}$, and for any $\delta>0$ with probability at least $1- C e^{-n\delta^2/C}$, we have (i) $\abs{\max_u {\ell}(w_\ast,u)- \psi(\beta_\ast,\gamma_\ast)}\leq \delta$, and (ii) $m^{-1}\pnorm{u_{\ast,n}-u_\ast}{}^2\leq \delta$, where recall $u_\ast$ is defined in (\ref{def:u_lasso}). Arguing as in the proof of the second part of Theorem \ref{thm:ridge_dist}, i.e., using Gaussian concentration conditionally on $\xi$, for any $\epsilon\geq K r_n^2$, with unconditional probability $1-Cn^{-100}$,
\begin{align*}
m^{-1}\pnorm{u-u_{\ast,n}}{}^2\geq \epsilon/4,\quad \forall u \in D_\epsilon. 
\end{align*}
Using the high probability strong concavity of ${\ell}(w_\ast,\cdot)$ and (i)-(ii), we now see that for all $\epsilon\geq Kr_n$, with probability at least $1- C n^{-100}$, 
\begin{align*}
\max_{u \in D_\epsilon} {\ell}(w_\ast,u)\leq \max_{u} {\ell}(w_\ast,u) - 2\epsilon/K \leq \psi(\beta_\ast,\gamma_\ast)-\epsilon/K.
\end{align*}
Combined with (\ref{ineq:lasso_u_dist_4}) and the definition of $\rho_0$ in (\ref{ineq:lasso_u_dist_1}), for all $\epsilon\geq Kr_n$,
\begin{align}\label{ineq:lasso_u_dist_5}
\mathfrak{p}_2\leq C (1\vee \epsilon^{-3})s_n.
\end{align}
Now combining (\ref{ineq:lasso_u_dist_2}), (\ref{ineq:lasso_u_dist_3}) and (\ref{ineq:lasso_u_dist_5}), we find that for all $\epsilon\geq Kr_n$,
\begin{align*}
\Prob\big(\hat{u} \in D_\epsilon\big)& \leq C (1\vee \epsilon^{-3})  s_n. 
\end{align*}
The proof is complete.
\end{proof}

\subsection{Proof of Theorem \ref{thm:lasso_dist}, distribution of $\hat{v}^{\lasso}_A$}
For a general design matrix $A$, let
\begin{align*}
	k(v,w;A) &\equiv \frac{1}{2m}\pnorm{Aw - \xi}{}^2 + \frac{\lambda}{m}v^\top(w + \mu_0) - \frac{\lambda}{m}\pnorm{\mu_0}{1},\\
	k_0(u,v,w;A) &\equiv \frac{1}{m}u^\top(Aw - \xi) - \frac{1}{2m}\pnorm{u}{}^2 + \frac{\lambda}{m}v^\top(w + \mu_0) - \frac{\lambda}{m}\pnorm{\mu_0}{1}.
\end{align*}
Recall $v_\ast$ defined in (\ref{def:v_lasso}), and let $\bar{s}_n\equiv s_n \log n$.

\begin{proof}[Proof of Theorem \ref{thm:lasso_dist}: distribution of $\hat{v}^{\lasso}_A$]
	For any $\epsilon > 0$ and $\mathsf{g}: \R^n \rightarrow \R$, let
	\begin{align*}
		D_\epsilon \equiv D_\epsilon(\mathsf{g}) \equiv \Big\{v\in\R^n: \bigabs{\mathsf{g}(v/\sqrt{n}) - \E \mathsf{g}(v_\ast/\sqrt{n})}\geq \epsilon^{1/2}\Big\}.
	\end{align*}
	First note by the proof \cite[Lemma E.1]{miolane2021distribution}, and the $\ell_\infty$ bounds proved in Proposition \ref{prop:lasso_risk}, with $L_n\equiv K\sqrt{\log n}$, on an event $E_1$ with $\Prob(E_1)\geq 1 - Cn^{-100}$, we have $
	\hat{v} \equiv \hat{v}_A \in \argmax_{v\in\R^n:\pnorm{v}{\infty}\leq 1} \min_{w\in [-L_n,L_n]^n} k(v,w;A)$ and
	\begin{align}\label{eq:cost_id_subg}
	\max_{v\in[-1,1]^n}\min_{w\in [-L_n,L_n]^n} k(v,w;A) &= \max_{v\in[-1,1]^n}\min_{w\in B_{n;1}(Kn)} k(v,w;A) \nonumber\\
	& =\min_{w\in [-L_n,L_n]^n} H(w;A) = \min_{w\in\R^n} H(w;A).
	\end{align}
	Without loss of generality, we also assume that these properties hold on $E_1$ for the standard Gaussian design $G$. With $z_0,\rho_0$ as in (\ref{ineq:lasso_u_dist_1}), it is easy to verify
	\begin{align}\label{ineq:lasso_subg_1}
		\Prob(\hat{v}\in D_\epsilon) &\leq \Prob\Big(\max_{v\in[-1,1]^n}\min_{w\in [-L_n,L_n]^n} k(v,w;A) < z_0 + 6\rho_0\Big) \nonumber \\
		&\qquad +\Prob\Big(\max_{v\in D_\epsilon \cap [-1,1]^n}\min_{w\in [-L_n,L_n]^n} k(v,w;A) > z_0 + 3\rho_0\Big)+ \Prob(E_1^c) \nonumber \\
		& \equiv \mathfrak{p}_1 + \mathfrak{p}_2+ \Prob(E_1^c).
	\end{align}
	For $\mathfrak{p}_1$, we apply (\ref{eq:cost_id_subg}) and then a similar argument as in the previous proof of the Lasso residual $\hat{r}$ (up to (\ref{ineq:lasso_u_dist_3})) to obtain that
	\begin{align}\label{ineq:lasso_subg_2}
		\mathfrak{p}_1 \leq \Prob\Big(\min_{w\in\R^n} H(w;A) < z_0 + 6\rho_0\Big)+\Prob(E_1^c) \leq  C(1\vee \rho_0^{-3})s_n.
	\end{align}
	For $\mathfrak{p}_2$, using the fact that $\max_{u\in\R^m} k_0(u,v,w;A) = k(v,w;A)$, we have
	\begin{align*}
		\mathfrak{p}_2 &= \Prob\Big(\max_{v\in D_\epsilon \cap [-1,1]^n}\min_{w\in [-L_n,L_n]^n}\max_{u\in \R^m}k_0(u,v,w;A) > z_0 + 3\rho_0\Big).
	\end{align*}
	Now note that for any fixed $v\in D_\epsilon$, the inner most optimizer is $\hat{u}(v) = A\hat{w}(v) - \xi$, where $\hat{w}(v) \in \argmax_{w\in\R^n} k(v,w;A)$. By Lemma \ref{lem:l_infty_u_subgradient} below, on an event $E_2$ with $\Prob(E_2)\geq 1-Cn^{-100}$, we have $\sup_{v\in[-1,1]^n}\pnorm{\hat{u}(v)}{\infty} \leq L_n^2$. Consequently,
	\begin{align*}
		\mathfrak{p}_2 &\leq \Prob\Big(\max_{v\in D_\epsilon \cap [-1,1]^n}\min_{w\in [-L_n,L_n]^n}\max_{u\in [-L_n^2, L_n^2]^m}k_0(u,v,w;A) > z_0 + 3\rho_0\Big)+\Prob(E_2^c)\\
		&= \Prob\Big(\max_{v\in D_\epsilon \cap [-1,1]^n}\max_{u\in [-L_n^2, L_n^2]^m}\min_{w\in [-L_n,L_n]^n}k_0(u,v,w;A) > z_0 + 3\rho_0\Big)+\Prob(E_2^c),
	\end{align*}
	where the last identity follows from Sion's min-max theorem (cf. Theorem \ref{thm:sion_minmax}). Now we may apply Corollary \ref{cor:min_max_universality} by viewing $(u,v)$ as a single variable, and enlarging $A$ with $n$ rows of $0$'s from the bottom. In particular, 
	\begin{align*}
	\mathfrak{p}_2 &\leq \Prob\Big(\max_{v\in D_\epsilon \cap [-1,1]^n}\max_{u\in [-L_n^2, L_n^2]^m}\min_{w\in [-L_n,L_n]^n}k_0(u,v,w;G) > z_0 + \rho_0\Big) + C(1\vee \rho_0^{-3}) \bar{s}_n. 
	\end{align*}
	Now by reversing the steps implemented for $A$, 
	\begin{align*}
	\mathfrak{p}_2 &\leq \Prob\Big(\max_{v\in D_\epsilon \cap [-1,1]^n}\min_{w\in B_{n;1}(Kn)} k(v,w;G) > z_0 + \rho_0\Big)+ C(1\vee \rho_0^{-3}) \bar{s}_n. 
	\end{align*}
	Let
	\begin{align*}
	V^{\errg}(v)&\equiv \min_{w \in B_{n;1}(Kn)}\bigg\{\frac{1}{2}\bigg(\sqrt{ \frac{\pnorm{w}{}^2}{m}+\sigma^2 }\cdot\frac{\pnorm{h}{}}{\sqrt{m}}-\frac{g^\top w}{m}+\frac{g'\sigma}{\sqrt{m}}\bigg)_+^2 +\frac{\lambda}{m}v^\top(w+\mu_0)-\frac{\lambda}{m}\pnorm{\mu_0}{1}\bigg\}
	\end{align*}
	be the Gordon cost associated for the subgradient (when the error $\xi=\sigma z$ is Gaussian), as defined in the beginning of \cite[Appendix E.2.1]{miolane2021distribution}. Then by an application of CGMT in the form given by \cite[Proposition E.1]{miolane2021distribution}, we obtain
	\begin{align}\label{ineq:lasso_subg_3}
	\mathfrak{p}_2 &\leq 2\Prob\Big(\max_{v\in D_\epsilon \cap [-1,1]^n}V^{\errg}(v) > z_0 + \rho_0\Big)+ C(1\vee \rho_0^{-3}) \bar{s}_n. 
	\end{align}
	On the other hand, using Gaussian concentration arguments (cf. (\ref{ineq:lasso_w_gap_1})), for $\epsilon\geq K r_n^2$, on an event $E_3$ with $\Prob(E_3)\geq 1- C n^{-100}$,
	\begin{align*}
	n^{-1}\pnorm{v-v_\ast}{}^2\geq \epsilon/2,\quad \forall v \in D_\epsilon(\mathsf{g}).
	\end{align*}
	By \cite[Theorem E.7, Lemma E.2]{miolane2021distribution}, the above display implies for $\epsilon\geq K r_n$, with probability at least $1-Cn^{-100}$, 
	\begin{align*}
	\max_{v\in D_\epsilon \cap [-1,1]^n}V^{\errg}(v) \leq \max_{v\in [-1,1]^n}V^{\errg}(v) -\epsilon/K=\min_{w \in \R^n} L^{\errg}(w)-\epsilon/K.
	\end{align*} 
	Combined with (\ref{ineq:lasso_subg_3}), we have for all $\epsilon\geq Kr_n$, 
	\begin{align*}
	\mathfrak{p}_2 &\leq 2\Prob\Big(\min_{w \in \R^n} L^{\errg}(w) > z_0 + \rho_0+\epsilon/K\Big)+ C(1\vee \rho_0^{-3}) \bar{s}_n. 
	\end{align*}
	Now repeating the arguments around (\ref{ineq:lasso_u_dist_3}) yields that for $\rho_0\geq K r_n$,
	\begin{align}\label{ineq:lasso_subg_4}
	\mathfrak{p}_2 &\leq  C(1\vee \rho_0^{-3}) \bar{s}_n. 
	\end{align}
	Strictly speaking the above display is derived under the Gaussian error assumption. For general error distributions we may proceed as in Proposition \ref{prop:lasso_gordon_cost_com} that quantifies the difference between the Gordon costs for Gaussian and non-Gaussian errors; details are omitted. Now combining (\ref{ineq:lasso_subg_1}), (\ref{ineq:lasso_subg_2}) and (\ref{ineq:lasso_subg_4}) to conclude. 
\end{proof}

\begin{lemma}\label{lem:l_infty_u_subgradient}
	Suppose (R1)-(R3) hold. Fix $L_n\geq 1$. For any $v\in\R^n$ such that $\pnorm{v}{\infty} \leq 1$, let $ \hat{u}(v)$ be the maximizer of the optimization problem $\min_{w\in [-L_n,L_n]^n}\max_{u\in \R^m} k_0(u,v,w;A)$. Then there exist constants $C,K>0$ depending only on $\sigma,\lambda,\tau,M_2$ such that with probability at least $1-C n^{-100}$, 
	\begin{align*}
		\sup_{v\in[-1,1]^n}\pnorm{\hat{u}(v)}{\infty} \leq K L_n\sqrt{\log n}. 
	\end{align*}
\end{lemma}
\begin{proof}
	We write $B_n\equiv [-L_n,L_n]^n$ in the proof. For any fixed $v$, the inner optimizer is $\hat{u}(v) = A\hat{w}(v) - \xi$, where 
	\begin{align*}
	\hat{w}(v) \in \argmin_{w\in B_n} k(v,w;A) = \argmin_{w\in B_n} \bar{H}_v(w;A).
	\end{align*}
	Here $\bar{H}_v(w;A)\equiv \pnorm{Aw-\xi}{}^2/2+\mathsf{f}_v(w)\equiv \pnorm{Aw-\xi}{}^2/2+\lambda v^\top (w+\mu_0)$. We claim that for any $w \in B_n$, 
	\begin{align}\label{ineq:l_infty_u_subgradient_1}
	\frac{1}{2}\pnorm{A(w-\hat{w}(v))}{}^2\leq \bar{H}_v(w;A)-\bar{H}_v(\hat{w}(v);A). 
	\end{align}
	To prove (\ref{ineq:l_infty_u_subgradient_1}), first note that by expanding the cost $\bar{H}_v$ at $\hat{w}(v)$, we have
	\begin{align*}
	\bar{H}_v(w;A)& = \bar{H}_v(\hat{w}(v);A) + \frac{1}{2}\pnorm{A(w - \hat{w}(v))}{}^2\\
	&\qquad  +\iprod{A(w-\hat{w}(v))}{A\hat{w}(v)-\xi}+\mathsf{f}_v(w)-\mathsf{f}_v(\hat{w}(v)).
	\end{align*}
	By convexity of $B_n$, $\hat{w}(v)+\epsilon(w-\hat{w}(v)) \in B_n$ for $\epsilon \in [0,1]$, so the cost optimality of $\bar{H}_v$ at $\hat{w}(v)$ entails
	\begin{align*}
	0&\leq \frac{\d}{\d \epsilon}\bar{H}_v(\hat{w}(v)+\epsilon(w-\hat{w}(v));A)\bigg|_{\epsilon=0}\\
	&= \iprod{w-\hat{w}(v)}{A^\top(A\hat{w}(v)-\xi)}+ \iprod{w-\hat{w}(v)}{\nabla \mathsf{f}_v(\hat{w}(v))}.
	\end{align*}
	Combining the above two displays and using the convexity (actually, linearity) of $\mathsf{f}_v$ yield the claim (\ref{ineq:l_infty_u_subgradient_1}). Now as in the proof of Lemma \ref{lem:ridge_loo}-(2), we let $w$ in (\ref{ineq:l_infty_u_subgradient_1}) be the row leave-one-out version of $\hat{w}(v)$, defined for each $t \in [m]$ by 
	\begin{align*}
		\hat{w}^{[t]}(v) \equiv \argmin_{w\in B_n} \bigg\{\frac{1}{2}\bigpnorm{A_{[-t]}w - \xi_{-t}}{}^2 + \mathsf{f}_v(w) \bigg\} = \argmin_{w\in B_n}\bigg\{\bar{H}_v(w;A)-\frac{\abs{a_t^\top w-\xi_t}^2}{2}\bigg\}.
	\end{align*}
	Repeating the proof of Lemma \ref{lem:ridge_loo}-(2) with (\ref{ineq:l_infty_u_subgradient_1}), we have
	\begin{align}\label{ineq:l_infty_u_subgradient_2}
	\pnorm{ A(\hat{w}^{[t]}(v)-\hat{w}(v)) }{}\leq 2 \big(\abs{\xi_t}+\abs{a_t^\top \hat{w}^{[t]}(v) }\big).
	\end{align}
	Following verbatim the proof of Proposition \ref{prop:ridge_risk}-(3) and using (\ref{ineq:l_infty_u_subgradient_2}) above, 
	\begin{align*}
	\abs{(A\hat{w}(v))_t}\leq \abs{\xi_t}+\abs{a_t^\top \hat{w}^{[t]}(v) }.
	\end{align*}
	The claim follows by recalling $\hat{u}(v) = A\hat{w}(v) - \xi$ and using a union bound.
\end{proof}

\subsection{Proof of Theorem \ref{thm:lasso_dist}, distribution of $\hat{s}^{\lasso}_A$} 

We shall divide the proof into upper and lower bounds separately. Recall $\bar{s}_n=s_n\log n$ and $s_\ast$ defined in (\ref{def:s_lasso}).

\begin{proof}[Proof of the upper bound for $\hat{s}^{\lasso}_A$]
As the subgradient $\abs{\hat{v}_j}=1$ if $\hat{\mu}_j\neq 0$, we have 
\begin{align*}
\hat{s}=\frac{\pnorm{\hat{\mu}}{0}}{n}\leq \mathsf{g}^+(\hat{v}),\quad \mathsf{g}^+(v)\equiv \frac{1}{n}\sum_{j=1}^n \mathsf{g}^+_0(v_j)\equiv\frac{1}{n}\sum_{j=1}^n \bm{1}(\abs{v_j}\geq 1).
\end{align*}
For any $\delta \in (0,1)$, let $I_\delta^+\equiv (-1,-1+\delta)\cup(1-\delta,1)$ and let $\mathsf{g}^+_{0,\delta}: \R \to [0,1]$ be defined by $\mathsf{g}^+_{0,\delta}(x)\equiv \mathsf{g}^+_{0}(x)$ for $x \in \R\setminus I_\delta^+$ and linearly interpolated otherwise. Clearly $\mathsf{g}^+_{0}\leq \mathsf{g}^+_{0,\delta}$. As $\mathsf{g}^+_{0,\delta}$ is $1/\delta$-Lipschitz, $\mathsf{g}^+_{\delta}:\R^n \to [0,1]$ defined by $\mathsf{g}^+_{\delta}(v)\equiv n^{-1}\sum_{j=1}^n \mathsf{g}^+_{0,\delta}(v_j)$ is $1/\delta$-Lipschitz as well. By the proven distributional result on $\hat{v}$, on an event $E_+(\epsilon,\delta)$ with probability at least $1- C(1\vee \epsilon^{-6}) \bar{s}_n$, we have
\begin{align*}
\delta\mathsf{g}^+_{\delta}(\hat{v})&\leq \delta\E \mathsf{g}^+_{\delta}(v_\ast)+\epsilon\leq \delta\E \mathsf{g}^+(v_\ast)+ \delta \Prob\big(v_\ast \in I_\delta^+\big)+\epsilon\leq  \delta\E \mathsf{g}^+(v_\ast)+ K\delta^2 +\epsilon.
\end{align*}
Here the last inequality follows by the anti-concentration of $v_\ast$ over $I_\delta^+$ (not including $\pm1$). So by choosing $\delta\equiv \epsilon^{1/2}$, with probability at least $1- C(1\vee \epsilon^{-6}) \bar{s}_n$,
\begin{align*}
\hat{s} \leq \mathsf{g}^+(\hat{v})\leq \mathsf{g}^+_{\epsilon^{1/2}}(\hat{v})\leq \E \mathsf{g}^+(v_\ast)+ K\epsilon^{1/2} = s_\ast+ K\epsilon^{1/2},
\end{align*}
proving the upper bound. 
\end{proof}
\begin{proof}[Proof of the lower bound for $\hat{s}^{\lasso}_A$]
The proof is similar to that of the upper bound, but now we work with $\hat{\mu}$ instead of the subgradient $\hat{v}$. We provide some details below. First note that 
\begin{align*}
\hat{s}=\frac{\pnorm{\hat{\mu}}{0}}{n}= \mathsf{g}^-(\hat{\mu}),\quad \mathsf{g}^-(\mu)\equiv \frac{1}{n}\sum_{j=1}^n \mathsf{g}^-_0(\mu_j)\equiv\frac{1}{n}\sum_{j=1}^n \bm{1}(\abs{\mu_j}\neq 0).
\end{align*}
For any $\delta \in (0,1)$, let $I_\delta^-\equiv (-\delta,0)\cup (0,\delta)$ and let $\mathsf{g}^-_{0,\delta}: \R \to [0,1]$ be defined by $\mathsf{g}^-_{0,\delta}(x)\equiv \mathsf{g}^-_{0}(x)$ for $x \in \R\setminus I_\delta^{-}$ and linearly interpolated otherwise. Clearly $\mathsf{g}^-_{0}\geq \mathsf{g}^-_{0,\delta}$. As $\mathsf{g}^-_{0,\delta}$ is $1/\delta$-Lipschitz, $\mathsf{g}^-_{\delta}:\R^n \to [0,1]$ defined by $\mathsf{g}^-_{\delta}(\mu)\equiv n^{-1}\sum_{j=1}^n \mathsf{g}^-_{0,\delta}(\mu_j)$ is $1/\delta$-Lipschitz as well. 	By the proven distributional result on $\hat{\mu}$, on an event $E_-(\epsilon,\delta)$ with probability at least $1- C(1\vee \epsilon^{-6}) \bar{s}_n$, we have
\begin{align*}
\delta\mathsf{g}^-_{\delta}(\hat{\mu})&\geq \delta\E \mathsf{g}^-_{\delta}(w_\ast+\mu_0)-\epsilon\geq \delta\E \mathsf{g}^-(w_\ast+\mu_0)- \delta \Prob\big(w_\ast+\mu_0 \in I_\delta^-\big)-\epsilon\\
&\geq  \delta\E \mathsf{g}^-(w_\ast+\mu_0)- K\delta^2 -\epsilon = \delta s_\ast - K\delta^2 -\epsilon.
\end{align*}
Here the last inequality follows by the anti-concentration of $w_\ast+\mu_0$ over $I_\delta^-$ (not including $0$). Choosing similarly $\delta\equiv \epsilon^{1/2}$ to conclude the lower bound. 
\end{proof}

\subsection{Proof of Theorem \ref{thm:debiased_lasso}}
Recall that for two probability measures $\mu,\nu$ on $\R^n$ with finite second moments, their (squared) Wasserstein-2 distance is defined as $
\mathsf{W}_2^2(\mu,\nu)\equiv \inf_{\gamma} \int \pnorm{x-y}{}^2\,\gamma(\d x,\d y)$, where the infimum is taken over all couplings $\gamma$ with marginal distributions $\mu$ and $\nu$. 

We shall first establish distributional characterizations of $\hat{w}_A,\hat{v}_A$ in the Wasserstein-2 distance, that allow couplings to relate the joint distribution of $(\hat{w}_A,\hat{v}_A)$. To this end, let
\begin{align*}
\Pi_W&\equiv \hbox{ law of } \bigg(\eta_1\bigg(\Pi_{\mu_0}+\gamma_\ast Z; \frac{\gamma_\ast \lambda}{\beta_\ast}\bigg)-\Pi_{\mu_0}, \Pi_{\mu_0} \bigg),\\
\Pi_V&\equiv \hbox{ law of } \bigg(-\frac{\beta_\ast}{\gamma_\ast \lambda}\bigg[\eta_1\bigg(\Pi_{\mu_0}+\gamma_\ast Z; \frac{\gamma_\ast \lambda}{\beta_\ast}\bigg)-(\Pi_{\mu_0}+\gamma_\ast Z)\bigg], \Pi_{\mu_0} \bigg).
\end{align*}
\begin{proposition}\label{prop:lasso_w2}
Assume the same conditions as in Theorem \ref{thm:lasso_universality_generic}. Then there exists some $K=K(\sigma,\lambda,\tau,M_2)>0$ such that for any $\epsilon>0$,
\begin{align*}
\Prob\Big(\mathsf{W}_2^2\big(\Pi_{(\hat{w}_A^{\lasso},\mu_0)}, \Pi_W\big)\geq \epsilon\Big)\vee \Prob\Big(\mathsf{W}_2^2\big(\Pi_{(\hat{v}_A^{\lasso},\mu_0)}, \Pi_V\big)\geq \epsilon\Big)\leq K(1\vee \epsilon^{-3}) n^{-1/6}\log^3 n.
\end{align*}
\end{proposition}
\begin{proof}
First consider the claim for the Lasso error $\hat{w}$. Let $D_\epsilon\equiv \{w \in \R^n: \mathsf{W}_2^2\big(\Pi_{(w,\mu_0)}, \Pi_W\big)\geq \epsilon\} $. Recall $w_\ast$ defined in (\ref{def:w_lasso}). By \cite[Proposition F.2]{miolane2021distribution}, on an event $E_1$ with probability at least $1-Cn^{-100}$, we have $\mathsf{W}_2^2\big(\Pi_W, \Pi_{(w_\ast,\mu_0)}\big)\leq K n^{-1/3}\log n$. So on the event $E_1$, for all $w \in D_\epsilon$,
\begin{align*}
&n^{-1}\pnorm{w-w_\ast}{}^2\geq \mathsf{W}_2^2\big(\Pi_{(w,\mu_0)},\Pi_{(w_\ast,\mu_0)}\big)\\
&\geq \Big(\mathsf{W}_2\big(\Pi_{(w,\mu_0)},\Pi_W\big)-\mathsf{W}_2\big(\Pi_W,\Pi_{(w_\ast,\mu_0)})\Big)_+^2\geq \big(\epsilon^{1/2}- \{K n^{-1/3} \log n\}^{1/2}\big)_+^2.
\end{align*}
Consequently, for $\epsilon\geq K n^{-1/3}\log n$, on the event $E_1$,
\begin{align*}
n^{-1}\pnorm{w-w_{\ast}}{}^2\geq \epsilon/2,\quad \forall w \in D_\epsilon.
\end{align*}
Repeating the arguments in the proof of Proposition \ref{prop:lasso_est_gap} below (\ref{ineq:lasso_w_gap_1}), the claim of Proposition \ref{prop:lasso_est_gap} is valid with $D_\epsilon$ under the additional constraint $\epsilon\geq K  n^{-1/3}\log n$. Now we may proceed as the proof of Theorem \ref{thm:lasso_dist} for $\hat{w}$ to conclude (the additional constraint on $\epsilon$ can be dropped for free).

The claim for the Lasso subgradient $\hat{v}$ can be proved similarly, by considering the exceptional set  $D_\epsilon'\equiv \{v \in \R^n: \mathsf{W}_2^2\big(\Pi_{(v,\mu_0)}, \Pi_V\big)\geq \epsilon\} $. Then a similar argument as above shows that, for $\epsilon\geq K n^{-1/3}\log n$, on an event $E_2$ with $\Prob(E_2)\geq 1-C n^{-100}$,
\begin{align*}
n^{-1}\pnorm{v-v_{\ast}}{}^2\geq \epsilon/2,\quad \forall v \in D_\epsilon'.
\end{align*}
The same proof of Theorem \ref{thm:lasso_dist} for $\hat{v}$ applies to conclude the claim.
\end{proof}

\begin{proof}[Proof of Theorem \ref{thm:debiased_lasso}]
\noindent (1). The proof of the first inequality follows that of \cite[Theorem 3.3]{miolane2021distribution} in Appendix F.5 therein with some modifications. We provide some details below. We will work on the (high probability) event $\mathcal{E}$ defined by
\begin{align*}
\mathcal{E}&\equiv \bigg\{\mathsf{W}_2^2\big(\Pi_{(\hat{w},\mu_0)}, \Pi_W\big)\vee \mathsf{W}_2^2\big(\Pi_{(\hat{v},\mu_0)}, \Pi_V\big)\leq \epsilon^6\bigg\} \cap \bigg\{\abs{\hat{s}-s_\ast}\vee \biggabs{\frac{1}{n}\sum_{j=1}^n\bm{1}_{\abs{\hat{v}_j}=1}-s_\ast }\leq \epsilon^2\bigg\}. 
\end{align*}
By Proposition \ref{prop:lasso_w2} and (the proof of) Theorem \ref{thm:lasso_dist} for $\hat{s}$, we have $\Prob(\mathcal{E})\geq 1- K(1\vee \epsilon^{-24}) \bar{s}_n$. On the event $\mathcal{E}$, using the definition of $\mathsf{W}_2$, we may find couplings $(\Pi_{\mu_0}^w,Z^w), (\Pi_{\mu_0}^v,Z^v)$ both distributed as $ (n^{-1}\sum_{j=1}^n\delta_{\mu_{0,j}})\otimes \mathcal{N}(0,1)$, such that with
\begin{align*}
(\Pi_{\mu_0},\Pi_{\hat{\mu}},\Pi_{\hat{v}},\Pi_{\hat{\mu}^{\delasso}})=\frac{1}{n}\sum_{j=1}^n \delta_{(\mu_{0,j}, \hat{\mu}_j, \hat{v}_j,\hat{\mu}_j^{\delasso})},\quad \alpha_\ast\equiv \frac{\gamma_\ast\lambda}{\beta_\ast},
\end{align*}
we have
\begin{align*}
&\E^\circ \Big[\Big(\Pi_{\hat{\mu}}-\eta_1\big(\Pi_{\mu_0}^w+\gamma_\ast Z^w;\alpha_\ast\big)\Big)^2+\Big(\Pi_{\mu_0}-\Pi_{\mu_0}^w\Big)^2\Big]\leq 2\epsilon^6,\\
&\E^\circ \Big[\Big(\Pi_{\hat{v}}+\alpha_\ast^{-1}\Big\{ \eta_1\big(\Pi_{\mu_0}^v+\gamma_\ast Z^v;\alpha_\ast\big)-\Pi_{\mu_0}^v-\gamma_\ast Z^v\Big\}  \Big)^2+\Big(\Pi_{\mu_0}-\Pi_{\mu_0}^v\Big)^2\Big]\leq 2\epsilon^6.
\end{align*}
Here $\E^\circ[\cdot] = \E[\cdot|A,\xi]$ is taken over $(\Pi_{\mu_0}^w,Z^w), (\Pi_{\mu_0}^v,Z^v)$ and $(\Pi_{\mu_0},\Pi_{\hat{\mu}},\Pi_{\hat{v}},\Pi_{\hat{\mu}^{\delasso}})$. Proceeding as in the proof of \cite[Theorem 3.3]{miolane2021distribution} in Appendix F.5, up to Lemma F.8 therein, on an event $E$ with $\Prob^\circ (E)\geq 1- C\epsilon^2$, we may relate the two couplings $(\Pi_{\mu_0}^w,Z^w), (\Pi_{\mu_0}^v,Z^v)$ by the following relation
\begin{align}\label{ineq:lasso_debias_1}
\abs{\Pi_{\mu_0}^v+\gamma_\ast Z^v}\geq \alpha_\ast \Leftrightarrow \abs{\Pi_{\mu_0}^w+\gamma_\ast Z^w}\geq \alpha_\ast.
\end{align}
Now define 
\begin{align*}
X^{\delasso}\equiv \eta_1\big(\Pi_{\mu_0}^w+\gamma_\ast Z^w;\alpha_\ast\big)+\big\{\Pi_{\mu_0}^v+\gamma_\ast Z^v- \eta_1\big(\Pi_{\mu_0}^v+\gamma_\ast Z^v;\alpha_\ast\big)\big\}.
\end{align*}
Then using $\Pi_{\hat{\mu}_A^{\delasso}} = \Pi_{\hat{\mu}}+\mathsf{c}^{\delasso}\Pi_{\hat{v}}$, where $\mathsf{c}^{\delasso}\equiv\lambda(1-\pnorm{\hat{\mu}}{0}/m)^{-1}=\alpha_\ast+\bigo(\epsilon^2)$ on $\mathcal{E}$,
\begin{align}\label{ineq:lasso_debias_2}
&\bigabs{ \E^\circ\mathsf{g}\big(\Pi_{\hat{\mu}^{\delasso}},\Pi_{\mu_0}\big)-\E^\circ \mathsf{g}\big(X^{\delasso},\Pi_{\mu_0}\big) }\leq \pnorm{\mathsf{g}}{\lip} \cdot \E^\circ \abs{\Pi_{\hat{\mu}^{\delasso}}-X^{\delasso} }\nonumber\\
&\lesssim \pnorm{\mathsf{g}}{\lip}\cdot \Big[\E^{\circ}\Big(\Pi_{\hat{\mu}}-\eta_1\big(\Pi_{\mu_0}^w+\gamma_\ast Z^w;\alpha_\ast\big)\Big)^2\nonumber\\
&\qquad\qquad +\E^\circ \Big(\mathsf{c}^{\delasso}\Pi_{\hat{v}}-\big\{\Pi_{\mu_0}^v+\gamma_\ast Z^v- \eta_1\big(\Pi_{\mu_0}^v+\gamma_\ast Z^v;\alpha_\ast\big)\big\}\Big)^2\Big]^{1/2}\nonumber\\
& \leq K\cdot \pnorm{\mathsf{g}}{\lip}\cdot \epsilon^2. 
\end{align}
On the other hand, by (\ref{ineq:lasso_debias_1}), on the event $E$, we have 
\begin{align*}
X^{\delasso} &= (\Pi_{\mu_0}^w+\gamma_\ast Z^w)\bm{1}_{\abs{\Pi_{\mu_0}^w+\gamma_\ast Z^w}\geq \alpha_\ast}+ (\Pi_{\mu_0}^v+\gamma_\ast Z^v)\bm{1}_{\abs{\Pi_{\mu_0}^w+\gamma_\ast Z^w}< \alpha_\ast}\\
& = (\Pi_{\mu_0}^w+\gamma_\ast Z^w)\bm{1}_{\abs{\Pi_{\mu_0}^w+\gamma_\ast Z^w}\geq \alpha_\ast}+ (\Pi_{\mu_0}^v+\gamma_\ast Z^v)\bm{1}_{\abs{\Pi_{\mu_0}^v+\gamma_\ast Z^v}< \alpha_\ast}.
\end{align*}
Further using that $(\Pi_{\mu_0}^w+\gamma_\ast Z^w,\Pi_{\mu_0}^w)\equald (\Pi_{\mu_0}^v+\gamma_\ast Z^v,\Pi_{\mu_0}^v)$ under $\E^\circ$, we have
\begin{align*}
&\E \mathsf{g}\big(\Pi_{\mu_0}+\gamma_\ast Z,\Pi_{\mu_0}\big) \\
& = \E^\circ \mathsf{g}\big(\Pi_{\mu_0}^w+\gamma_\ast Z^w,\Pi_{\mu_0}^w\big) \bm{1}_{\abs{\Pi_{\mu_0}^w+\gamma_\ast Z^w}\geq \alpha_\ast}+\E^\circ \mathsf{g}\big(\Pi_{\mu_0}^w+\gamma_\ast Z^w,\Pi_{\mu_0}^w\big) \bm{1}_{\abs{\Pi_{\mu_0}^w+\gamma_\ast Z^w}< \alpha_\ast}\\
& = \E^\circ \mathsf{g}\big(\Pi_{\mu_0}^w+\gamma_\ast Z^w,\Pi_{\mu_0}^w\big) \bm{1}_{\abs{\Pi_{\mu_0}^w+\gamma_\ast Z^w}\geq \alpha_\ast}+\E^\circ \mathsf{g}\big(\Pi_{\mu_0}^v+\gamma_\ast Z^v,\Pi_{\mu_0}^v\big) \bm{1}_{\abs{\Pi_{\mu_0}^v+\gamma_\ast Z^v}< \alpha_\ast}.
\end{align*}
The above two displays imply that
\begin{align}\label{ineq:lasso_debias_3}
&\bigabs{\E^\circ \mathsf{g}\big(X^{\delasso},\Pi_{\mu_0}\big)-\E \mathsf{g}\big(\Pi_{\mu_0}+\gamma_\ast Z,\Pi_{\mu_0}\big)} \nonumber\\
& \leq \bigabs{\E^\circ \mathsf{g}\big(X^{\delasso},\Pi_{\mu_0}\big)\bm{1}_E-\E \mathsf{g}\big(\Pi_{\mu_0}+\gamma_\ast Z,\Pi_{\mu_0}\big)}+\pnorm{\mathsf{g}}{\infty}\Prob^\circ(E^c) \nonumber\\
&\leq \bigabs{ \E^\circ \big[\mathsf{g}\big(\Pi_{\mu_0}^w+\gamma_\ast Z^w,\Pi_{\mu_0}\big)-\mathsf{g}\big(\Pi_{\mu_0}^w+\gamma_\ast Z^w,\Pi_{\mu_0}^w\big) \big]\bm{1}_{\abs{\Pi_{\mu_0}^w+\gamma_\ast Z^w}\geq \alpha_\ast}\bm{1}_E }\nonumber\\
&\qquad + \bigabs{ \E^\circ \big[\mathsf{g}\big(\Pi_{\mu_0}^v+\gamma_\ast Z^v,\Pi_{\mu_0}\big)-\mathsf{g}\big(\Pi_{\mu_0}^v+\gamma_\ast Z^v,\Pi_{\mu_0}^v\big) \big]\bm{1}_{\abs{\Pi_{\mu_0}^v+\gamma_\ast Z^v}< \alpha_\ast}\bm{1}_E }\nonumber\\
&\qquad +3\pnorm{\mathsf{g}}{\infty}\Prob^\circ(E^c)\nonumber\\
&\lesssim  \pnorm{\mathsf{g}}{\lip}\max_{u \in \{w,v\}}\big\{\E^\circ \big(\Pi_{\mu_0}-\Pi_{\mu_0}^u\big)^2\big\}^{1/2} + \pnorm{\mathsf{g}}{\infty}\Prob^\circ(E^c)\nonumber\\
&\leq K\cdot ( \pnorm{\mathsf{g}}{\lip}\vee\pnorm{\mathsf{g}}{\infty})\cdot\epsilon^2.
\end{align}
Combining (\ref{ineq:lasso_debias_2})-(\ref{ineq:lasso_debias_3}) concludes the desired inequality.

\noindent (2). By Lemma \ref{lem:lasso_gamma_est} below, on an event $E_1$ with $\Prob(E_1)\geq 1-K\epsilon^{-12} \bar{s}_n$,  $\abs{\hat{\gamma}-\gamma_\ast}\leq \epsilon$. Let $\mathsf{g}^{\pm}(x,y)\equiv \mathsf{g}_0^\pm(x-y)\equiv \bm{1}\big(\abs{x-y}\leq z_{\alpha/2}(\gamma_\ast\pm \epsilon)\big)$. For $\delta \in (0,1)$, let $I_\delta^+\equiv (-z_{\alpha/2}(\gamma_\ast+ \epsilon)-\delta,-z_{\alpha/2}(\gamma_\ast+ \epsilon))\cup (z_{\alpha/2}(\gamma_\ast+ \epsilon),z_{\alpha/2}(\gamma_\ast+ \epsilon)+\delta)$, and $\mathsf{g}_{0,\delta}^+\equiv \mathsf{g}_0^+$ on $\R\setminus I_\delta^+$ and linearly interpolated otherwise. Let $\mathsf{g}_\delta^+(x,y)\equiv \mathsf{g}_{0,\delta}^+(x-y)$. It is easy to see that $\mathsf{g}_\delta^+$ is $\sqrt{2}/\delta$-Lipschitz. Then on the event $E_1$, for any $\delta \in (0,1)$,
\begin{align*}
\mathscr{C}^{\delasso} = \frac{1}{n}\sum_{j=1}^n \bm{1}\big(\abs{\hat{\mu}^{\delasso}_j-\hat{\mu}_{0,j}}\leq z_{\alpha/2} \hat{\gamma}\big)\leq \E^\circ \mathsf{g}^+\big(\Pi_{\hat{\mu}_A^{\delasso}},\Pi_{\mu_0}\big)\leq \E^\circ \mathsf{g}_\delta^+\big(\Pi_{\hat{\mu}_A^{\delasso}},\Pi_{\mu_0}\big).
\end{align*}
By the proven claim in (1), on an event $E_2$ with $\Prob(E_2)\geq 1-K\epsilon^{-12}\bar{s}_n$, 
\begin{align*}
\E^\circ \mathsf{g}_\delta^+\big(\Pi_{\hat{\mu}_A^{\delasso}},\Pi_{\mu_0}\big)&\leq \E \mathsf{g}_\delta^+\big(\Pi_{\mu_0}+\gamma_\ast Z,\Pi_{\mu_0}\big)+\sqrt{2}\epsilon/\delta\\
&\leq \E \mathsf{g}^+\big(\Pi_{\mu_0}+\gamma_\ast Z,\Pi_{\mu_0}\big)+\Prob\big(\gamma_\ast Z \in I_\delta^+\big)+\sqrt{2}\epsilon/\delta\\
& \leq \Prob\big(\abs{\gamma_\ast Z}\leq z_{\alpha/2}(\gamma_\ast+\epsilon)\big)+ K\delta+\sqrt{2}\epsilon/\delta\\
&\leq (1-\alpha)+ \bigo\big(\epsilon+\delta+\epsilon/\delta\big).
\end{align*}
Now choosing $\delta\equiv \epsilon^{1/2}$, we have shown that on $E_1\cap E_2$, $\mathscr{C}^{\delasso} \leq (1-\alpha)+\bigo(\epsilon^{1/2})$. A similar lower bound for $\mathscr{C}^{\delasso}$ can be proven analogously by smoothing $\mathsf{g}^-$. 
\end{proof}

\begin{lemma}\label{lem:lasso_gamma_est}
Assume the same conditions as in Theorem \ref{thm:lasso_universality_generic}. Then there exists some $K=K(\sigma,\lambda,\tau,M_2)>0$ such that for $\epsilon \in (0,1)$,
\begin{align*}
\Prob\big(\abs{\hat{\gamma}-\gamma_\ast}\geq \epsilon\big)\leq K \epsilon^{-12} n^{-1/6}\log^3 n.
\end{align*}
\end{lemma}
\begin{proof}
By Theorem \ref{thm:lasso_dist}, on an event $E_1$ with $\Prob(E_1)\geq 1- K \epsilon^{-12}\bar{s}_n$,
\begin{align*}
\biggabs{ \frac{ \pnorm{\hat{r}}{} }{\sqrt{m}}- \frac{\beta_\ast}{\gamma_\ast \sqrt{m}} \E_h \bigpnorm{\sigma \xi_0+\sqrt{\gamma_\ast^2-\sigma^2} h}{}  } \vee \bigabs{\hat{s}-s_\ast}\leq \epsilon. 
\end{align*}
Note that with $F(h)\equiv \pnorm{\sigma \xi_0+\sqrt{\gamma_\ast^2-\sigma^2} h}{} $, $\pnorm{\nabla F(h)}{}\leq \sqrt{\gamma_\ast^2-\sigma^2}$ and therefore Gaussian-Poincar\'e inequality yields that
\begin{align*}
0&\leq \E_h \bigpnorm{\sigma \xi_0+\sqrt{\gamma_\ast^2-\sigma^2} h}{}^2-\Big(\E_h \bigpnorm{\sigma \xi_0+\sqrt{\gamma_\ast^2-\sigma^2} h}{}\Big)^2\\
&=\var\big(F(h)\big)\leq \E\pnorm{\nabla F(h)}{}^2\leq \gamma_\ast^2-\sigma^2. 
\end{align*}
On the event $E_2\equiv \{\abs{\pnorm{\xi_0}{}^2/m-1}\leq Kr_n\}$ so that $\Prob(E_2)\geq 1- Cn^{-100}$,  we have 
\begin{align*}
\E_h \bigpnorm{\sigma \xi_0+\sqrt{\gamma_\ast^2-\sigma^2} h}{}^2 = m\Big[\sigma^2 (\pnorm{\xi_0}{}^2/m)+ (\gamma_\ast^2-\sigma^2)\Big] = m \big(\gamma_\ast^2+\bigo(r_n)\big).
\end{align*}
So on $E_1\cap E_2$, we have
\begin{align*}
\biggabs{ \frac{ \pnorm{\hat{r}}{} }{\sqrt{m}}- \beta_\ast}\leq \epsilon+\mathcal{O}(r_n). 
\end{align*}
By the fixed point equation (\ref{eqn:lasso_fpe}) and the definition of $s_\ast$ in (\ref{def:s_lasso}), we have $s_\ast = (m/n)\big(1-\{\beta_\ast/\gamma_\ast\}\big)$. Consequently by definition of $\hat{\gamma}$ in (\ref{def:gamma_est_lasso}), on the event $E_1\cap E_2$,
\begin{align*}
\hat{\gamma} = \frac{\beta_\ast+\bigo(\epsilon \vee r_n)}{\beta_\ast/\gamma_\ast+\bigo(\epsilon)} = \gamma_\ast+\bigo(\epsilon\vee r_n). 
\end{align*}
The claim follows as we do not need to consider the regime $\epsilon\leq r_n$. 
\end{proof}

\section{Proofs for Section \ref{section:examples}: Robust regression}\label{section:proof_robust}

\noindent \emph{Convention}: We shall write
\begin{align*}
\bar{H}^{\rob}(w) \equiv \bar{H}^{\rob}(w,A)\equiv \bar{H}^{\rob}(w,A,\xi),
\end{align*}
and will usually omit the superscript $(\cdot)^{\rob}$ if no confusion could arise. We also usually omit the subscript $A$ that indicates the design matrix, but we will use the subscript $G$ for Gaussian designs when needed. All notation will be local in this section. 

\subsection{Proof of Proposition \ref{prop:robust_risk}}

	The proof idea is similar to that of Proposition \ref{prop:ridge_risk}-(2).  Fix any $s\in[n]$. Define the (column) leave-one-out version by $\hat{w}^{(s)} \equiv \argmin_{w\in\R^n: w_s = 0} H(w)$. By the cost optimality of $\hat{w}$,
	\begin{align*}
	0 &\leq \bar{H}(\hat{w}^{(s)}) - \bar{H}(\hat{w}) \nonumber \\
	&= -\sum_{i=1}^m\Big(\psi_0(a_i^\top \hat{w}-\xi_i) - \psi_0(a_i^\top \hat{w}^{(s)}-\xi_i)\Big) + \frac{\lambda}{2}\big(\pnorm{\hat{w}^{(s)} + \mu_0}{}^2 - \pnorm{\hat{w} + \mu_0}{}^2\big)\nonumber\\
	&\stackrel{(*)}{\leq} -\sum_{i=1}^m \psi_0'\big(a_i^\top\hat{w}^{(s)}-\xi_i\big)\cdot \big[a_{i,s}\hat{w}_s - a_{i,-s}^\top(\hat{w}^{(s)}_{-s} - \hat{w}_{-s})\big] + \frac{\lambda}{2}\big(\pnorm{\hat{w}^{(s)} + \mu_0}{}^2 - \pnorm{\hat{w} + \mu_0}{}^2\big)\nonumber\\
	&\stackrel{(**)}{=} -\hat{w}_s \cdot \sum_{i=1}^m \psi_0'\big(a_i^\top\hat{w}^{(s)}-\xi_i\big)\cdot a_{i,s} + \frac{\lambda}{2}\big((\mu_{0,s})^2 - (\hat{w}_{s} + \mu_{0,s})^2\big)\nonumber\\
	&\qquad + \frac{\lambda}{2}\Big(\pnorm{\hat{w}^{(s)}_{-s} + \mu_{0,-s}}{}^2 - \pnorm{\hat{w}_{-s} + \mu_{0,-s}}{}^2 - 2(\hat{w}^{(s)}_{-s} + \mu_{0,-s})(\hat{w}^{(s)}_{-s} - \hat{w}_{-s})\Big)\nonumber\\
	&= -\hat{w}_s \cdot \sum_{i=1}^m \psi_0'\big(a_{i,-s}^\top\hat{w}^{(s)}_{-s}-\xi_i\big)\cdot a_{i,s} - \frac{\lambda}{2}(\hat{w}_s)^2 - \lambda\hat{w}_{s}\mu_{0,s} - \frac{\lambda}{2}\pnorm{\hat{w}^{(s)}_{-s} - \hat{w}_{-s}}{}^2.
	\end{align*}
	Here $(*)$ follows from the convexity of $\psi_0(\cdot)$, and $(**)$ follows from the KKT condition for $\hat{w}^{(s)}$ that reads $
	\sum_{i=1}^m a_{i,-s}\psi_0'(a_{i,-s}^\top\hat{w}^{(s)}_{-s}-\xi_i) + \lambda(\hat{w}^{(s)}_{-s} + \mu_{0,-s}) = 0$. Rearranging terms yields that
	\begin{align*}
	\abs{\hat{w}_s} \lesssim \lambda^{-1}\cdot\biggabs{\sum_{i=1}^m a_{i,s}\cdot\psi_0'(a_{i,-s}^\top \hat{w}^{(s)}_{-s}-\xi_i)} + \pnorm{\mu_0}{\infty}.
	\end{align*}
	So for $p\geq 2$, 
	\begin{align*}
	\E \abs{\hat{w}_s}^p&\lesssim_p \lambda^{-p} \E \Big|\sum_{i=1}^m a_{i,s}\cdot\psi_0'(a_{i,-s}^\top \hat{w}^{(s)}_{-s}-\xi_i)\Big|^p + \pnorm{\mu_0}{\infty}^p\nonumber\\
	&\stackrel{(a)}{\leq} (L_0/\lambda)^p \E \biggabs{\sum_{i=1}^m a_{i,s}}^p+ \pnorm{\mu_0}{\infty}^p \stackrel{(b)}{\lesssim_p} (L_0/\lambda)^p \E \biggabs{\sum_{i=1}^m \epsilon_i a_{i,s}}^p+ \pnorm{\mu_0}{\infty}^p\nonumber \\
	&\stackrel{(c)}{\lesssim_p } (L_0/\lambda)^p \E \biggabs{\sum_{i=1}^m a_{i,s}^2}^{p/2}+ \pnorm{\mu_0}{\infty}^p  \lesssim_p (L_0/\lambda)^p M_{p;A}+ \pnorm{\mu_0}{\infty}^p.
	\end{align*}
	Here $\epsilon_i$'s are i.i.d. Rademachers that are also independent of other random variables, and $(a)$ follows from the contraction principle (cf. \cite[Corollary 3.1.18]{gine2015mathematical}), $(b)$ follows from the symmetrization inequality (cf. \cite[Theorem 3.1.21]{gine2015mathematical}) and $(c)$ follows from Khintchine's inequality (cf. \cite[Proposition 3.2.8]{gine2015mathematical}).\qed

\subsection{Proof of Theorem \ref{thm:robust_universality_generic}}

	Note that by Proposition \ref{prop:robust_risk}, with $p\equiv 6+\delta$, 
	\begin{align*}
	\Prob\big(\pnorm{\hat{w}_A}{\infty}>L_n\big)\leq L_n^{-p}\E \max_{j \in [n]} \abs{\hat{w}_{A,s}}^p \lesssim_{p,L_0,\lambda} (nL_n^{-p})\big(M_{p;A}+ \pnorm{\mu_0}{\infty}^p\big).
	\end{align*}
	It is easy to verify (\ref{cond:f_moduli}) for the ridge penalty, so we may apply Theorem \ref{thm:universality_reg} to obtain that for some $K=K(p,M_{6+\delta;A},\tau,\lambda)>0$, 
	\begin{align*}
	\Prob\big(\hat{w}_A \in \mathcal{S}_n\big)\leq 4\epsilon_n + K\bigg\{(nL_n^{-(6+\delta)})\big(1\vee \pnorm{\mu_0}{\infty}^{6+\delta}\big)+(1\vee \rho_0^{-3}) \bigg(\frac{ L_n \log^{2/3} n}{n^{1/6}}\bigg)^{1/7}\bigg\}.
	\end{align*}
	Now setting $L_n\equiv n^{1/6-\epsilon}$ with $\epsilon \equiv \epsilon_\delta \equiv \delta/(6\cdot (6+1/7+\delta))\geq \delta/43$ for $\delta \in (0,1)$, and the right hand side of the above display is bounded by 
	\begin{align*}
	4\epsilon_n + K'\big(1+\pnorm{\mu_0}{\infty}^{6+\delta}+\rho_0^{-3}\big) \cdot n^{-(1\wedge\delta)/500},
	\end{align*}
	where $K'>0$ further depends on $\delta \in (0,1)$. This completes the proof. \qed

\subsection{Proof of Theorem \ref{thm:robust_risk_asymp}}

	We will use the results of \cite{thrampoulidis2018precise} to derive the result. Assume for notational simplicity that $m/n=\tau_0$ (instead of equal in limit). Now consider a reparametrized regression model $Y=G_n \nu_0+\xi$, where $G_n=\sqrt{\tau_0} G$ whose entries are $\mathcal{N}(0,1/n)$ and $\nu_0\equiv \mu_0/\sqrt{\tau_0}$. Let $\lambda_\nu \equiv \tau_0 \lambda$, and 
	\begin{align*}
	\hat{\nu}\equiv\argmin_{\nu \in \R^n} \bigg\{\sum_{i=1}^m \psi_0\big(y_i-(G_n \nu)_i\big)+\frac{\lambda_\nu}{2}\pnorm{\nu}{}^2\bigg\}.
	\end{align*}
	Then $\hat{\nu}=\hat{\mu}/\sqrt{\tau_0}$. The purpose of this reparametrization is to match the setup of \cite{thrampoulidis2018precise} exactly. For this reparametrized model, conditions (1)-(5) ensure that both Theorem \ref{thm:robust_universality_generic} and \cite[Theorem 4.1]{thrampoulidis2018precise} may be applied. In particular, regularities conditions on $\psi_0$, $\xi_0$ and moment conditions on $A_0,\mu_0$ in \cite[Theorem 4.1]{thrampoulidis2018precise} are guaranteed by (1)-(5); (stochastic) boundedness of the normalized estimation error in \cite[Theorem 4.1]{thrampoulidis2018precise} is guaranteed by Proposition \ref{prop:robust_risk}; the moment condition in (5) ensures that $\E \pnorm{\nu_0}{\infty}^{6+\delta}/n^\epsilon \to 0$ for any $\epsilon>0$. Now \cite[Eqn. (23)]{thrampoulidis2018precise} applies so the system of equations 
	\begin{align}\label{ineq:robust_risk_asymp_1}
	{\gamma}_\ast^2& =\tau_0 {\beta}_\ast^2\cdot \E \mathsf{e}_{\psi_0}'\big({\gamma}_\ast Z+\xi_1;{\beta}_\ast\big)^2+{\lambda}_\nu^2 {\beta}_\ast^2 \big(\E \Pi_0^2/\tau_0),\nonumber\\
	{\gamma}_\ast(1-\lambda_\nu {\beta}_\ast)& = \tau_0{\beta}_\ast\cdot \E \Big(\mathsf{e}_{\psi_0}'\big({\gamma}_\ast Z+\xi_1;{\beta}_\ast\big)\cdot Z\Big)
	\end{align}
	admits a unique non-trivial solution $({\beta}_\ast,{\gamma}_\ast)\in (0,\infty)^2$. As $\mathsf{e}_{\psi_0}'(x;\tau)=\tau^{-1}(x-\prox_{\psi_0}(x;\tau))$, the first equation in (\ref{ineq:robust_risk_asymp_1}) becomes
	\begin{align}\label{ineq:robust_risk_asymp_2}
	\gamma_\ast^2 /\tau_0 =  \E \Big(\gamma_\ast Z+\xi_1- \prox_{\psi_0}\big(\gamma_\ast Z+\xi_1;\beta_\ast\big)\Big)^2 +\lambda^2 {\beta}_\ast^2 \cdot \E \Pi_0^2. 
	\end{align}
	Furthermore, since $x\mapsto \prox_{\psi_0}(x;\tau)$ is $1$-Lipschitz (cf. Lemma \ref{lem:prox_lipschitz}), we may apply Stein's identity to the second equation in (\ref{ineq:robust_risk_asymp_1}), which reduces to 
	\begin{align}\label{ineq:robust_risk_asymp_3}
	& 1-\lambda_\nu {\beta}_\ast = \tau_0{\beta}_\ast\cdot \E \mathsf{e}_{\psi_0}''\big({\gamma}_\ast Z+\xi_1;{\beta}_\ast\big) = \tau_0\Big(1- \E \prox_{\psi_0}'\big({\gamma}_\ast Z+\xi_1;{\beta}_\ast\big)  \Big),\nonumber\\
	&\Leftrightarrow\quad  1-\tau_0^{-1}+ \lambda \beta_\ast=\E \prox_{\psi_0}'\big({\gamma}_\ast Z+\xi_1;{\beta}_\ast\big). 
	\end{align}
	On the other hand, \cite[Eqns. (85)-(86)]{thrampoulidis2018precise} along with \cite[Lemma A.5-(b)]{thrampoulidis2018precise} ensure that the limits in probability (denoted $\plim$) of the normalized estimation error satisfy $
	\plim {\pnorm{\hat{\nu}-\nu_0}{}^2}/{n}=\plim {\pnorm{\hat{\mu}-\mu_0}{}^2}/(\tau_0 n)={\gamma}_\ast^2$, 
	in the stronger sense that there exist $z, \rho_0>0$ such that for any small enough $\epsilon>0$, with $\mathcal{S}_n(\epsilon)\equiv \big\{w \in \R^n: \abs{ \pnorm{w}{}^2/n-\tau_0{\gamma}_\ast^2}>\epsilon\big\}$, 
	\begin{align*}
	\Prob\bigg(\min_{w \in \R^n} H(w,G)\geq z+\rho_0 \bigg)\vee 	\Prob\bigg(\min_{w \in \mathcal{S}_n(\epsilon)} H(w,G)\leq z+2\rho_0 \bigg)\to 0.
	\end{align*}
	The claim now follows by utilizing (\ref{ineq:robust_risk_asymp_2})-(\ref{ineq:robust_risk_asymp_3}). \qed

\appendix

\section{Technical tools}

The following Lindeberg principle is essentially taken from \cite{chatterjee2006generalization}.

\begin{theorem}\label{thm:lindeberg}
Let $X=(X_1,\ldots,X_n)$ and $Y=(Y_1,\ldots,Y_n)$ be two random vectors in $\R^n$ with independent component and matching first and second moments: $\E X_i^\ell = \E Y_i^\ell$ holds for all $i \in [n]$ and $\ell = 1,2$. Then for any $f \in C^3(\R^n)$,
\begin{align*}
\bigabs{\E f(X) - \E f(Y)}\leq \sum_{i=1}^n \max_{U_i \in \{X_i,Y_i\}}\biggabs{\E \int_0^{U_i} \partial_i^3 f(X_{[1:(i-1)]},t, Y_{[(i+1):n]} )(U_i-t)^2\,\d{t}}.
\end{align*}
\end{theorem}
\begin{proof}
The proof is essentially a repetition of \cite[Theorem 1.1]{chatterjee2006generalization} by using the integral remainder in the Taylor expansion. 
\end{proof}

We need the following min-max theorem due to Sion \cite{sion1958general}.

\begin{theorem}[Sion's min-max theorem]\label{thm:sion_minmax}
	Let $X$ be a compact convex subset of a linear topological space and $Y$ a convex subset of a linear topological space. If $f$ is a real-valued function on $X\times Y$ satisfying:
	\begin{enumerate}
		\item $y\mapsto f(x,y)$ is upper-semicontinuous and quasi-concave for all $x \in X$;
		\item$x\mapsto f(x,y)$ is lower-semicontinuous and quasi-convex  for all $y \in Y$.
	\end{enumerate}
	Then $\min_{x \in X} \sup_{y \in Y} f(x,y)= \sup_{y \in Y} \min_{x \in X}f(x,y)$.
\end{theorem}

The following version of convex Gaussian min-max theorem, proved using Gordon's min-max theorem \cite{gordan1985some,gordon1988milman}, is taken from \cite[Theorem 6.1]{thrampoulidis2018precise} or \cite[Theorem 5.1]{miolane2021distribution}.

\begin{theorem}[Convex Gaussian Min-Max Theorem]\label{thm:CGMT}
	Suppose $D_u \in \R^m, D_v \in \R^n$ are compact sets, and $Q: D_u\times D_v \to \R$ is continuous. Let $G=(G_{ij})_{i \in [m],j\in[n]}$ with $G_{ij}$'s i.i.d. $\mathcal{N}(0,1)$, and $g \sim \mathcal{N}(0,I_m)$, $h \sim \mathcal{N}(0,I_n)$ be independent Gaussian vectors. Define 
	\begin{align*}
	\Phi^{\textrm{p}} (G) &= \min_{u \in D_u}\max_{v \in D_v} \Big( u^\top G v + Q(u,v)\Big), \nonumber\\
	\Phi^{\textrm{a}}(g,h) &= \min_{u \in D_u}\max_{v \in D_v} \Big(\pnorm{v}{} g^\top u + \pnorm{u}{} h^\top v+ Q(u,v)\Big).
	\end{align*}
	Then the following hold.
	\begin{enumerate}
		\item For all $t \in \R$, $
		\Prob\big(\Phi^{\textrm{p}} (G)\leq t\big)\leq 2 \Prob\big(\Phi^{\textrm{a}}(g,h)\leq t\big)$. 
		\item If $(u,v)\mapsto u^\top G v+ Q(u,v)$ satisfies the conditions of Sion's min-max theorem (cf. Theorem \ref{thm:sion_minmax}) the pair $(D_u,D_v)$ a.s. (for instance, $D_u,D_v$ are convex, and $Q$ is convex-concave), then $
		\Prob\big(\Phi^{\textrm{p}} (G)\geq t\big)\leq 2 \Prob\big(\Phi^{\textrm{a}}(g,h)\geq t\big)$. 
	\end{enumerate}
\end{theorem}

\section{Auxiliary results}

\begin{lemma}\label{lem:smooth_approx_l1}
Let $f:\R \to \R_{\geq 0}$ be absolutely continuous with $\esssup\,\abs{f'}\leq L$. For $\rho>0$, let $f_\rho(x)\equiv \E f(x+\rho Z)$ where $Z\sim \mathcal{N}(0,1)$. Then the following hold.
\begin{enumerate}
	\item $\pnorm{f_\rho-f}{\infty}\leq L\rho$. 
	\item $\pnorm{f_\rho^{(\ell)}}{\infty}\leq 2L\rho^{-\ell+1}$ for $\ell =1,2,3$.
\end{enumerate}
\end{lemma}
\begin{proof}
Some easy facts: $\E\abs{Z}\leq 1$ and $\E\abs{Z^2-1}\leq \var^{1/2}(Z^2)=2^{1/2}$.	
	
\noindent (1). As $f(\cdot)$ is $L$-Lipschitz, we have for any $x \in \R$, $\abs{f_\rho(x)-f(x)}\leq \E \abs{f(x+\rho Z)-f(x)}\leq \rho\cdot (L\E\abs{Z})$.

\noindent (2). First by absolute continuity of $f$, $f_\rho'(x) = \E f'(x+\rho Z)$, so $\pnorm{f_\rho'}{\infty}\leq L$. Now we give a different representation of $f_\rho'$. With $\varphi(\cdot)$ denoting the d.f. for $\mathcal{N}(0,1)$, we may write $
f_\rho(x)= \int f(x+\rho z) \varphi(z)\,\d{z}=\rho^{-1} \int f(y) \varphi\big((y-x)/\rho\big)\,\d{y}$, 
so using the identity $\varphi'(x)=-x \varphi(x)$, 
\begin{align*}
f_\rho'(x) &= -\frac{1}{\rho^2} \int f(y) \varphi'\bigg(\frac{y-x}{\rho}\bigg)\,\d{y}  = -\frac{1}{\rho} \int f(x+\rho z) \varphi'(z)\,\d{z}= \frac{1}{\rho}\E Z f(x+\rho Z).
\end{align*}
Using the absolute continuity of $f$ again, we have $f_\rho''(x)= \rho^{-1} \E Z f'(x+\rho Z)$. This implies that $\pnorm{f_\rho''}{\infty}\leq (L\E \abs{Z})/\rho$. For the third derivative, we proceed similarly by representing $f_\rho''$ as 
\begin{align*}
f_\rho''(x)& = \frac{1}{\rho^2} \int f(x+\rho z) \varphi''(z)\,\d{z} = \frac{1}{\rho^2} \E (Z^2-1) f(x+\rho Z).
\end{align*}
The last equality follows as $\varphi''(x)=(x^2-1)\varphi(x)$. This implies $
f_\rho^{(3)}(x) = \rho^{-2} \E(Z^2-1)f'(x+\rho Z)$, and therefore $\pnorm{f_\rho^{(3)}}{\infty}\leq (L \E\abs{Z^2-1})/\rho^2$. 
\end{proof}

\begin{lemma}[Sparse eigenvalues]\label{lem:sparse_eigenvalue}
	Suppose that the entries of $A_0 \in \R^{m\times n}$ are independent, mean-zero, variance $\sigma^2$ and uniformly sub-Gaussian. Let $\hat{\Sigma} = A_0^\top A_0/m$ be the sample covariance, and $
	\phi_{+}(k) \equiv \sup_{v\in\R^n: \pnorm{v}{0}\leq k}{v^\top\hat{\Sigma}v}/{\pnorm{v}{}^2}$,  $\phi_{-}(k) \equiv \inf_{v\in\R^n: \pnorm{v}{0}\leq k}{v^\top\hat{\Sigma}v}/{\pnorm{v}{}^2}$. 
	Then for $c \in (0,1/2)$, 
	\begin{align*}
	\bigg[\frac{\phi_{\pm}(k)}{\sigma^2}- \bigg((1 \pm c) \pm K\cdot\frac{k\log(en/k)}{m}\bigg)_+\bigg]_\pm=0
	\end{align*}
	holds simultaneously for $k\in[n]$ with probability at least $1 - \exp(-c'm)$. 
	Here $c'>0$ only depends on $c$, and $K>0$ is universal. Consequently, if $\tau \leq n/m\leq 1/\tau$ for some $\tau\in(0,1)$, then $\phi_{-}(k)\geq \sigma^2/2$ if $k \leq c_0 m$ for some $c_0 = c_0(K,\tau)$.
\end{lemma}
\begin{proof}
	We only consider  $A_0$ with independent mean-zero, unit-variance and uniformly sub-Gaussian entries. We will prove that the desired bound for $\phi_{+}(k)$ and $\phi_{-}(k)$ holds for each $k$ with probability $1 - \exp(-C_0 k\log(en/k) - c'm)$ for some universal $C_0>0$ and $c'>0$ depending on $c$ only. Then the claim follows by taking the union bound over $k\in[n]$. 
	
	To this end, note that $
	\sup_{v\in\R^n: \pnorm{v}{0}\leq k}{v^\top\hat{\Sigma}v}/{\pnorm{v}{}^2} \leq \sup_{S\subset[n]: |S| = k}\pnorm{\hat{\Sigma}_{S,S}}{\op}$. 
	For the latter quantity, using the sub-Gaussianity of the distribution of $A_0$ and a standard covering argument, we have for any $S\subset [n]$ with $\abs{S}=k$, and any $t > 0$,
	\begin{align*}
	\Prob\Big(\pnorm{\hat{\Sigma}_{S,S} - I_k}{\op} \geq t\Big) \leq \exp\Big(C_1k - C_2m(t^2 \wedge t)\Big),
	\end{align*}
	where $C_1,C_2>0$ are universal constants. The claim now follows by taking the union bound over $S\subset[n]$ and choosing $t = c' + K\cdot k\log(en/k)/m$ for some small enough $c'$ depending on $c$ and some universal $K$ depending on $C_1,C_2$. A similar argument holds for $\phi_{-}(k)$.
\end{proof}

\begin{lemma}\label{lem:prox_lipschitz}
Let $f$ be a proper, closed convex function defined on $\R$. Then the map $x\mapsto \prox_f(x;\tau)$ is $1$-Lipschitz for any $\tau>0$.
\end{lemma}
\begin{proof}
Take $x_1,x_2\in \R$ and let $z_i\equiv \prox_f(x_i;\tau)$ for $i=1,2$. By the first-order optimality condition, $(x_i-z_i)/\tau \in \partial f(z_i)$ for $i=1,2$. Using the monotonicity of subdifferential, we have $
\iprod{(x_1-z_1)/\tau-(x_2-z_2)/\tau}{z_1-z_2}\geq 0$, which is equivalent to $\pnorm{z_1-z_2}{}^2\leq \iprod{x_1-x_2}{z_1-z_2}$. Use Cauchy-Schwarz to conclude.
\end{proof}

\section*{Acknowledgments}
The authors are indebted to Cun-Hui Zhang for a number of stimulating discussions during various stages of this research.

\bibliographystyle{amsalpha}
\bibliography{mybib}

\end{document}